\documentclass{amsart}
\usepackage{amsmath,amsthm,indentfirst}
\usepackage{amssymb}
\usepackage{amsfonts,dsfont}
\usepackage{graphicx}
\usepackage{caption}
\usepackage{wrapfig}
\usepackage{a4}
\usepackage{amscd}
\usepackage[latin1]{inputenc}
\usepackage{url}
\theoremstyle{plain}
\newtheorem{thm}{Theorem}
\newtheorem{cor}[thm]{Corollary} 
\newtheorem{lem}[thm]{Lemma}
\newtheorem{prop}[thm]{Proposition}

\newtheorem{definition}[thm]{Definition}

\newtheorem{remark}[thm]{Remark}

\def\bbz{\mathbb{Z}}
\def\bbq{\mathbb{Q}}

\def\bbr{\mathbb{R}}
\def\bba{\mathbb{A}}

\def\bbc{\mathbb{C}}

\def\bbe{\mathbb{E}}
\def\bbh{\mathbb{H}}

\def\bbg{\mathbb{G}}
\def\bbt{\mathbb{T}}
\def\bbv{\mathbb{V}}

\def\bbp{\mathbb{P}}

\def\gcal{\mathcal{G}}

\def\ucal{\mathcal{U}}

\def\ocal{\mathcal{O}}
\def\ocal{\mathcal{O}}

\def\bcal{\mathcal{B}}
\def\ccal{\mathcal{C}}

\def\tcal{\mathcal{T}}

\def\cal{\mathcal{H}}

\def\pcal{\mathcal{P}}
\def\wcal{\mathcal{W}}
\def\hfr{\mathfrak{h}}
\def\afr{\mathfrak{a}}
\def\Bfr{\mathfrak{B}}
\def\Pfr{\mathfrak{P}}

\def\pfr{\mathfrak{p}}

\def\bfr{\mathfrak{b}}

\def\gl{\mathfrak{gl}}
\def\gfr{\mathfrak{g}}

\def\tfr{\mathfrak{t}}

\def\f{\mathfrak{f}}

\def\ybf{\mathbf{y}}
\def\xbf{\mathbf{x}}

\def\vbf{\mathbf{v}}
\def\ebf{\mathbf{e}}

\DeclareMathOperator\Spec{Spec}

\DeclareMathOperator\SL{SL}
\DeclareMathOperator\GL{GL}

\DeclareMathOperator\Ad{Ad}
\DeclareMathOperator\ad{ad}
\DeclareMathOperator\Lie{Lie}

\DeclareMathOperator\ind{ind}
\DeclareMathOperator\diag{diag}
\DeclareMathOperator\Cay{Cay}
\DeclareMathOperator\Tr{Tr}
\DeclareMathOperator\supp{supp}
\DeclareMathOperator\conj{Conj}

\DeclareMathOperator\gal{Gal}
\newcommand{\pl}[1]{{\rm Pl}(#1)}

\newcommand{\lin}[2]{\Psi_{#1}^{#2}}
\newcommand{\linValue}[3]{\lin{#1}{#2}(\pi_{#2}(#3))}
\newcommand{\wt}[1]{\widetilde{#1}}
\newcommand{\wh}[1]{\widehat{#1}}
\def\h{\hspace{1mm}}

\def\vare{\varepsilon}
\def\be{\begin{equation}}
\def\ee{\end{equation}}
\def\one{\mathds{1}}
\def\qbar{\overline{\bbq}}

\def\act{\curvearrowright}

\begin{document}
\title{Super-approximation, I: $\pfr$-adic semisimple case.}

\author{Alireza Salehi Golsefidy}
\address{Mathematics Dept, University of California, San Diego, CA 92093-0112}
\email{golsefidy@ucsd.edu}
\thanks{A. S-G. was partially supported by the NSF grants DMS-1160472, DMS-1303121 and A. P. Sloan Research Fellowship.  Parts of this work was done when I was visiting Isaac Newton Institute and the MSRI, and I would like to thank both of these institutes for their hospitality.
}
\date{\today}
\begin{abstract}
Let $k$ be a number field, $\Omega$ be a finite  symmetric subset of $\GL_{n_0}(k)$, and $\Gamma=\langle \Omega\rangle$. Let 
\[
\ccal(\Gamma):=\{\pfr\in V_f(k)|\h \Gamma \text{ is a bounded subgroup of } \GL_{n_0}(k_{\pfr})\},
\]
and $\Gamma_{\pfr}$ be the closure of $\Gamma$ in $\GL_{n_0}(k_{\pfr})$. Assuming that the Zariski-closure of $\Gamma$ is semisimple, we prove that the family of left translation actions $\{\Gamma\act \Gamma_{\pfr}\}_{\pfr\in \ccal(\Gamma)}$ has {\em uniform spectral gap}. 

As a corollary we get that the left translation action $\Gamma\act G$ has {\em local spectral gap} if $\Gamma$ is a countable dense subgroup of a semisimple $p$-adic analytic group $G$ and $\Ad(\Gamma)$ consists of matrices with algebraic entries in some $\bbq_p$-basis of $\Lie(G)$. This can be viewed as a (stronger) $p$-adic version of \cite[Theorem A]{BISG}, which enables us to give applications to the Banach-Ruziewicz problem and orbit equivalence rigidity.
\end{abstract}
\maketitle
\tableofcontents
\section{Introduction}
\subsection{Statement of the main result.}\label{ss:Statement}
For a symmetric probability measure\footnote{A measure $\mu$ is called symmetric if $i^{*}(\mu)=\mu$ where $i:G\rightarrow G, i(g):=g^{-1}$.} $\mu$ with finite support on a compact group $G$, let 
\[
T_{\mu}:L^2(G)\rightarrow L^2(G),\h\h T_{\mu}(f):=\mu \ast f,
\]
where $(\mu \ast f)(g):=\sum_{g'\in \supp \mu}\mu(g')f(g'^{-1}g)$. 
It is easy to see that $T_{\mu}$ is a self-adjoint operator, $\|T_{\mu}\|=1$, and $T_{\mu}(\one_{G})=\one_G$ where $\one_G$ is the constant function one on $G$. Let $L^2(G)^{\circ}:=\{f\in L^2(G)|\h f\perp \one_{G}\}$. Hence $T_{\mu}$ induces a self-adjoint operator $(T_{\mu})|_{L^2(G)^{\circ}}$ on $L^2(G)^{\circ}$, and let $\lambda(\mu;G):=\|(T_{\mu})|_{L^2(G)^{\circ}}\|$ be its operator norm. Notice that, if eigenvalue $1$ has multiplicity one, then 
\[
\lambda(\mu;G)=\sup\{|\lambda| |\h \lambda \in {\rm spec}(T_{\mu}), \h \lambda<1\}.
\]
We say $\mu$ has {\em spectral gap} if $\lambda(\mu;G)<1$. 
So if the group generated by the support of $\mu$ is dense in $G$, the random walk on $G$ with respect to $\mu$ equidistributes exponentially fast if $\mu$ has spectral gap.  
The following {\em uniform spectral gap} is the main result of this article.  
\begin{thm}\label{t:main}
Let $\Omega$ be a finite symmetric subset of $\GL_n(k)$, where $k$ is a number field. Let $\Gamma=\langle \Omega\rangle$ and $\Gamma_{\pfr}$ be its closure in $\GL_n(k_{\pfr})$ for any finite place $\pfr$ of $k$. Let $\ccal:=\{\pfr \in V_f(k)|\h \Gamma_{\pfr} \text{ is compact}\}$, where $V_f(k)$ is the set of all finite places of $k$. 

Suppose the Zariski-closure $\bbg$ is a semisimple group. Then there is $\lambda_0<1$ such that  
\[
\sup_{\pfr\in\ccal} \lambda(\pcal_{\Omega};\Gamma_{\pfr})\le \lambda_0<1,
\] 
where $\pcal_{\Omega}$ is the probability counting measure on $\Omega$.
\end{thm}

\subsection{Applications.} 
\subsubsection{Explicit construction of expanders.}
A family $\{X_i\}_{i=1}^{\infty}$ of $d_0$-regular graphs is called expanders if there is $\delta_0>0$ such that for any $i$ and any non-empty proper subset $A$ of the set $V(X_i)$ of vertices of $X_i$ we have 
\[
\frac{|E(A,A^c)|}{\min\{|A|,|A^c|\}}>\delta_0
\]
where $A^c=V(X_i)\setminus A$ and $E(A,A^c)$ is the set of edges with one vertex in $A$ and the other in $A^c$. What makes expanders interesting is the fact that they are {\em sparse} and at the same time {\em highly connected}. Expanders are extremely useful in communication and theoretical computer science (see a survey by Hoory, Linial and Widgerson~\cite{HLW}). In the past decade they have been found to be extremely useful in a wide range of pure math problems, e.g. affine sieve~\cite{BGS, SGS}, sieve in groups~\cite{LM}, variation of Galois representations~\cite{EHK}, etc. (see~\cite{MSRI} for a collection of surveys of related works and applications).
  
Margulis~\cite{Mar} observed the connection between spectral gap properties of a finitely generated dense subgroup of a profinite group and explicit construction of expanders (e.g. we refer the interested reader to~\cite[Theorem 4.3.2]{Lub} for more details.). Based on this observation and Theorem~\ref{t:main}, we get
\begin{thm}\label{t:expanders}
Let $\Omega$ be a finite symmetric subset of $\GL_n(k)$, where $k$ is a number field. Let $\Gamma=\langle \Omega\rangle$ and $\ccal:=\{\pfr \in V_f(k)|\h \Omega\subseteq \GL_n(\ocal_{\pfr})\}$, where $\ocal_{\pfr}$ is the ring of integers of $k_{\pfr}$. Suppose the Zariski-closure $\bbg$ of $\Gamma$ in $(\GL_n)_k$ is semisimple. Then the family of Cayley graphs
\[
\{\Cay (\pi_{\pfr^m}(\Gamma),\pi_{\pfr^m}(\Omega))|\h \pfr\in \ccal, m\in \bbz^+\}
\]
is a family of expanders, where $\pi_{\pfr^m}$ is the residue map $\pi_{\pfr^m}:\GL_n(\ocal_{\pfr})\rightarrow \GL_n(\ocal_{\pfr}/\pfr^m)$.
\end{thm}
\subsubsection{Local spectral gap.} In a recent work Boutonnet, Ioana, and the author~\cite{BISG} introduced {\em local spectral gap} for a measure class preserving action, and proved many interesting properties of such actions. As an application of Theorem~\ref{t:main}, we get local spectral gap in the $p$-adic setting (For the definitions and basic facts on $p$-adic analytic groups that are used here see~\cite[1.1, Proposition 1.9, 1.2, Proposition 1.14, 8.4, Corollary 8.33]{DDMS}): 
\begin{thm}\label{t:LocalSepctralGap}
Let $G$ be a semisimple $p$-adic analytic group. Let $\Gamma$ be a countable dense subgroup of $G$. Suppose that $\Ad(\Gamma)$ consists of elements with algebraic entries in some $\bbq_p$-basis of $\Lie(G)$. 

Then the left translation action $\Gamma \act (G,m_G)$ has local spectral gap.
\end{thm}
\begin{proof}
By \cite[Remark 1.5]{BISG}, we are free to choose any measurable subset $B\subseteq G$ with compact closure and non-empty interior and prove that the left translation action $\Gamma \act (G,m_G)$ has local spectral gap with respect to $B$. 
So, since $G$ is $p$-adic analytic, we can  and will assume that $B$ is a (topologically) finitely generated pro-$p$ group. Hence the Frattini subgroup $\Phi(B)$ is an open subgroup of $B$. Since $\Gamma$ is dense in $G$, for any coset $b\Phi(B)$ of $\Phi(B)$ in $B$ there is $\gamma_b \in \Gamma\cap b\Phi(B)$. Hence there is a finite symmetric subset $\Omega$ of $\Gamma\cap B$ which maps onto $B/\Phi(B)$. Thus $\langle \Omega\rangle$ is a dense subgroup of $B$. 

By our assumption, there is a number field $k$ and a $\bbq_p$-basis $\Bfr$ of $\Lie(G)$ such that all the entries of $\Ad(\Omega)$ in the basis $\Bfr$ are in $k$. Note that $k$ comes with an embedding into $\bbq_p$. Let $\pfr_0$ be the corresponding place of $k$. Let 
\[
\Gamma_0:=\langle [\Ad(\Omega)]_{\Bfr}\rangle \subseteq \GL_{\dim G}(k),
\] 
and $\bbg_0$ be its Zariski-closure in $(\GL_{\dim G})_k$. Since $\Gamma_0$ is dense in an open subgroup of $\Ad(G)$, we have that $G$ and $\bbg_0(k_{\pfr_0})=\bbg_0(\bbq_p)$ are locally isomorphic. Therefore $\bbg_0$ is semisimple. So by Theorem~\ref{t:main} we have $\lambda(\pcal_{\Omega};\Ad(B))<1$, which implies that $\Gamma\act (G,m_G)$ has local spectral gap.   
\end{proof}

\subsubsection{Banach-Ruziewicz problem.} In~\cite{BISG}, it is proved that the left translation action $\Gamma \act (G,m_G)$ having local spectral gap has many implications. Now some of these implications are mentioned.
\begin{thm}\label{t:BanachRuziewicz}
Let $\Gamma$ be a countable dense subgroup of a semisimple $p$-adic analytic group $G$. Suppose $\Ad(\Gamma)$ consists of matrices with algebraic entries in some $\bbq_p$-basis of $\Lie(G)$. Denote by $\ccal(G)$  the family of measurable subsets $A\subseteq G$ with compact closure. 

Then up to a scalar there is a unique $\Gamma$-invariant, finitely additive measure $\nu:\ccal(G)\rightarrow [0,\infty)$.
\end{thm}
\begin{proof}
By \cite[Theorem D]{BISG}, we know that, up to a multiplicative constant, the Haar measure of $G$ is the unique finitely additive $\Gamma$-invariant measure defined on $\ccal(G)$ if and only if the left translation action $\Gamma\act (G,m_G)$ has local spectral gap. Hence by Theorem~\ref{t:LocalSepctralGap} we get the desired result.
\end{proof}

\subsubsection{Orbit equivalence.} Another application of (local) spectral gap is in the theory of orbit equivalence of actions. There have been many breakthroughs in this subject in the past 15 years (see the surveys~\cite{Pop,Fur,Gab}). Let us recall the notion of {\em orbit equivalence} for two left translation actions. Suppose $\Gamma\subseteq G$ and $\Lambda \subseteq H$ are dense subgroups. We say the left translation actions $\Gamma \act (G,m_G)$ and $\Lambda \act (H,m_H)$ are orbit equivalent if there exists a measure class preserving Borel isomorphism $\theta:G\rightarrow H$ such that $\theta(\Gamma g)=\Lambda \theta(g)$, for $m_G$-almost every $g\in G$. Of course any action is orbit equivalent to an isomorphic copy of itself, i.e. if there is a topological isomorphism $\delta:G\rightarrow H$ such that $\delta(\Gamma)=\Lambda$, then $\Gamma\act (G,m_G)$ and $\Lambda\act (H,m_H)$ are orbit equivalent. In such a case, we say the two actions are {\em conjugate}. 

Two actions that are orbit equivalent can be far from being conjugate. In fact, surprisingly when $\Gamma$ and $\Lambda$ are infinite amenable groups, the mentioned actions are orbit equivalent~\cite{OW,CFW}. In the past decade it has been shown that in the presence of certain (local) spectral gap, however, the actions are {\em rigid}, i.e. orbit equivalence implies conjugation~\cite{Ioa1,Ioa2,BISG}. 

\begin{thm}\label{t:OER}
Let $G$ be a semisimple $p$-adic analytic group, and $\Gamma\subseteq G$ be a countable dense subgroup. Suppose $\Ad(\Gamma)$ consists of matrices with algebraic entries in some $\bbq_p$-basis of $\Lie(G)$.
Let $H$ be a locally compact, second countable, Hausdorff topological group, and $\Lambda$ be a countable dense subgroup of $H$. Suppose $H$ has a profinite open compact subgroup.

If $\Gamma\act (G,m_G)$ and $\Lambda\act (H,m_H)$ are orbit equivalent, then there are open compact subgroups $G_0\subseteq G$ and $H_0\subseteq H$ and a topological isomorphism $\delta:G_0\rightarrow H_0$ such that $\delta(\Gamma\cap G_0)=\Lambda\cap H_0$.
\end{thm}
\begin{proof}
Let $G_0'\subseteq G$ be a pro-$p$ open subgroup of $G$. As in the proof of Theorem~\ref{t:LocalSepctralGap}, by Theorem~\ref{t:main}, we have that $\Gamma\cap G_0'\act (G_0',m_{G_0'})$ has spectral gap.  

Now let $H_0'\subseteq H$ be a profinite open subgroup. Since $\Gamma\act (G,m_G)$ and $\Lambda\act (H,m_H)$ are orbit equivalent, $\Gamma\cap G_0' \act (G_0',m_{G_0'})$ and $\Lambda\cap H_0'\act (H_0',m_{H_0'})$ are stably orbit equivalent (see \cite[Definition 1.2]{Ioa2} for its definition). Hence by \cite[Theorem A]{Ioa1} there are open subgroups $G_0$ and $H_0$, and a topological isomorphism $\delta:G_0\rightarrow H_0$ such that $\delta(\Gamma\cap G_0)=\Lambda\cap H_0$. 
\end{proof}
\subsubsection{Super-approximation.}\label{ss:Super-Approximation}
As it was pointed out earlier, the connection between spectral gap and construction of expanders was first observed by Margulis~\cite{Mar} where he proved that the family of Cayley graphs of finite quotients of a finitely generated group with property(T) is a family of expanders. Later due to works of several mathematicians~\cite{Sel,Kaz,BuSa,Clo,CU}, it was proved that the same holds for {\em congruence quotients} of an $S$-arithmetic group of a semisimple group. In the spirit of strong approximation, Lubotzky~\cite{123} asked if this is an {\em algebraic} phenomenon or an {\em arithmetic} one.\footnote{Following A. Kontorovich's suggestion, I call this phenomenon {\em super-approximation}. It is worth pointing out that this phenomenon has been called {\em superstrong approximation}~(e.g. see \cite{MSRI}) by some authors.} More precisely, it was asked if the Cayley graphs of $\SL_2(\f_p)$ with respect to the generating set $\left\{\left[\begin{array}{cc}
1&\pm 3\\0&1	
\end{array}\right],\left[\begin{array}{cc}
1&0\\\pm 3&1	
\end{array}\right]\right\}$ is a family of expanders, where $\f_p$ is the finite field with $p$ elements. In \cite{Gam} it was proved that, if Hausdorff dimension of the limit set of $\Gamma=\langle \Omega\rangle\subseteq \SL_2(
\bbz)$ is larger than $5/6$, then $\{\Cay(\pi_p(\Gamma),\pi_p(\Omega))\}_p$ is a family of expanders as $p$ ranges over primes.

Later based on Helfgott's product theorem~\cite{Hel1}, Bourgain and Gamburd in their seminal work~\cite{BG1} proved the same result for any finite symmetric subset $\Omega\subseteq \SL_2(\bbq)$ as long as $\Gamma=\langle \Omega \rangle$ is Zariski-dense in $\mathbb{SL}_2$. Since then, based on their machinery and generalizations of Helfgott's work~\cite{Hel2,BGT,PS}, this result has been extended in several works~\cite{BG2,BG3,BGS,Var,BV,SGV}. 

As it was pointed out earlier, Theorem~\ref{t:main} implies that {\em super-approximation} holds for prime power modulus when the Zariski-closure of the group is semisimple. In a subsequent work based on Theorem~\ref{t:main}, the author proves the best possible result for prime power modulus.
\begin{thm}\label{t:PerfectPrimePower}
Let $\Omega$ be a finite symmetric subset of $\GL_{n_0}(\bbz[1/q_0])$ and $\Gamma=\langle \Omega \rangle$. Then 
\[
\{\Cay(\pi_{p^n}(\Gamma),\pi_{p^n}(\Omega))\}_{p\nmid q_0, n\in \bbz^+}
\]
 is a family of expanders if and only if the connected component $\bbg^{\circ}$ of the Zariski-closure $\bbg$ of $\Gamma$ is perfect, i.e. $\bbg^{\circ}=[\bbg^{\circ},\bbg^{\circ}]$.
\end{thm}

\subsubsection{Comparing the semisimple case of Theorem~\ref{t:PerfectPrimePower} with Theorem~\ref{t:main}.}
Theorem~\ref{t:main} gives us spectral gap at any finite place $\pfr$ where $\Gamma$ is bounded, but it seems that Theorem~\ref{t:PerfectPrimePower} only deals with finite places of $\bbq$ which do not divide the denominator $q_0$. Let's briefly explain why Theorem~\ref{t:PerfectPrimePower} implies the following:
\begin{cor}\label{c:Perfect}
Let $\Omega$ be a finite symmetric subset of $\GL_{n_0}(\bbq)$, $\Gamma=\langle \Omega\rangle$, and 
\[
\ccal(\Gamma):=\{p\in V_f(\bbq)|\h \Gamma \text{ is a bounded subgroup of }\GL_{n_0}(\bbq_p)\}.
\]
Then $\sup_{p\in \ccal(\Gamma)}\lambda(\pcal_{\Omega};\GL_{n_0}(\bbq_p))<1$ if and only if the connected component $\bbg^{\circ}$ of the Zariski-closure $\bbg$ of $\Gamma$ is perfect. 
\end{cor}
\begin{proof}
For any $p\in \ccal(\Gamma)$, there is $g_p\in \GL_{n_0}(\bbq_p)$ such that $g_p \Gamma g_p^{-1}\subseteq \GL_{n_0}(\bbz_p)$ as all the maximal compact subgroups of $\GL_{n_0}(\bbq_p)$ are conjugates of $\GL_{n_0}(\bbz_p)$. Moreover, since $\Gamma$ is finitely generated, for all except finitely many $p$ we can assume $g_p=1$. Since $\GL_{n_0}(\bbq)\cdot \GL_{n_0}(\bbr)\prod_{p\in V_f(\bbq)}\GL_{n_0}(\bbz_p)=\GL_{n_0}(\bba)$ (see \cite[Proposition 8.1]{PR}), where $\GL_{n_0}(\bbq)$ is embedded diagonally and $\bba$ is the ring of adeles, there is $g\in \GL_{n_0}(\bbq)$ such that $g\GL_{n_0}(\bbz_p)=g_p \GL_{n_0}(\bbz_p)$ for any $p\in V_f(\bbq)$. Hence we have $g\Gamma g^{-1} \subseteq \GL_{n_0}(\bbz_p)$ for any $p\in \ccal(\Gamma)$. Therefore $g\Gamma g^{-1}\subseteq \GL_{n_0}(\bbz[1/q_0])$, where $q_0=\prod_{p\in V_f(\bbq)\setminus \ccal(\Gamma)} p$. So using Theorem~\ref{t:PerfectPrimePower} for $g\Omega g^{-1}$, we get the desired result (see \cite[Theorem 4.3.2]{Lub}). 
\end{proof}
Next we point out that Corollary~\ref{c:Perfect} does {\em not} imply Theorem~\ref{t:main}:

Suppose, as in Theorem~\ref{t:main}, $\Gamma=\langle\Omega\rangle$ is a finitely generated subgroup $\GL_{n_0}(k)$ where $k$ is a number field, and its Zariski-closure $\bbg$ in $(\GL_{n_0})_k$ is semisimple. So the Zariski-closure of $\Gamma$ in the Weil restriction of scalars $R_{k/\bbq}(\bbg)$ is semisimple (see Section~\ref{ss:StrongApproximation}). Let $\Gamma'$ and $\Omega'$ be the copies of $\Gamma$ and $\Omega$, respectively, which are embedded in $\GL_{n_0d_0}(\bbq)$ where $d_0=[k:\bbq]$. Hence Corollary~\ref{c:Perfect} implies that 
\[
\sup\{\lambda(\pcal_{\Omega'};\GL_{n_0d_0}(\bbz_p))|\h \Gamma' \text{ is a bounded subgroup of } \GL_{n_0d_0}(\bbq_p)\}<1. 
\]
On the other hand, $\Gamma'$ is a bounded subgroup of $\GL_{n_0d_0}(\bbq_p)$ if and only if $\Gamma$ is a bounded subgroup of $\GL_{n_0}(k_{\pfr})$ for any $\pfr\in V_f(k)$ which divides $p$. Hence, assuming $k$ is a finite Galois extension of $\bbq$, Corollary~\ref{c:Perfect} implies uniform spectral gap {\em only} for places $\pfr$ where $\Gamma$ is bounded in $\GL_{n_0}(k_{\pfr'})$ for {\em all} the Galois conjugates $\pfr'$ of $\pfr$ (see Section 2.4). 

For instance, suppose $\Gamma_0$ is an arithmetic lattice in $\SL_{n_0}(\bbq_{p_0})$, and its trace field $k$ is {\em not} $\bbq$ (see \cite[Section 3.2]{MSG} for a description of such groups). Then as part of its arithmetic description, there are an algebraic group $\bbg$ defined over $k$ and a place $\pfr_0\in V_f(k)$ such that
\begin{enumerate}
\item  $k_{\pfr_0}\simeq \bbq_{p_0}$ and $\bbg\times_k k_{\pfr_0}\simeq (\SL_{n_0})_{\bbq_{p_0}}$,
\item For any Archimedean place $\nu$, we have that $\bbg(k_{\nu})$ is compact,
\item $\Gamma_0$ is commensurable with $\bbg(k)\cap \prod_{\pfr\neq \pfr_0} P_{\pfr}$ where $\bbg(k)$ is embedded diagonally and $P_{\pfr}$ is a compact open subgroup of $\bbg (k_{\pfr})$.
\end{enumerate}
Since $k\neq \bbq$ and $k_{\pfr_0}\simeq \bbq_{p_0}$, $k$ has a place $\pfr_0'\neq \pfr_0$ which divides $p_0$. Hence a Zariski-dense subgroup $\Gamma=\langle \Omega \rangle$ of $\Gamma_0$ is unbounded in $\bbg(k_{\pfr_0})$ and bounded in $\bbg(k_{\pfr'_0})$, and Corollary~\ref{c:Perfect} does {\em not} imply that $\lambda(\Omega;\bbg(k_{\pfr'}))<1$.  
  
\subsection{Comparing with the previous related works.}\label{ss:RelatedResults}
There are two prior results related to Theorem~\ref{t:main}. Both  Bourgain-Gamburd~\cite{BG2,BG3} and Bourgain-Varj\'{u}~\cite{BV} focus on the case of finitely generated, Zariski-dense subgroup $\Gamma=\langle \Omega \rangle$ of $\SL_{n_0}(\bbz)$. In \cite{BG2,BG3}, it is proved that $\lambda(\pcal_{\Omega};\Gamma_p)<\lambda(p)<1$ if $p$ is large enough depending on $\Omega$. In~\cite{BV}, using a different technique, it is proved that $\lambda(\pcal_{\Omega};\wh{\Gamma})<1$ where $\wh{\Gamma}$ is the closure of $\Gamma$ in $\prod_{p\in V_f(\bbq)} \SL_{n_0}(\bbz_p)$; in particular, $\sup\{\lambda(\pcal_{\Omega};\Gamma_p)|\h p\in V_f(\bbq)\}<1$. 

Both~\cite{BV} and \cite{BG3} rely on Archimedean dynamics. The former is based on~\cite{BFLM} and the latter uses random matrix theory in $\SL_{n_0}(\bbz)$. And both of these works need the existence of proximal elements in the adjoint representation over $\bbr$. In particular, their arguments cannot be extended to the case where  $\bbg(\bbr)$ is compact.

Our proof is inspired by the approach presented in~\cite{BG3}. Instead of random matrix theory, however, the main result of~\cite{SGV} is used. 
\section*{Acknowledgments.} 
I am in debt of Adrian Ioana for many insightful conversations about orbit equivalence rigidity, and in particular for telling me how to get such a result from the main theorem of this article. I would like to thank Kiran Kedlaya for the fruitful discussion about complete intersection varieties. I am thankful of Hee Oh for asking me about the $p$-adic version of local spectral gap after my talk at the MSRI. Special thanks are due to anonymous referee for his/her thorough report which made this article much better to read. 
\section{Preliminary results}

\subsection{Spreading of weight under convolution, and $l^2$-flattening.}\label{ss:BalogSzemerediGowers}
Let me start with the following (easy) lemma (see~\cite[Remark, page 1060]{BG3}).
\begin{lem}\label{l:BGRemark}
Let $H$ be a group, $\mu$ be a symmetric probability measure on $H$, and $A$ be a symmetric subset of $H$. If $l>l_0$ are two positive integers, then $\mu^{(2l_0)}(A\cdot A)\ge \mu^{(l)}(A)^2$. 
\end{lem}
\begin{proof}
First we notice that, since $\mu^{(l)}(A)=\sum_{h\in H}\mu^{(l-l_0)}(h)\mu^{(l_0)}(h^{-1}A)$, $\mu^{(l)}(A)\le \max_{h\in H}\mu^{(l_0)}(hA)$. Now since $A$ and $\mu$ are symmetric, we also have that $\mu^{(2l_0)}(A\cdot A)\ge \max_{h\in H}\mu^{(l_0)}(hA)^2$. 
\end{proof}

Let me recall the {\em $l^2$-flattening phenomenon} proved by Bourgain-Gamburd (see \cite[Lemma 15]{Var} and \cite{BG1}) which is based on the non-commutative version of Balog-Szemer\'{e}di-Gowers Theorem~\cite{Tao}. 
\begin{lem}\label{l:BSG}
Let $H$ be a finite group and $\mu$ be a probability measure on $H$. Let $K$ be a positive number. 
If $\|\mu *\mu\|_2\ge \frac{1}{K}\|\mu\|_2$, then there is a subset $A$ of $H$ with the following properties:
\begin{enumerate}
\item $\frac{1}{K^R\|\mu\|_2^2}\le |A|\le \frac{K^R}{\|\mu\|_2^2}$,
\item $|A\cdot A\cdot A|\le K^R |A|$,
\item $\min_{h\in A} (\widetilde{\mu}*\mu)(h)\ge \frac{1}{K^R|A|}$, where $\widetilde{\mu}(h)=\mu(h^{-1})$ for any $h\in H$,
\end{enumerate}
where $R$ is an absolute constant. 
\end{lem}

\subsection{Regularization.}~\label{ss:Regularization}
In this section, we recall a regularization of a subset of a rooted regular tree due to Bourgain and Gamburd~\cite[Section A.3]{BG3} (see Corollary~\ref{c:Regularization}). The proofs are included for the convenience of the reader.
\begin{definition}\label{d:RegularRootedTree}
\begin{enumerate}
	\item Let $T=T_{k,n}$ be a rooted tree such that each vertex has $k$ children and it has $n$ generations (levels). The root is considered the zero-generation. 
	\item For any $0\le i\le n$, let $T^{(i)}$ be the set of vertices in the $i$-th generation.
	\item For any $0\le i\le j\le n$, let $\pi_{i,j}:T^{(j)}\rightarrow T^{(i)}$ be the projection map (which gives the ``ancestor'' of a vertex).
	\item For any vertex $v\in T^{(i)}$, let $B(v)=\pi_{i,n}^{-1}(v)\subseteq T^{(n)}$ (the set of all the grand children of $v$ after $n-i$ generations).
	\item For a subset $A$ of $T^{(n)}$ and a vertex $x\in T^{(i)}$, let $\deg_A(x)=|\pi_{i+1,n}(B(x)\cap A)|.$ 
\end{enumerate} 
\end{definition}
\begin{lem}\label{l:ParentsGeneration}
Let $k,n$ be two positive integers, $T=T_{k,n}$ and $A\subseteq T^{(n)}$. Then there is $k'$, a power of 2, and $A'\subseteq A$ such that
\begin{enumerate}
	\item For every $x\in \pi_{n-1,n}(A')$, we have $|B(x)\cap A'|=k'$.
	\item $|A'|\ge |A|/(2\log k)$.
\end{enumerate}
\end{lem}
\begin{proof} 
We have that 
\[
|A|=\sum_{x\in \pi_{n-1,n}(A)}|B(x)\cap A|=
\sum_{i=0}^{\lfloor\log k\rfloor}\left(\sum_{\begin{array}{l} x\in \pi_{n-1,n}(A)\\ 2^i\le \deg_A(x)<2^{i+1}\end{array}}\deg_A x\right).
\] 
So there are $1\le i\le \lfloor\log k \rfloor$ and  
\[
A'\subseteq \bigcup_{\begin{array}{l} x\in \pi_{n-1,n}(A)\\ 2^i\le \deg_A(x)<2^{i+1}\end{array}} (B(x)\cap A)
\]
such that $k'=2^i$ and $A'$ satisfy the desired properties.
\end{proof}
\begin{lem}\label{l:Regularization}
Let $k,n$ be two positive integers, $\vare>0$, $T=T_{k,n}$ and $A\subseteq T^{(n)}$. Suppose that $|A|\ge |T^{(n)}|^{\vare}=k^{n\vare}$ and $k^{\vare/4}>2\log k$ (i.e. $k$ is ``large enough'' depending on $\vare$). Then there are $B\subseteq A$ and a positive integer $m\le n(1-\vare/4)$ such that
\begin{enumerate}
	\item $\pi_{m,n}(B)=\{v\}$.
	\item For any $m\le l\le n$, we have that $|\pi_{l,n}(B)|\ge k^{(l-m)\vare/2}$.
\end{enumerate}
In particular, we have $|B|\ge k^{n\vare^2/8}=|T^{(n)}|^{\vare^2/8}$.
\end{lem}
\begin{proof} 
Repeated use of Lemma~\ref{l:ParentsGeneration}, we get a sequence of dyadic numbers $k_0, k_1, \ldots, k_{n-1}$ and a chain of subsets of $A$
\[
A_0\subseteq A_1\subseteq \cdots\subseteq A_{n-1}\subseteq A_n=A,
\] 
such that
\begin{enumerate}
	\item For any $0\le i\le n-1$ and any $x\in \pi_{i,n}(A_i)$, we have $\deg_{A_i}(x)=\cdots=\deg_{A_{n-1}}(x)=k_i$. And so $|A_i|=|\pi_{j,n}(A_i)|\cdot k_{j}\cdot \cdots \cdot k_{n-1}$ for any $i\le j\le n-1$.
	\item For any $1\le i\le n$, we have $|\pi_{i,n}(A_{i-1})|\ge |\pi_{i,n}(A_i)|/(2\log k)$. And so $|A_{i-1}|\ge |A_i|/(2\log k)$ and $|A_i|\ge |A|/(2 \log k)^{n-i}$.
\end{enumerate}
In particular, we have $|A_0|=k_0\cdot\cdots k_{n-1}\ge |A|/(2\log k)^n\ge k^{n\vare}/k^{n\vare/4}=k^{3n\vare/4}$. Let
\[
m=\max\{i|\h 0\le i\le n-1,\h \prod_{j=0}^{i-1} k_j< k^{i\vare/2}\}.
\] 
Then for any $x\in \pi_{m,n}(A_0)$ and any $m+1\le l\le n$ we have 
\[
|\pi_l(B(x)\cap A_0)|=\prod_{i=m}^{l-1} k_i=(\prod_{i=0}^{l-1} k_i)/(\prod_{i=0}^{m-1} k_i)\ge k^{l\vare/2}/k^{m\vare/2}=k^{(l-m)\vare/2}.
\]
On the other hand, we have
\[
k^{3n\vare/4}\le |A_0|=\prod_{i=0}^{n-1} k_i=\prod_{i=0}^{m-1} k_i\cdot \prod_{i=m}^{n-1} k_i< k^{m\vare/2} \cdot k^{n-m}=k^{n-m+m\vare/2}.
\]
And so $m<n(1-3\vare/4)/(1-\vare/2)<n(1-\vare/4)$, which means $B=B(x)\cap A_0$ for $x\in \pi_{m,n}(A_0)$ and $m$ satisfy the desired properties. 
\end{proof}
\begin{cor}\label{c:Regularization}
Let $k,n$ be two positive integers, $\vare>0$, $T=T_{k,n}$ and $A\subseteq T^{(n)}$. Suppose that $|A|\ge |T^{(n)}|^{\vare}=k^{n\vare}$ and $k^n\gg_{\vare} 1$. Then there are $B\subseteq A$ and a positive integer $m\le n(1-\vare/4)$ such that
\begin{enumerate}
	\item $\pi_{m,n}(B)=\{v\}$.
	\item For any $m\le l\le n$, we have that $|\pi_{l,n}(B)|\gg_{\vare} k^{(l-m)\vare/4}$.
\end{enumerate}
In particular, we have $|B|\geq k^{n\vare^2/32}=|T^{(n)}|^{\vare^2/32}$. 
\end{cor}
\begin{proof}
Let $K(\varepsilon)$ be the smallest positive real number such that $K(\varepsilon)^{\varepsilon/4}\ge 2\log K(\varepsilon)$. And let $s$ be the smallest positive integer such that $k^s\ge K(\varepsilon)$. Now let us consider $T':=T_{k^s,\lfloor n/s\rfloor}$ and identify $\pi_{s\lfloor n/s\rfloor,n}(A)$ with a subset $A'$ of $T'^{(\lfloor n/s\rfloor)}$. We notice that $|A'|\ge K(\varepsilon)^{-1} |A|\ge  K(\varepsilon)^{-1} k^{n\vare}\ge k^{n\vare/2}$ if $k^{n}\ge K(\vare)^{2/\vare}$. So by Lemma~\ref{l:Regularization} there are $B'\subseteq A'$ and a positive integer $m'\le \lfloor n/s\rfloor(1-\vare/8)$ such that
\begin{enumerate}
\item $\pi_{sm',s\lfloor n/s\rfloor}(B')=\{v\}$.
\item For any $m'\le l\le \lfloor n/s\rfloor$, we have that $|\pi_{sl,s\lfloor n/s\rfloor}(B)|\ge k^{s(l-m')\varepsilon/4}$.
\end{enumerate}
Let $B\subseteq A$ be such that $\pi_{s\lfloor n/s\rfloor,n}(B)=B'$ and let $m=sm'$. Then for any positive integer $l$ between $m$ and $n$ we have 
\[
|\pi_{l,n}(B)|\ge K(\varepsilon)^{-1} |\pi_{s\lfloor l/s\rfloor,s\lfloor n/s\rfloor}(B')|\ge K(\varepsilon)^{-1} k^{s(\lfloor l/s\rfloor-m')\vare/4}\ge K(\varepsilon)^{-2} k^{(l-m)\varepsilon/4}.
\]
In particular, $|B|\ge K(\varepsilon)^{-2} k^{n\vare^2/16}\ge k^{n\vare^2/32}$ if $k^n\ge K(\varepsilon)^{\max\{64/\vare^2,2/\vare\}}$.
\end{proof}

\subsection{Reduction to the connected, simply-connected case.}\label{ss:GoingSimplyConnected} 
First let us quickly recall a definition of property ($\tau$).
\begin{definition}\label{d:tau}
Let $\Omega$ be a finite generating set of a group $\Gamma$, and $\{N_i\}_{i=1}^{\infty}$ be a family of finite index normal subgroups of $\Gamma$. 
\begin{enumerate}
\item For $\vare>0$, we say a unitary representation $\rho$ of $\Gamma$ does not have an $\vare$-almost invariant vector if 
$\max_{s\in\Omega} \|\rho(s)(v)-v\|\ge \vare \|v\|$ for any vector $v\in V_{\rho}$. 
\item We say $\Gamma$ has property($\tau$) with respect to $\{N_i\}$, if there is $\vare>0$, such that any unitary representation $\rho$ of $\Gamma$ which factors through $\Gamma/N_i$ for some $i$ does not have an $\vare$-almost invariant vector.
\end{enumerate} 
\end{definition}
Let us make a few well-known remarks, in the above setting (See~\cite[Chapter 4.3]{Lub} or \cite[Chapter 1.4]{LZ} for proofs of these facts and other basic results about property($\tau$).): 
\begin{remark}\label{r:BasicPropertiesTau}
\begin{enumerate}
\item Though we have fixed a generating set $\Omega$ in Definition~\ref{d:tau}, having property($\tau$) with respect to $\{N_i\}$ depends only on $\Gamma$ and $\{N_i\}$, and it does {\em not} depend on $\Omega$. Starting with any other generating set only results in changing the constant $\vare$. 
\item Using complete reducibility of unitary representations, one can see that $\Gamma$ has property($\tau$) with respect to $\{N_i\}$ if and only if there is $\vare>0$ such that  any non-trivial {\em irreducible} unitary representation of $\Gamma$ which factors through $N_i$ for some $i$ does not have an $\vare$-almost invariant vector. 
\item Suppose $\{N_i\}$ is a decreasing chain. Then $\Gamma$ has property($\tau$) with respect to $\{N_i\}$ if and only if there is $\vare>0$ such that any non-trivial irreducible continuous unitary representation of $\varprojlim \Gamma/N_i$ does not have an $\vare$-almost invariant vector.
\item $\Gamma$ has property($\tau$) with respect to $\{N_i\}$ if and only if $\sup_i \lambda(\pcal_{\pi_{N_i}(\Omega)};\Gamma/N_i)<1$ where $\pi_{N_i}:\Gamma\rightarrow \Gamma/N_i$ is the natural quotient map.
\item Suppose $\{N_i\}$ is a decreasing chain and $\cap_i N_i=1$. Then, by Peter-Weyl theorem, one has 
\[
\lambda(\pcal_{\Omega}; \varprojlim \Gamma/N_i)=\sup_i \lambda(\pcal_{\pi_{N_i}(\Omega)};\Gamma/N_i).
\]
\item $\Gamma$ has property($\tau$) with respect to $\{N_i\}$ if and only if the family of Cayley graphs $\{{\rm Cay}(\Gamma/N_i,\pi_{N_i}(\Omega))\}_i$ is a family of expanders. 
\end{enumerate}
\end{remark}

Next we recall~\cite[Theorem 8]{AE} which is crucial in the reduction to the connected, simply-connected case.

\begin{lem}\label{l:FiniteIndexSubgroup}
Let $\Gamma$ be a finitely generated group generated by a symmetric set $\Omega$, and $\Lambda$ be a finite index normal subgroup of $\Gamma$. Let $\{N_i\}$ be a family of finite index normal subgroups of $\Gamma$. Then $\Gamma$ has property($\tau$) with respect to $\{N_i\}$ if and only if $\Lambda$ has property($\tau$) with respect to $\{N_i\cap \Lambda\}$. 
\end{lem}
\begin{proof}
The {\em only if} part is proved in \cite[Theorem 8]{AE}.

Suppose $S=\{s_1,\ldots,s_k\}$ is a symmetric generating set of $\Gamma$ and at the same time $\Gamma=\bigcup_{i=1}^k s_i\Lambda$. So for any $1\le i,j\le k$ there is $1\le l=l(i,j)\le k$ such that $s_l\Lambda=s_is_j\Lambda$. Let $s'_{i,j}:=s_{l(i,j)}^{-1}s_is_j\in \Lambda$. Then we claim that $\Lambda$ is generated by $S':=\{s'_{i,j}\}\cup (S\cap \Lambda)$. Let $\Lambda_1$ be the group generated by $S'$. Clearly it is a subgroup of $\Lambda$. By the definition of $s'_{i,j}$ we have that $s_{l(i,j)}\Lambda_1=s_is_j\Lambda_1$. And so $\Gamma=\bigcup_{i=1}^k s_i\Lambda_1$ since $S$ generates $\Gamma$ and is symmetric. Therefore we get that $\Lambda_1=\Lambda$.

Now assume that $\Lambda$ has property($\tau$) with respect to $\{N_i\cap \Lambda\}$. So there is $\vare>0$ such that for any unitary representation $\theta$ of $\Lambda$ which factors through $\Lambda\cap N_i$ and has no non-zero fixed vector, we have
\be\label{e:NotAlmostInvariant2}
\max_{s'\in S'}\|\theta(s')(u)-u\|>\vare \|u\|, 
\ee
for any vector $u\in V_{\theta}$.
 
Let $\rho$ be an irreducible representation of $\Gamma$ which factors through $\Gamma/N_i$ and does not factor through $\Gamma/\Lambda$. Suppose that $v$ is a unit vector in $V_{\rho}$ which is $\vare'$-almost invariant vector with respect to $S$, i.e.
\[
\|\rho(s)(v)-v\|\le \vare',
\]
for any $s\in S$. Therefore for any $s'\in S'$ we have
\[
\|\rho(s')(v)-v\|\le 3\vare'.
\]

Since $\Lambda$ is a normal subgroup of $\Gamma$, for the trivial representation $\rho_0$ of $\Lambda$ we have that the restriction of the induced representation $(\ind_{\Lambda}^{\Gamma} \rho_0)_{\Lambda}$ is the direct sum of $[\Gamma:\Lambda]$-many copies of the trivial representation $\rho_0$. By Frobenius reciprocity, we have that either the restriction $\rho_{\Lambda}$ of $\rho$ to $\Lambda$ does not have a non-zero fixed vector, or $\rho$ is a subrepresentation of $\ind_{\Lambda}^{\Gamma}(\rho_0)$. By the above discussion, however, the latter cannot occur as $\rho$ doers not factor through $\Gamma/\Lambda$. Hence $\rho_{\Lambda}$ does not have a non-zero fixed vector, and it factors through $\Lambda\cap N_i$. Therefore we have
\[
\max_{s'\in S'} \|\rho(s')(v)-v\|\ge \vare,
\]     
for any unit vector $v\in V_{\rho}$. Thus by (\ref{e:NotAlmostInvariant2}) we have $\vare<3 \vare'$, which implies that any non-trivial representation of $\Gamma/N_i$ has no $\vare''$-almost invariant vector with respect to $S$ for some $0<\vare''<\vare/3$. 
\end{proof}
\begin{cor}\label{c:FiniteIndexSubgroup}
Let $\Gamma$ be a finitely generated group and $\Omega$ be a symmetric finite generating set of $\Gamma$. Assume $\Lambda$ is a normal finite index subgroup of $\Gamma$. Then $\Lambda$ has a finite symmetric generating set $\Omega'$, and for any infinite set $\{N_i\}$ of normal finite index subgroups of $\Gamma$,  $\{\Cay(\Gamma/N_i,\pi_{N_i}(\Omega))\}_i$ is a family of expanders if and only if $\{\Cay(\Lambda/(N_i\cap \Lambda),\pi_{N_i\cap \Lambda}(\Omega'))\}_i$ is a family of expanders.
\end{cor}
\begin{proof}
This is a direct corollary of Lemma~\ref{l:FiniteIndexSubgroup} and \cite[Theorem 4.3.2]{Lub}.
\end{proof} 

Let $\Gamma^{\circ}=\Gamma\cap \bbg^{\circ}$. So $\Gamma^{\circ}$ is a Zariski-dense subgroup of $\bbg^{\circ}$ which is a normal finite index subgroup of $\Gamma$. Let 
\[
S:=\{\pfr\in V_f(k)|\h \Gamma \text{ is unbounded in } \GL_{n}(k_{\pfr})\}.
\]
For any $\pfr\in V_f(k)\setminus S$, let $\Gamma_{\pfr}$ (resp. $\Gamma^{\circ}_{\pfr}$) be the closures of $\Gamma$ (resp. $\Gamma^{\circ}$) in $\GL_n(k_{\pfr})$. Let $\{N_{i,\pfr}\}$ be a chain of open normal subgroups of $\Gamma_{\pfr}$ which forms a basis for the neighborhoods of the identity. Then, by Remark~\ref{r:BasicPropertiesTau}, we have $\lambda(\pcal_{\Omega};\Gamma_{\pfr})=\sup_i \lambda(\pcal_{\Omega};\Gamma_{\pfr}/N_{i,\pfr})$ and $\lambda(\pcal_{\Omega^{\circ}};\Gamma^{\circ}_{\pfr})=\sup_i \lambda(\pcal_{\Omega^{\circ}};\Gamma^{\circ}_{\pfr}/\Gamma^{\circ}_{\pfr}\cap N_{i,\pfr})$. Hence, by Lemma~\ref{l:FiniteIndexSubgroup} and Remark~\ref{r:BasicPropertiesTau}, it is enough to prove $\sup_{\pfr\in V_f(k)\setminus S} \lambda(\pcal_{\Omega^{\circ}}; \Gamma^{\circ}_{\pfr})<1$. So from this point on we can and will assume that $\bbg$ is a Zariski-connected semisimple group. Let $\wt{\bbg}$ be the simply-connected form of $\bbg$ that is defined over $k$. Let $\iota:\wt{\bbg}\rightarrow \bbg$ be the $k$-covering map, $\wt{\Lambda}:=\iota^{-1}(\Gamma)\cap \wt{\bbg}(k)$, and $\Lambda:=\iota(\wt{\Lambda})$. By a similar argument as in~\cite[Lemma 24]{SGS} we have that $\Lambda$ is a normal finite index subgroup of $\Gamma$. Hence again by Lemma~\ref{l:FiniteIndexSubgroup} and Remark~\ref{r:BasicPropertiesTau}, without loss of generality we can assume that $\Gamma=\Lambda$. 

On the other hand, $\iota$ induces a continuous homomorphism with finite kernel from $\wt{\bbg}(k_{\pfr})$ to $\bbg(k_{\pfr})$. So after fixing a $k$-embedding $\wt{\bbg}\xrightarrow{f} \mathbb{GL}_{n_0}$ and passing to a finite index subgroup (as we are allowed by Corollary~\ref{c:FiniteIndexSubgroup}), if needed, we can assume that  $f(\wt{\Lambda})\subseteq \GL_{n_0}(\ocal_k(S))$ and $\Gamma\subseteq \GL_n(\ocal_k(S))$, where 
\[
S:=\{\pfr\in V_f(k)|\h \Gamma \text{ is unbounded in } \GL_{n}(k_{\pfr})\}.
\]
And there is $(m_{\pfr})\in \bigoplus_{\pfr\in V_f(k)} \bbz$ such that for any $\pfr\in V_f(k)$,
\[
|\iota(\wt{\lambda})-1|_{\pfr}\le |N_{k/\bbq}(\pfr)|^{m_{\pfr}}|f(\wt{\lambda})-1|_{\pfr}.
\]
Hence $\iota$ induces an epimorphism from $\pi_{\pfr^n}(f(\wt{\Lambda}))$ onto $\pi_{\pfr^{n-m_{\pfr}}}(\Gamma)$. So $\{\Cay(\pi_{\pfr^n}(f(\wt{\Lambda})),\pi_{\pfr^n}(f(\wt{\Omega})))|\h \pfr\not\in S, n\in \bbz^+\}$ is a family of expanders if and only if $\{\Cay(\pi_{\pfr^n}(\Gamma),\pi_{\pfr^n}(\Omega))|\h \pfr\not\in S, n\in \bbz^+\}$  is a family of expanders. This implies that without loss of generality we can assume $\bbg$ is simply-cosnnected. So we have proved that:

\begin{lem}\label{l:BoundedIntegral}
To prove Theorem~\ref{t:main}, it is enough to prove the following:

Let $k$ be a number field, $S$ be a finite subset of $V_f(k)$, and $\Omega$ be a finite symmetric subset of $\GL_{n_0}(\ocal_k(S))$, where $\ocal_k(S)$ is the ring of $S$-integers of $k$. Let $\Gamma=\langle \Omega\rangle$ and $\bbg$ be its Zariski-closure in $(\GL_{n_0})_k$. Suppose $\bbg$ is a Zariski-connected, simply-connected, semisimple group. Then 
\[
\sup_{\pfr\in V_f(k)\setminus S} \lambda(\pcal_{\Omega}; \Gamma_{\pfr})<1,
\]
where $\Gamma_{\pfr}$ is the closure of $\Gamma$ in $\GL_{n_0}(k_{\pfr})$.
\end{lem}

\subsection{Strong approximation and the $\pfr$-adic closures of $\Gamma$.}\label{ss:StrongApproximation}Here using strong approximation~\cite[Theorem 5.4]{Nor}, we describe structure of the closure  $\Gamma_{\pfr}$ of $\Gamma$ in $\GL_{n_0}(\ocal_{\pfr})$ for any $\pfr\not\in S$. Consider $\Gamma$ as a subgroup of $R_{k/\bbq}(\bbg)(\bbq)$ and let $\bbg_1$ be its Zariski-closure. Since $R_{k/\bbq}(\bbg)$ is isomorphic to 
$\prod_{\sigma\in \hom(k,\overline{\bbq})} \bbg^{\sigma}$ over $k$ and the projection of $\Gamma$ to each factor is Zariski-dense, $\bbg_1$ is a semisimple $\bbq$-group. Since $\bbg$ is simply-connected, so is $\bbg_1$. Let 
\[
S_0:=\{p\in V_f(\bbq)|\h \pfr| p \text{ for some } \pfr\in S\}.  
\]
Let us fix a $k$-embedding of $\bbg$ in $\GL_{n_0}$, a $\bbz$-basis of $\ocal_k$, and fix the induced $\bbq$-embedding of $R_{k/\bbq}(\bbg)$ and get a $\bbq$-embedding of $\bbg_1$ in $\GL_{n_1}$.
Then $\Gamma\subseteq \GL_{n_1}(\bbz_{S_0})$. Hence by Nori's strong-approximation, after going to a finite index subgroup if necessary (it is allowed by Corollary~\ref{c:FiniteIndexSubgroup}), we can assume the following:
\begin{enumerate}
	\item The closure $\wh{\Gamma}$ of $\Gamma$ in $\prod_{p\not \in S_0}\bbg_1(\bbz_p)$ is an open compact subgroup. 
	\item $\wh{\Gamma}=\prod_{p\not \in S_0} K_p$ where $K_p$ is an open compact subgroup of $\bbg_1(\bbq_p)$.
	\item For large enough $p$, $K_p$ is a hyperspecial parahoric (see~\cite[Section 3.8]{Tit} for its definition and basic properties).
	\item If $K_p$ is not a hyperspecial parahoric, then it is a uniformly powerful pro-$p$ group (see \cite{DDMS} for the definition of a uniform pro-$p$ group).  
	\item For any $p\not \in S_0$, either $\pi_p(\Gamma)$ is perfect or it is trivial. And $\Gamma_{\pfr}$ is a quotient of $K_p$. 
\end{enumerate}
Since $\Gamma_{\pfr}$ is a quotient of $K_p$ for any $p\not\in S_0$ and $\pfr|p$, we have that $\lambda(\pcal_{\Omega};K_p)\ge \lambda(\pcal_{\Omega}; \Gamma_{\pfr})$. Let 
\[
\widetilde{S}:=\{\pfr \in V_f(k)|\h \pfr\not \in S, \pfr|p \text{ for some } p\in S_0\}.
\]
We split the proof of Theorem~\ref{t:main} into two parts: proving a uniform spectral gap for $k=\bbq$ and proving a spectral gap for a fixed $\pfr\in V_f(k)\setminus S$. 

\begin{lem}\label{l:StructurePAdicClosure}
Let $k$ be a finite Galois extension of $\bbq$. Let $\Gamma$ be a finitely generated subgroup of $\GL_{n_0}(k)$. Suppose the Zariski-closure $\bbg$ of $\Gamma$ is a connected, simply-connected, semisimple group. Then for any
\[
\pfr\in \ccal(\Gamma):=\{\pfr\in V_f(k)|\h \Gamma \text{ is a bounded subgroup of } \bbg(k_{\pfr})\}
\]
there is a subfield $k(\pfr)$ of $k$, a Zariski-connected, simply-connected,  semisimple $k(\pfr)$-group $\bbg_{1,\pfr}$ such that 
\begin{enumerate}
\item Under the natural embedding of $k$ into $k_{\pfr}$, $k(\pfr)$ is mapped to $\bbq_p$.
\item The closure $\Gamma_{\pfr}$ of $\Gamma$ in $\GL_{n_0}(k_{\pfr})$ is an open compact subgroup of $\bbg_{1,\pfr}(\bbq_p)$.
\item $R_{k(\pfr)/\bbq}(\bbg_{1,\pfr})$ is naturally isomorphic to the Zariski-closure $\bbg_1$ of $\Gamma$ in $R_{k/\bbq}(\bbg)$. 
\end{enumerate}
\end{lem}
\begin{proof}
For any $\pfr\in \ccal$, by assumption, the closure $\Gamma_{\pfr}$ of $\Gamma$ in $\bbg(k_{\pfr})$ is a compact group. 
Since $\Gamma_{\pfr}$ is a $p$-adic analytic compact group, it has a finite index uniformly powerful pro-$p$ subgroup. Since $\bbg$ is connected, any finite index subgroup of $\Gamma$ is still Zariski-dense in $\bbg$. So without loss of generality, we can and will assume that $\Gamma_{\pfr}$ is a uniformly powerful pro-$p$ group, and the logarithmic map induces 
a bijection between $\Gamma_{\pfr}$ and a Lie $\bbz_p$-subalgebra $\gfr_{\pfr}$ of $\Lie(\bbg)(k_{\pfr})$.
Since $\Gamma$ is Zariski-dense in $\bbg$, the $k_{\pfr}$-span of $\gfr_{\pfr}$ is the entire $\Lie(\bbg)(k_{\pfr})$. 
And so $\gfr_{\pfr}\otimes_{\bbz_p} \bbq_p$ is a semisimple Lie algebra. Hence the Zariski-closure $\overline{\bbg}_{1,\pfr}$ of $\Gamma_{\pfr}$ in $R_{k_{\pfr}/\bbq_p}(\bbg\times_k k_{\pfr})$ is a semisimple $\bbq_p$-group and $\Gamma_{\pfr}$ is an open compact subgroup of $\overline{\bbg}_{1,\pfr}(\bbq_p)$. 

 Let $\gal(k/\bbq)$ be the Galois group, and $\gal(k/\bbq;\pfr):=\{\sigma\in\gal(k/\bbq)|\h \sigma(\pfr)=\pfr\}$ be the decomposition group. Let $k(\pfr)$ be the subfield of fixed points of $\gal(k/\bbq;\pfr)$. It is well-known that the restriction map induces an isomorphism between the Galois group $\gal(k_{\pfr}/\bbq_p)$ and the decomposition group $\gal(k/\bbq;\pfr)$. In particular, the natural embedding of $k$ in $k_{\pfr}$ sends $k(\pfr)$ to $\bbq_p$. And so we have a (natural) $\bbq_p$-isomorphism 
\be\label{e:ReorderingRestrictionOfScalars}
R_{k_{\pfr}/\bbq_p}(\bbg\times_k k_{\pfr})\simeq R_{k/k(\pfr)}(\bbg)\times_{k(\pfr)} \bbq_p.
\ee
There is a natural projection ${\rm Pr}_{\pfr}$ from $R_{k/\bbq}(\bbg)$ to $R_{k/k(\pfr)}(\bbg)$ that is defined over $k(\pfr)$. Let $\bbg_{1,\pfr}:={\rm Pr}_{\pfr}(\bbg_1)$. By (\ref{e:ReorderingRestrictionOfScalars}), we get the following natural $\bbq_p$-isomorphism
\be\label{e:pAdicAlgebraicStructure}
\overline{\bbg}_{1,\pfr}\simeq \bbg_{1,\pfr}\times_{k(\pfr)} \bbq_p,
\ee
which completes the proof. 
\end{proof}

An important application of strong approximation and the above discussion is the following:
\begin{lem}\label{l:ClosureOfZariskiDenseSubgroup}
Let $\Omega\subseteq \GL_{n_0}(k)$ be a finite symmetric set. Suppose the Zariski-closure $\bbg$ of the group $\Gamma$ generated by $\Omega$ is a Zariski-connected, simply-connected, semisimple group. Let 
\[
\ccal(\Gamma):=\{\pfr\in V_f(k)|\h \Gamma \text{ is a bounded subgroup of } \bbg(k_{\pfr})\}.
\]
Suppose $\Lambda$ is a finitely generated subgroup of $\Gamma$ which is Zariski-dense in $\Gamma$ considered as a subgroup of $R_{k/\bbq}(\bbg)(\bbq)$. Then for any $\pfr\in \ccal(\Gamma)$ we have that 
$\Lambda_{\pfr}$ is a finite-index subgroup of $\Gamma_{\pfr}$, where $\Gamma_{\pfr}$ (resp. $\Lambda_{\pfr}$) is the closure of $\Gamma$ (resp. $\Lambda$) in $\bbg(k_{\pfr})$. Furthermore for almost all $\pfr\in \ccal(\Gamma)$ we have 
$\Gamma_{\pfr}=\Lambda_{\pfr}$. 
\end{lem}
\begin{proof}
Since $\Gamma$ is finitely generated, $S:=V_f(k)\setminus \ccal(\Gamma)$ is a finite set. Passing to a finite-index subgroup of $\Gamma$, if needed, we can assume that $\Gamma\subseteq \GL_{n_0}(\ocal_k(S))$. Let $\wt{S}:=\{p \in V_f(\bbq)|\h \pfr|p \text{ for some } \pfr\in S\}$. Then $\Gamma\subseteq R_{\ocal_k/\bbz}(\GL_{n_0})(\bbz_{\wt{S}})$. By strong approximation, the closure $\widehat{\Lambda}$ of $\Lambda$ in 
\[
\prod_{p\in V_f(\bbq)\setminus \wt{S}} \bbg_1(\bbz_p)
\]
is an open subgroup. Hence $\widehat{\Lambda}$ is a finite-index subgroup of $\widehat{\Gamma}$, which proves that for almost all $\pfr$ we have $\Gamma_{\pfr}=\Lambda_{\pfr}$. 

For a given $\pfr$, we know that $\Lambda_{\pfr}$ is an open subgroup of a $p$-adic analytic group. And so it is open in the $\bbq_p$-points of the Zariski-closure of $R_{k_{\pfr}/\bbq_p}(\bbg)(\bbq_p)$. Since $\Lambda$ is Zariski-dense in $\bbg_1$ (when considered as a subgroup of $R_{k/\bbq}(\bbg)(\bbq)$), its Zariski-closure as a subgroup of $R_{k_{\pfr}/\bbq_p}(\bbg)(\bbq_p)$ is ${\rm Pr}_{\pfr}(\bbg_1)\times_{k(\pfr)} \bbq_p$ (by the argument in the proof of Lemma~\ref{l:StructurePAdicClosure}). This implies that $\Lambda_{\pfr}$ is an open subgroup of $\Gamma_{\pfr}$. 
\end{proof}

\subsection{Ping-pong players and the main result of~\cite{SGV}.} One of the essential ingredients of this work is the following case of~\cite[Corollary 6]{SGV}.
\begin{thm}\label{t:ExpanderSquareFree}
Let $k$ be a number field, $S$ be a finite subset of $V_f(k)$, and $\Omega$ be a finite symmetric subset of $\GL_{n_0}(\ocal_k(S))$, where $\ocal_k(S)$ is the ring of $S$-integers of $k$. Let $\Gamma=\langle \Omega\rangle$ and $\bbg$ be its Zariski-closure in $(\GL_{n_0})_k$. Suppose $\bbg$ is a semisimple group. Then $\{\Cay(\pi_{\pfr}(\Gamma),\pi_{\pfr}(\Omega))\}_{\pfr\in V_f(k)\setminus S}$ is a family of expanders.
\end{thm}

As in~\cite{SGV}, first we change the probability law of the random walk. The following can be deduced as a byproduct of the proof of \cite[Proposition 6]{SGV}.
\begin{prop}\label{p:EscapeSubgroups}.
In the setting of Theorem~\ref{t:main}, let $\bbg_1$ be the Zariski-closure of $\Gamma$ in $R_{k/\bbq}(\bbg)$. Then there are a finite subset $\overline{\Omega'}$ of $\Gamma$ and $\delta_0>0$ and $l_0$ (which depend on $\Omega$) such that $\overline{\Omega'}$ freely generates a Zariski-dense subgroup of $\bbg_1$ and
\be\label{e:EscapeSubgroups}
\pcal_{\Omega'}^{(l)}(\bbh)\le e^{-\delta_0 l}
\ee
for any proper subgroup $\bbh$ of $\bbg_1$ and $l\ge l_0$, where $\Omega'=\overline{\Omega'}\cup\overline{\Omega'}^{-1}$. 

Furthermore if $\bbg_1^{(i)}$ are $\bbq$-almost simple factors of $\bbg_1$, then we can also assume that ${\rm pr}_i(\Ad(\overline{\Omega'}))$ freely generates a Zariski-dense subgroup of $\Ad(\bbg_1^{(i)})$, where ${\rm pr}_i:\Ad(\bbg)\rightarrow \Ad(\bbg_1^{(i)})$ are the projection maps. In particular, $\ker({\rm pr}_i\circ\Ad)\cap \langle \overline{\Omega'}\rangle=\{1\}$.
\end{prop}
\begin{proof}
This is essentially proved in~\cite[Proposition 17, Proposition 20, Proposition 7]{SGV} (for more general groups). Let us quickly give the outline of the argument for the semisimple case. 

By \cite[Proposition 17]{SGV}, there are finitely many non-trivial irreducible representations $\rho_i$ of $\bbg_1$ defined over local fields $F_i$ such that 
\begin{enumerate}
\item $\rho_i(\Gamma)\subseteq \GL(V_{\rho_i})(F_i)$ is unbounded for any $i$, and
\item for any proper infinite algebraic subgroup $\bbh$ of $\bbg_1$ for some $i$ there is an element of the projective space $[w]\in {\bf P}(V_{\rho_i})$ such that $\rho_i(\bbh)([w])=[w]$.
\end{enumerate}
 We add all the the irreducible subrepresentations of $\wedge \Lie(\bbg)$ that factor through $\bbg_1^{(j)}$ for some $j$. For any such subrepresentation $\rho_i$ we can find a local field such that $\rho_i(\Gamma)$ is unbounded as these representations are defined over number fields and $\rho_i(\Gamma)$ is Zariski-dense in $\rho_i(\bbg_1)$.  
 
 By \cite[Proposition 20, Proposition 21]{SGV}, there is a finite subset $\overline{\Omega'}\subseteq \Gamma$ such that
 \begin{enumerate}
 \item For any $i$, $\rho_i(\overline{\Omega'})$ freely generates a subgroup of $\rho_i(\Gamma)$.
 \item There are $l_0, c>0$ such that for any $i$, $[w]\in {\bf P}(V_{\rho_i})(F_i)$, and integer $l\ge l_0$ we have
\[
|\{\gamma\in B_l(\overline{\Omega'})|\h \rho_i(\gamma)([w])=[w]\}|<|B_l(\overline{\Omega'})|^{1-c}, 
\]
where $B_l(\overline{\Omega'})$ is the set of reduced words over $\overline{\Omega'}$ of length at most $l$.
 \end{enumerate}
 
By the above geometric description of proper algebraic subgroups, we have that 
\[
|B_l(\overline{\Omega'})\cap \bbh|<|B_l(\overline{\Omega'})|^{1-c},
\]
for any proper algebraic subgroup $\bbh$ and any $l\ge l_0$. Now using a result of Kesten~\cite[Theorem 3]{Kes} and Cauchy-Schwarz as in the proof of \cite[Proposition 6]{SGV}, one can finish the proof.

On the other hand, since for any $j$ there is an $i$ such that $\rho_i$ is a subrepresentation of ${\rm pr}_j\circ \Ad$, we have that ${\rm pr}_j(\Ad(\overline{\Omega}'))$ freely generates a subgroup of $\Ad(\bbg_1^{(j)})$. 
\end{proof}

\begin{lem}\label{l:p-adicClosureZariskiDenseSubgroup}
Let $\Omega$ be a finite symmetric subset of $\GL_{n_0}(k)$ and $\Gamma=\langle \Omega\rangle$. Assume the Zariski-closure $\bbg$ of $\Gamma$ is a connected, simply-connected, semisimple $k$-group. Suppose $\Omega'\subseteq \Gamma$ is a finite symmetric set which generates a Zariski-dense subgroup $\Lambda$ of $\Gamma$ viewed as a subgroup of $R_{k/\bbq}(\bbg)(\bbq)$. Let 
\[
\ccal(\Gamma):=\{\pfr\in V_f(k)|\h \Gamma \text{ is a bounded subgroup of } \bbg(k_{\pfr})\}.
\]

Then if $\sup\{\lambda(\pcal_{\Omega'};\Lambda_{\pfr})|\h \pfr\in \ccal(\Gamma)\}<1$, then $\sup\{\lambda(\pcal_{\Omega};\Gamma_{\pfr})|\h \pfr\in \ccal(\Gamma)\}<1$.
\end{lem}
\begin{proof}
Suppose the contrary. So there are $\pfr_i\in \ccal(\Gamma)$, non-trivial unitary  irreducible representations $\rho_i:\Gamma_{\pfr_i}\rightarrow \ucal(V_i)$, and unit vectors $v_i\in V_i$ such that
\be\label{e:AlmostInvariantFunctions}
\|\rho_i(\gamma)(v_i)-v_i\| \rightarrow 0
\ee
for any $\gamma\in \Omega$. Since $\Gamma=\langle \Omega\rangle$, for any $\gamma\in \Gamma$, Equation (\ref{e:AlmostInvariantFunctions}) holds. In particular, $\{v_i\}$ are almost invariant under $\Omega'$. 
Since $\sup\{\lambda(\pcal_{\Omega'};\Lambda_{\pfr})|\h \pfr\in \ccal(\Gamma)\}<1$, a sequence of almost invariant functions are invariant from some point on. Which means the restriction of $\rho_i$ to $\Lambda_{\pfr_i}$ has a fixed point. Hence $\Lambda_{\pfr_i}\neq \Gamma_{\pfr_i}$. By Lemma~\ref{l:ClosureOfZariskiDenseSubgroup}, there are only finitely many of such $\pfr_i$. So passing to a subsequence, if needed, we can assume that $\pfr_i=\pfr$ is fixed. Again by Lemma~\ref{l:ClosureOfZariskiDenseSubgroup}, $\Lambda_{\pfr}$ is a finite index  subgroup of $\Gamma_{\pfr}$. So by Frobenius reciprocity, there are only finitely many  irreducible representations of $\Gamma_{\pfr}$ whose restriction to $\Lambda_{\pfr}$ has the trivial representation as a subrepresentation. Hence passing to a subsequence, if needed, we can assume that $\rho_i=\rho$ is a fixed representation. This is clearly impossible as $\rho$ is a finite-dimensional non-trivial irreducible representation of $\Gamma_{\pfr}$ and such a representation cannot have almost invariant vectors.   
\end{proof}

\subsection{Summary of the initial reductions.}\label{ss:InitialReductions} In this short section, for the convenience of the reader, all the reductions are summarized. It is enough to prove Theorem~\ref{t:main} under the following conditions:
\begin{enumerate}
\item The Zariski-closure $\bbg$ of $\Gamma$ as a subgroup of $\GL_{n_0}(k)$ is a {\em Zariski-connected, simply-connected}, semisimple $k$-group.
\item $\Omega=\overline{\Omega}\cup \overline{\Omega}^{-1}$ is a subset of $\GL_{n_0}(\ocal_k(S))$, where $S$ is a finite subset of $V_f(k)$. And, for any $\pfr\in S$, $\Gamma$ is {\em unbounded} in $\GL_{n_0}(k_{\pfr})$. 
\item For any $\bbq$-almost simple factor $\bbg_1^{(i)}$ of the Zariski-closure $\bbg_1$ of $\Gamma$ in $R_{k/\bbq}(\bbg)$, ${\rm pr}_i(\Ad(\overline{\Omega}))$ {\em freely generates} a Zariski-dense subgroup of $\Ad(\bbg_1^{(i)})$.
\item For any proper algebraic $\bbq$-subgroup $\bbh$ of $\bbg_1$ and positive integer $l\gg_{\Omega} 1$, we have
\[ 
\pcal_{\Omega}^{(l)}(\bbh(\bbq))\le e^{-\Theta_{\Omega}(l)}.
\]
\item $\{\Cay(\pi_{\pfr}(\Gamma),\pi_{\pfr}(\Omega))\}_{\pfr\in V_f(k)\setminus S}$ is a family of expanders.
\end{enumerate}
It is worth pointing out that by Lemma~\ref{l:FiniteIndexSubgroup} and a similar argument as in the proof of Lemma~\ref{l:BoundedIntegral}, we have.
\begin{lem}\label{l:RepIntegrality}
For a given finite set of representations $\rho_i:\bbg_1\rightarrow \mathbb{GL}_{n_i}$ that are defined over a finite Galois extension $k'$ of $k$, passing to a finite-index subgroup of $\Gamma$, if necessary, we can assume that the following holds: for $\wt{\pfr}\in V_f(k')$, if $\rho_i(\Gamma)$ is a bounded subgroup of $\GL_{n_i}(k'_{\wt{\pfr}})$, then $\rho_i(\Gamma)\subseteq \GL_{n_i}(\ocal_{\wt{\pfr}})$.  
\end{lem}
  
\subsection{Escaping from subschemes, and ramified primes of a scheme.} As it was observed in~\cite{BGS}, one of the important consequences of having spectral gap modulo primes (see Theorem~\ref{t:ExpanderSquareFree}) is the fact that the probability of hitting a proper subvariety decays exponentially (see Proposition~\ref{p:EscapeSubvariety}). Subsets of a finitely generated group with this property are called {\em exponentially small} by Lubotzky-Meiri~\cite{LM}. Since the set of non-regular elements of a semisimple group is a proper subvariety, it was observed in~\cite{LM} that the set of non-regular elements is an exponentially small set.

Saying that $\gamma$ is a regular element is equivalent to saying that the connected component of the centralizer subgroup $C_{\bbg}(\gamma)$ is a (maximal) torus. In this work, however, this type of information is not enough. Here one needs to understand the structure of the centralizer subgroups $C_{\pi_{\pfr^n}(\Gamma)}(\pi_{\pfr}^n(\gamma))$. For this purpose, one needs to know that the {\em distance} between $\gamma$ and the variety of non-regular elements is at least $p^{-c n}$ for a small positive $c$. In fact, one needs to find such $\gamma$ in a way that $C_{\pi_{\pfr^n}(\Gamma)}(\pi_{\pfr^n}(\gamma))$ intersects a given approximate subgroup $\pi_{\pfr^n}(A)$ in a large set. This forces us to look at the scheme-theoretic closure of non-regular elements and its shifts by elements of $\Gamma$. And we have to escape these schemes in $\Theta(n\log p)$-steps\footnote{This kind of approach is inspired by~\cite{BG2} where they use random matrix theory to get a similar result for a Zariski-dense subgroup of $\SL_n(\bbz)$.} (see Proposition~\ref{p:q-nonRegularSemisimple}).   
Results of this section rely on the existence of points with {\em small} logarithmic height which is proved in Appendix A based on arithmetic B\'{e}zout. 

Before stating this proposition, let us see a lemma on {\em ramified primes}.

Using the usual convention, for a field $F$, $\overline{F}$ denotes its algebraic closure. 
\begin{prop}~\label{p:RamifiedPrime}
Let $k$ be a number field and $S$ be a finite subset of $V_f(k)$. Let $\wcal=\Spec(A)$, where $A=\ocal_k(S)[X_1,\ldots,X_n]/\langle f_1,\ldots,f_{m}\rangle$. If $\pfr|p$ and $\log p \gg h:=\max_i\{h(f_i)\}$ (see Appendix A for the definition of logarithmic height $h(f)$ of a polynomial $f$) where the constant depends on $\deg f_i$, $n$, and $k$, then 
\[
\dim \wcal\times_{\ocal_k(S)}\overline{\f_{\pfr}}=\dim \wcal\times_{\ocal_k(S)}\overline{k}.
\]
 In particular, $|\wcal(\f_p)|\ll p^d$, where $d=\dim \wcal\times_{\ocal_k(S)}\overline{k}$ and the implied constant depends just on $\deg(f_i)$ and $k$. 
\end{prop}
\begin{proof}
By Lemma~\ref{l:CompleteIntersection} in Appendix A, there are
$
\wt{f}_1,\ldots, \wt{f}_{n-d}\in \sum_j \bbz f_j,
$
such that $h(\wt{f}_i)\ll \max_j h(f_i)$ where the implied constant depends on $n$ and $\deg f_j$. And the generic fiber of $\wt{\wcal}:=\ocal_k(S)[\underline{X}]/\langle \wt{f}_1,\ldots,\wt{f}_{n-d}\rangle$ is geometrically a complete intersection variety, i.e. $\dim(\overline{\bbq}[\underline{X}]/\langle \wt{f}_1,\ldots,\wt{f}_{n-d}\rangle)= d$. 

Now by repeated application of Lemma~\ref{l:CuttingByHyperplane} in Appendix A and Lemma~\ref{l:CompleteIntersection} in Appendix A, there are functions $\wt{f}_{n-d+1},\ldots,\wt{f}_n \in \sum_j\bbz f_j$ such that $h(\wt{f}_i)\ll \max_j h(f_i)$ where the implied constant depends on $n$ and $\deg f_j$. And $\dim(\overline{\bbq}[\underline{X}]/\langle \wt{f}_1,\ldots,\wt{f}_{n}\rangle)=0$.

 By effective Nullstellensatz~\cite[Theorem IV]{MW}, we have that there are $p_i(x_i)\in \ocal_k(S)[x_i]$ (single variable $x_i$) and $h_{ij}\in\ocal_k(S)[\underline{X}]$ such that
 \begin{enumerate}
 \item For any $i$, $p_i(x_i)=h_{i1} \wt{f}_1+ \cdots +h_{in} \wt{f}_n$.
 \item For any $i$ and $j$, $\deg p_i, \deg h_{ij} \ll 1$ and $h(p_i), h(h_{ij})\ll h$, where the implied constants depend on $\max_i\{\deg f_i\}$, $n$, and $k$. 
 \end{enumerate} 
  
Hence $\dim(\ocal_k(S)[\underline{X}]/\langle \wt{f}_1,\ldots,\wt{f}_{n}\rangle\otimes_{\ocal_k(S)} \f_{\pfr})=0$ if $\log p\gg h$. Thus $\dim \wt{\wcal}\times_{\ocal_k(S)} \f_{\pfr}=d$ if $\log p\gg h$.  

On the other hand, for any $\pfr$ we have 
\[
\dim \wt{\wcal} \times_{\ocal_k(S)} \overline{\f_{\pfr}}\ge \dim\wcal \times_{\ocal_k(S)} \overline{\f_{\pfr}}\ge \dim \wcal\times_{\ocal_k(S)} \overline{k}=\dim \wt{\wcal} \times_{\ocal_k(S)} \overline{k},
\] 
which implies the first part of Proposition.

The rest can be deduced using generalized B\'{e}zout and \cite[Lemma 1]{LW} (also see \cite[Lemma 3.1]{FHJ}). 
\end{proof}

\begin{definition}\label{d:RamifiedPrime}
Let $k$ be a number field and $S$ be a finite subset of $V_f(k)$. Let 
\[
\wcal:=\Spec(\ocal_k(S)[X_1,\ldots,X_n]/\langle f_1,\ldots,f_{m}\rangle).
\]
 We say $\pfr\in V_f(k)\setminus S$ is a {\em ramified place of} $\wcal$ if $\dim \wcal\times_{\ocal_k(S)}\overline{\f_{\pfr}}\neq \dim \wcal\times_{\ocal_k(S)}\overline{k}.$ And we let 
\[
p_0(\wcal):=\inf \{p\in V_f(\bbq)|\h p'\in V_f(\bbq),
p'\geq p, \pfr'| p'\implies \text{ is NOT a ramified place of } \wcal\}.
\]
\end{definition}
Hence the following is a direct consequence of Proposition~\ref{p:RamifiedPrime} and Definition~\ref{d:RamifiedPrime}.
\begin{cor}\label{c:RamifiedPrime}
Let $k$ be a number field and $S$ be a finite subset of $V_f(k)$. Let \[
\wcal:=\Spec(\ocal_k(S)[X_1,\ldots,X_n]/\langle f_1,\ldots,f_{m}\rangle).
\]
Then $\log p_0(\wcal)\ll \max_i h(f_i)$ where the implied constant depends on $k$, $n$, and $\deg f_i$.
\end{cor}
\begin{lem}\label{l:SizeOfQuotient}
Let $\Gamma$ be a finitely generated subgroup of $\GL_m(\bbz_{S_0})$. Let $\gcal$ be the Zariski-closure of $\Gamma$ in $(\GL_m)_{\bbz_{S_0}}$. Suppose $\gcal\times_{\bbz_{S_0}} \overline{\bbq}$ is a Zariski-connected, simply-connected, semisimple group. Then for large enough $p$ depending on the embedding of $\Gamma$ in $\GL_m(\bbz_{S_0})$ we have
\begin{enumerate}
\item $\gcal\times_{\bbz_{S_0}}\overline{\f}_p$ is a Zariski-connected, simply-connected, semisimple group,
\item $\pi_p(\Gamma)=\gcal(\f_p)$,
\item $p^d/2< |\gcal(\f_p)|< 2 p^d$ where $d:=\dim \gcal\times_{\bbz_{S_0}} \overline{\bbq}$.
\end{enumerate} 
\end{lem}
\begin{proof}
Part (1) is proved in~\cite[Lemma 64]{SGV} and Part (2) is proved in~\cite[Theorem 41]{SGV}. Since $\gcal_p:=\gcal \times_{\bbz_{S_0}} \f_p$ is geometrically irreducible, by Lang-Weil~\cite{LW} we have $|\gcal(\f_p)|=p^d+O(p^{d-1})$, where the implied constant depends on the degree of $\gcal_p$. By classification of semisimple groups over a finite field, one can get an upper bound for the degree of $\gcal_p$ based on its dimension. Hence $|\gcal(\f_p)|=p^d+O_d(p^{d-1})$.   
\end{proof}
\begin{definition}\label{d:LargeModP}
Let $\Gamma$ be a finitely generated subgroup of $\GL_m(\bbz_{S_0})$. If the Zariski-closure of $\Gamma$ in $(\GL_m)_{\bbq}$ is a Zariski-connected, simply-connected, semisimple group, then we let $p_0(\Gamma)$ (to be precise it depends on the {\em embedding} and not just $\Gamma$) denote the smallest prime such that any prime $p'\ge p$ satisfies all the assertions of Lemma~\ref{l:SizeOfQuotient}. 
\end{definition}

\begin{prop}\label{p:EscapeSubvariety}
In the setting of Section~\ref{ss:InitialReductions}, let $S_0:=\{p\in V_f(\bbq)|\h \exists \pfr\in S, \pfr|p\}$. And let $\gcal_1$ be the Zariski-closure of $\Gamma$ in $R_{\ocal_k(S_0)/\bbz_{S}}(\GL_{n_0})$.  Let $\wcal$ be a closed $\bbz_S$-subscheme of $\gcal_1$ given by equations $f_i$'s. Assume that the dimension of the generic fiber of $\wcal$ is less than $\dim \bbg_1$. Then there are a positive number $\delta_0$ and a positive integer $l_0$ depending on $\bbg_1$ such that 
\[
\forall l\ge l_0,\h\h \pcal_{\Omega}^{(l)}(\wcal(\bbz_S))\ll p_0(\wcal)e^{-\delta_0 l},
\]  
where the implied constant depends on the geometric degree of the generic fiber of $\wcal$ and $\Gamma$ (see Lemma~\ref{d:RamifiedPrime} for the definition $p_0(\wcal)$).
\end{prop}
In fact, we will see that the dependence of the implied constant on $\Gamma$ is essentially on $p_0(\Gamma)$, the spectral gap of $\Cay(\pi_p(\Gamma),\pi_p(\Omega))$ for $\pfr\in V_f(\bbq)\setminus S$, and $\dim \bbg_1$. 
Before proving Proposition~\ref{p:EscapeSubvariety}, let us start with the following well-known lemma concerning the rate of convergence of a random walk to the equidistribution.

\begin{lem}\label{l:EquidistributionErrorRate}
Let $\overline{\Omega}$ be a symmetric generating set of a finite group $H$. Assume that the second largest eigenvalue of $(\ast \pcal_{\overline{\Omega}})^2$ is at most $\lambda_1^2$, where $\ast \pcal_{\overline{\Omega}}:l^2(H)\rightarrow l^2(H)$, $\ast\pcal_{\overline{\Omega}}(f):=f\ast \pcal_{\overline{\Omega}}$. Then for any $X\subseteq H$ and any positive integer $l$ we have 
\[
\left|\pcal_{\overline{\Omega}}^{(l)}(X)-\frac{|X|}{|H|}\right|\le \lambda_1^l \sqrt{|X|}.
\]
\end{lem}
\begin{proof}
Let $P:l^2(H)\rightarrow l^2(H)$ be the orthogonal projection to the constant functions, i.e. 
\[
P(f):=\frac{\langle f, \one \rangle}{\langle \one,\one \rangle} \one, 
\]
where $\one(h)=1$ for any $h\in H$ and $\langle f_1,f_2\rangle=\sum_{h\in H}f_1(h)\overline{f_2(h)}$. Let $T=\ast \pcal_{\overline{\Omega}}$. Then $T$ is a self-adjoint operator and leaves $l^2(H)^{\circ}:=\{f\in l^2(H)|\h \langle f,\one\rangle=0\}$ invariant. Moreover, the operator norm of the restriction of $T$ to $l^2(H)^{\circ}$ is at most $\lambda_1$. Hence for any positive integer $l$ 
\be\label{e:l2norm}
\|T^l-P\|\le \lambda_1^l.
\ee
On the other hand, we have that 
\be\label{e:WeightToInnerProduct}
\pcal_{\overline{\Omega}}^{(l)}(X)=\langle \chi_X,T^l(\delta_1)\rangle,
\ee
where $\chi_X$ is the characteristic function of $X$ and $\delta_1$ is the characteristic function of $\{1\}$. We also notice that $P(\delta_1)=\frac{1}{|H|}\one$ and so
\be\label{e:1}
\langle \chi_X, P(\delta_1)\rangle=\frac{|X|}{|H|}.
\ee
Hence by Equations (\ref{e:l2norm}), (\ref{e:WeightToInnerProduct}) and (\ref{e:1}) we have 
\[
\left|\pcal_{\overline{\Omega}}^{(l)}(X)-\frac{|X|}{|H|}\right|=|\langle \chi_X,T^l(\delta_1)\rangle-\langle \chi_X, P(\delta_1)\rangle|=|\langle \chi_X, (T^l-P)(\delta_1)\rangle|\le \|\chi_X\|_2 \|T^l-P\|_2 \|\delta_1\|_2\le \lambda_1^l \sqrt{|X|}.
\]
\end{proof}
\begin{proof}[Proof of Proposition~\ref{p:EscapeSubvariety}]
For any prime $p$ and positive integer $l$ we have that 
\be\label{e:PassageToModp}
\pcal_{\Omega}^{(l)}(\wcal(\bbz_S))\le \pi_p[\pcal_{\Omega}]^{(l)}(\wcal(\f_p)).
\ee

On the other hand, by Theorem~\ref{t:ExpanderSquareFree} and Lemma~\ref{l:EquidistributionErrorRate}, there is $\lambda_1<1$ such that for any prime $p$ and positive integer $l$ we have that
\be\label{e:EscapeModp}
\left|\pi_p[\pcal_{\Omega}]^{(l)}(\wcal(\f_p))-\frac{|\wcal(\f_p)|}{|\pi_p(\Gamma)|}\right|\le \lambda_1^l \sqrt{|\wcal(\f_p)|}.
\ee

We view $\Gamma$ as a subgroup of $R_{\ocal_k(S_0)/\bbz_S}(\GL_{n_0})(\bbz_S)$, and let $p_0(\Gamma)$ be the prime attached to this embedding which is described in Definition \ref{d:LargeModP}. So for $p\ge \max\{p_0(\Gamma),p_0(\wcal)\}$, we have 
\[
\frac{|\wcal(\f_p)|}{|\pi_p(\Gamma)|}\le \frac{2|\wcal(\f_p)|}{p^{\dim \bbg_1}}\ll \frac{p^{\dim \wcal\times_{\bbz_S} \overline{\bbq}}}{p^{\dim \bbg_1}}\le \frac{1}{p},
\]
where the implied constant depends on the geometric degree of the generic fiber of $\wcal$. Hence for $p\ge p_0(\Gamma)$ we have $|\wcal(\f_p)|/|\pi_p(\Gamma)|\ll \frac{p_0(\wcal)}{p}$. Thus by Proposition~\ref{p:RamifiedPrime} and Equations (\ref{e:PassageToModp}) and (\ref{e:EscapeModp}) we have that
\be\label{e:UpperBound1}
\pcal_{\Omega}^{(l)}(\wcal)\ll p_0(\wcal)\left(\frac{1}{p}+\lambda_1^l p^{(d-1)/2}\right),
\ee
where $d=\dim \bbg_1$. Now it is clear that there is a positive number $\delta_0$ depending on $d$ and $\lambda_1$ such that, if $l\ge l_0(d,\lambda_1)$, then there is a prime $p$ with the following properties:
\be\label{e:PrimeSelection} 
1/p\le e^{-\delta_0 l},\h\h\h{\rm and}\h\h\h \lambda_1^lp^{(d-1)/2}\le e^{-\delta_0l}.
\ee
Equations (\ref{e:UpperBound1}) and (\ref{e:PrimeSelection}) give us that 
\[
\pcal_{\Omega}^{(l)}(\wcal)\ll p_0(\wcal)e^{-\delta_0 l},
\]
where the implied constant depends on the degree of the generic fiber of $\wcal$, $d$, $p_0(\Gamma)$ and $\lambda_1$.
\end{proof}
\subsection{Multiplicity bound.}
In order to execute Sarnak-Xue~\cite{SX} trick, we will be needing a lower bound on the dimension of irreducible representation of $p$-adic analytic groups. To that end, Howe's Kirillov theory for compact $p$-adic analytic groups is recalled~\cite{How}.

For any $\pfr\in V_f(k)$, there is a positive integer $a_0:=a_0(\pfr)$ such that $\exp:\pfr^{a_0} \gl_n(\ocal_{\pfr})\rightarrow 1+\pfr^{a_0}\gl_n(\ocal_{\pfr})$, and $\log$ going backward, can be defined and satisfy the usual properties. It is worth mentioning that $a_0(\pfr)=1$ except for finitely many $\pfr$. There is a positive integer $b_0:=b_0(\pfr)$ such that, if $\hfr$ is a Lie $\ocal_{\pfr}$-subalgebra of $\pfr^{a_0}\gl_n(\ocal_{\pfr})$ and $[\hfr,\hfr]\subseteq \pfr^{b_0}\hfr$, then $\exp(\hfr)$ is a subgroup of $\GL_n(\ocal_{\pfr})$.   
\begin{lem}\label{l:DimensionPAdicAnalytic}
Let $\hfr$ be a Lie $\ocal_{\pfr}$-subalgebra of $\pfr^{a_0}\gl_n(\ocal_{\pfr})$. Suppose $\pfr^{s_0}\hfr\subseteq [\hfr,\hfr]\subseteq \pfr^{b_0+1} \hfr$ for some positive integer $s_0$. For any non-negative integer $i$, let $H_i:=\exp(\pfr^{i-1}\hfr)$. Suppose $\rho$  is an irreducible unitary representation of $H_1$, and let 
\[
m_{\rho}:=\min\{m\in \bbz^{\ge 0}|\h H_{m+1}\subseteq \ker \rho\}.
\]  
Then $\dim \rho\ge p^{(\lfloor (m_{\rho}-s_0)/v_{\pfr}(p)\rfloor)/2}$, where $v_{\pfr}(p)$ is the $\pfr$-adic valuation of $p$, defined by $\pfr^{v_{\pfr}(p)}\ocal_{\pfr}=p\ocal_{\pfr}$.  
\end{lem}  
\begin{proof}
By~\cite[Theorem 1.1]{How}, there is $\psi\in \hom(\hfr,S^1)$, where $S^1$ is the group of complex numbers of modulus 1, such that
\[
{\rm ch}(\rho)(\exp x)=|O(\psi)|^{-1/2} \sum_{\phi\in O(\psi)} \phi (x),
\]
where ${\rm ch}(\rho)$ is the character of $\rho$ and $O(\psi)$ is the coadjoint orbit of $\psi$ under $H_1$. In particular, 
\be\label{e:Depth}
m_{\rho}=m'_{\psi}:=\min\{m\in \bbz^{\ge 0}|\h \psi(\pfr^{m}\hfr)=1\}.
\ee
We also get that $\dim \rho=\sqrt{|O(\psi)|}=\sqrt{[H_1:H_{1,\psi}]}$, where 
$
H_{1,\psi}:=\{h\in H_1|\h (\Ad^{\ast}h)(\psi)=\psi\}.
$
By~\cite[Lemma 1.1]{How}, we have 
$
H_{1,\psi}=\exp \{x\in \hfr|\h \psi([x,\hfr])=1\}.
$
Hence 
\be\label{e:Dimension}
\dim \rho=\sqrt{[\hfr: \hfr({\psi})]},\text{where} \h \hfr({\psi}):=\{x\in \hfr|\h \psi([x,\hfr])=1\}.
\ee
 
{\bf Claim:} (see~\cite[Lemma 3.3]{SGChar}) $\hfr({\psi^p})=\hfr({\psi})$ implies that $\hfr({\psi})=\hfr$.

{\em proof of Claim.} Let $\kappa_{\psi}:\hfr \rightarrow \hom(\hfr,S^1), \kappa_{\psi}(x)(y):=\psi([x,y])$. Then $\kappa_{\psi}$ is a group homomorphism and $\ker \kappa_{\psi}=\hfr({\psi})$ is an open subgroup of $\hfr$. On the other hand, $\kappa_{\psi^p}(x)=\kappa_{\psi}(px)$. Hence 
\[
\kappa_{\psi}(px)=1 \implies  x\in \ker \kappa_{\psi^p}=\hfr(\psi^p)=\hfr(\psi)=\ker \kappa_{\psi} \implies \kappa_{\psi}(x)=1.
\]  
Therefore $\hfr/\hfr({\psi})$ is torsion-free. Since $\hfr({\psi})$ is an open subgroup, we conclude that $\hfr=\hfr({\psi})$.

We also notice that $m'_{\psi^p}=m'_{\psi}-v_{\pfr}(p)$ if $m'_{\psi}> v_{\pfr}(p)$. We also know that $\hfr(\phi)=\hfr$ if and only if $\phi([\hfr,\hfr])=1$. Hence, if $\hfr(\phi)=\hfr$, then $\hfr(\pfr^{s_0}\hfr)=1$, Thus $m_{\phi}'<s_0$. Therefore we have
\[
\hfr(\psi)\subsetneq 
\hfr({\psi^p}) \subsetneq \cdots \subsetneq 
\hfr(\psi^{p^{\lfloor (m'_{\psi}-s_0)/v_{\pfr}(p) \rfloor-1}})\subsetneq \hfr,
\]
which implies $[\hfr:\hfr(\psi)]\geq p^{\lfloor (m'_{\psi}-s_0)/v_{\pfr}(p) \rfloor}$. And so $\dim \rho\geq p^{\lfloor (m_{\rho}-s_0)/v_{\pfr}(p) \rfloor/2}$.
\end{proof}
\begin{prop}\label{p:HighMultiplicity}
Let $k$ be a number field and $S$ be a finite subset of $V_f(k)$. Let $\Gamma\subseteq \GL_n(\ocal_k(S))$ be a finitely generated group whose Zariski-closure $\bbg$ is a Zariski-connected semisimple group. Then there is a finite-index subgroup $\Lambda$ of $\Gamma$ such that for any $\pfr\in V_f(k)\setminus S$, $n\in \bbz^+$, and complex irreducible representation $\rho$ of $\pi_{\pfr^n}(\Gamma)$ which does not factor through $\pi_{\pfr^{n-1}}(\Gamma)$, we have that either 
\begin{enumerate}
\item the restriction of $\rho$ to $\pi_{\pfr^n}(\Lambda)$ is trivial, or
\item $|\f_\pfr|^n\ll_{\Gamma} 1$, or
\item $\dim \rho\ge |\pi_{\pfr^n}(\Gamma)|^{\Theta_{\Gamma}(1)}$.
\end{enumerate}
\end{prop}
\begin{proof}
By the discussion in Section~\ref{ss:StrongApproximation}, by strong approximation, the closure $\Gamma_{\pfr}$ of $\Gamma$ in $\bbg(k_{\pfr})$ is a $p$-adic analytic group. For any $\pfr\in V_f(k)\setminus S$, let $c_0:=2(a_0(\pfr)+b_0(\pfr))$ (where $a_0$ and $b_0$ are described at the beginning of this section), and $\Gamma_{\pfr}[\pfr^{c_0}]:=\ker(\Gamma_{\pfr}\xrightarrow{\pi_{\pfr^{c_0}}} \pi_{\pfr^{c_0}}(\Gamma))$. Then log is well-defined on $\Gamma_{\pfr}[\pfr^{c_0}]$ and its image $\gfr_{\pfr}$ is a Lie algebra. Moreover for some $s_0$ (independent of $\pfr$) we have $\pfr^{s_0} \gfr_{\pfr} \subseteq [\gfr_{\pfr},\gfr_{\pfr}]$ because of strong approximation (see Section~\ref{ss:StrongApproximation}) and the fact that the Zariski-closure $\bbg_1$ of $\Gamma$ in $R_{k/\bbq}(\bbg)$ is a perfect $\bbq$-group.

Let $\wt{S}:=\{\pfr\in V_f(k)\setminus S\h |\h \Gamma_{\pfr} \text{ is not a hyperspecial parahoric subgroup}\}$. Let 
$
\wt{\Lambda}:=\Gamma \cap (\prod_{\pfr \in \wt{S}} \Gamma_{\pfr}[\pfr^{c_0}])
$
and 
$
\Lambda:=\wt{\Lambda} \cap (\prod_{\pfr\in \wt{S}}[\Gamma_{\pfr}[\pfr],\Gamma_{\pfr}[\pfr]]),
$
where $\Gamma$ is embedded diagonally in $\prod_{\pfr \in \wt{S}} \Gamma_{\pfr}$. Since $\wt{S}$ is a finite set~\cite[Section 3.9]{Tit} and $\Gamma_{\pfr}[\pfr^{c_0}]\cap [\Gamma_{\pfr}[\pfr],\Gamma_{\pfr}[\pfr]]$ is of finite-index in $\Gamma_{\pfr}$, $\Lambda$ is a finite index subgroup of $\Gamma$. Let $\wt{\Lambda}_{\pfr}$ be the closure of $\wt{\Lambda}$ in $\Gamma_{\pfr}$. Then we have $\wt{\Lambda}_{\pfr}=\Gamma_{\pfr}$ if $\pfr\in V_f(k)\setminus \wt{S}$, and $\wt{\Lambda}_{\pfr}=\Gamma_{\pfr}[\pfr^{c_0}]$ if $\pfr\in \wt{S}$. 

Now suppose the restriction of $\rho$ to $\wt{\Lambda}$ is not trivial, and $|\f_{\pfr}|^n> \max\{|\f_{\pfr'}|^{c_0+1}|\h \pfr'\in \wt{S}\}$.

If $n>c_0$, then we get the claim by Lemma~\ref{l:DimensionPAdicAnalytic}.

If $n\leq c_0$, then $\pfr\not\in \wt{S}$. So $\Gamma_{\pfr}=\wt{\Lambda}_{\pfr}$ is a hyperspecial parahoric subgroup. In particular,  $\pi_{\pfr}(\Gamma)$ is a product of almost simple finite groups of Lie type. So for $n=1$ we get the claim using \cite{LS} (see \cite[Section 4.2]{SGV}).

If $1<n\leq c_0$, then  $\pi_{\pfr^n}(\Gamma_{\pfr}[\pfr])$ is a finite $p$-group of order at most $p^{\Theta_{\bbg}(1)}$. Hence either the restriction of $\rho$ to $\pi_{\pfr^n}(\Gamma_{\pfr}[\pfr])$ has an irreducible subrepresentation with dimension at least $p$, or all the irreducible subrepresentations of the restriction of $\rho$ to $\pi_{\pfr^n}(\Gamma_{\pfr}[\pfr])$ have dimension 1. In the former case, we are done. In the latter case, the restriction of $\rho$ to $\Lambda$ is trivial. Overall we get the desired result.
\end{proof}
\subsection{Finite logarithmic maps.}\label{ss:FiniteLog}
Here we recall the {\em finite logarithmic maps} and their main properties. These maps are the key connection between congruence groups and the Lie algebra. All of these results are well-known for Chevalley groups, but I could not find them in the literature in the generality needed for this work. All the needed materials from the theory of group schemes can be found in~\cite{DG}. 

Let $A$ be a PID and $F$ be its field of fractions. Let $\bbh\subseteq (\mathbb{SL}_n)_{F}$ be a connected $F$-subgroup and $\cal$ be its Zariski-closure in $(\mathcal{SL}_n)_{A}$; in particular $\cal$ is a flat finite-type $A$-group scheme. 

Let us recall that $\Lie \cal$ is an $A$-group scheme and for any $A$-algebra $R$ we have that 
\be\label{e:ExactLieAlgebra}
1\rightarrow \Lie \cal(R)\rightarrow \cal(R[x]/\langle x^2\rangle) \rightarrow \cal(R)\rightarrow 1
\ee
is a short exact sequence~\cite[II, \textsection 4, 1.2]{DG}. If $\cal$ is a {\em smooth} $A$-group scheme, then $\Lie \cal=\underline{\hfr}$, where $\hfr=\Lie \cal(A)$ and $\underline{\hfr}(R):=\hfr \otimes_{A} R$ for any $A$-algebra $R$ \cite[II, \textsection 4, 4.8]{DG}. On the other hand, since the generic fiber $\bbh$ of $\cal$ is smooth, there is $a_0\in A$ such that $\cal\times_{A} A[1/a_0]$ is a smooth $A[1/a_0]$-group scheme, and in particular $\Lie \cal\times _{A} A[1/a_0]=\underline{\hfr}\times_{A} A[1/a_0]$ (This is due to Grothendieck. For its effective version see \cite[Lemma 45]{SGV}).  

To be more precise, suppose $f_1,\ldots,f_s\in A[X_{11},\ldots,X_{nn}]$ is a set of defining relations of $\cal$, i.e. as an $A$-scheme $\cal=\Spec(A[X_{11},\ldots,X_{nn}]/\langle f_1,\ldots,f_s\rangle)$. Then as an $A$-scheme 
\be\label{e:LieAlgebra}
\Lie \cal=\Spec(A[X_{11},\ldots,X_{nn}]/\langle (df_1)_I,\ldots,(df_s)_I\rangle),
\ee
 where $(df_l)_I(Y_{ij}):=\sum_{ij}(\partial f_l/\partial X_{ij})(I) Y_{ij}$. Looking at the Taylor expansion of $f_i$ we see that for an $A$-algebra $R$ and $Y\in {\rm M}_n(R)$ we have $(df_i)_I(Y)=0$ if $f_i(I+tY)=0$ and $t^2=0$ in $R$. Therefore, for $q_1|q_2|q_1^2$, we get the following maps
\[
\wt{\lin{q_1}{q_2}}:\pi_{q_2}(\cal(A)[q_1])\rightarrow \Lie \cal(A/q_3A), \h\h \wt{\lin{q_1}{q_2}}(\pi_{q_2}(g)):=\pi_{q_3}((g-1)/q_1),
\]
where $q_3=q_2/q_1$, $\pi_q:\cal(A)\rightarrow \cal(A/qA)$, and $\cal(A)[q]:=\ker(\cal(A)\xrightarrow{\pi_q}\cal(A/qA))$. Notice that $\Lie \cal(A/q_3A)$ is {\em not} necessarily isomorphic to $\hfr/q_3 \hfr$, but this is the case when $q_3A$ does not have small prime factors (as it is pointed out above as a consequence of the generic smoothness). In fact, by (\ref{e:LieAlgebra}), $\Lie \cal(R)=\{Y\in {\rm M}_n(R)|\h F(Y)=0\}$ where, $F(Y)=((df_1)_I(Y),\ldots,(df_s)_I(Y))$. So there is  $q_0\in A$ (depending on $f_i$) such that, for any $q\in A$, $\Lie\cal(A/qA)$ can be naturally mapped onto $\hfr/q^-\hfr$, where $q^-:=q/\gcd(q,q_0)$. Hence for $q_1|q_2|q_1^2$ we get the {\em finite logarithmic maps}: 
\be\label{e:FiniteLogMap}
\lin{q_1}{q_1q_3^-}:\pi_{q_1q_3^-}(\cal(A)[q_1])\rightarrow \hfr/q_3^-\hfr, \h\h \linValue{q_1}{q_1q_3^-}{g}:=\pi_{q_3^-}(x), 
\ee
where $x\in \hfr\subseteq \mathfrak{sl}_n(A)$ and $\pi_{q_3^-}(x)=\pi_{q_3^-}((g-1)/q_1)$.

 Let $\Gamma$ be a finitely generated subgroup of $\GL_{n_0}(\ocal_k(S))$. Suppose the Zariski-closure $\bbg$ of $\Gamma$ in $(\GL_{n_0})_k$ is a connected, simply-connected semisimple group. For any $\pfr\in V_f(k)\setminus S$, let $\bbg_1$ and $\bbg_{1,\pfr}$ be as in Lemma~\ref{l:StructurePAdicClosure}. Hence based on the way these groups are defined we have $\Gamma\subseteq \bbg_1(\bbq)\cap \GL_{n_0\dim k}(\bbz_{S_0})$, where $S_0:=\{p\in V_f(\bbq)|\h \exists \pfr\in S, \pfr|p\}$ and $\Gamma\subseteq \bbg_{1,\pfr}(\bbq_p)\cap \GL_{n_0[k:k(\pfr)]}(\bbz_p)$. Let $\gcal_1$ and $\gcal_{1,\pfr}$ be the corresponding Zariski-closures. In particular, they are flat group schemes of finite-type. Furthermore $\gcal_1$ and $\gcal_{1,\pfr}$ are defined over $\bbz_{S_0}$ and $\bbz_p$, respectively. By strong approximation, we have that the closure of $\Gamma$ in $\prod_{p\in V_f(\bbq)\setminus S_0} \gcal_1(\bbz_p)$ is an open subgroup. And for $\pfr\not\in S$, $\Gamma_{\pfr}$ is an open subgroup of $\gcal_{1,\pfr}(\bbz_p)$. Let 
\be\label{e:LieAlgebras}
\gfr_1:=\Lie \gcal(\bbz_{S_0}), \h\h \gfr_{1,\pfr}:=\Lie \gcal_{1,\pfr}(\bbz_p).
\ee
\begin{lem}\label{l:PropertiesOfLinearizationMap}
Let $\Gamma$, $\gcal_1$, $\gcal_{1,\pfr}$, $\gfr_1$, and $\gfr_{1,\pfr}$ be as above. Assume $q_1|q_2|q_1^2$, $q_1'|q_2'|q_1'^2$ and $q_1|q_1'$. Let $q=\gcd(q_2q_1',q_1q_2')$, $q_3:=q_2/q_1$, and $q_3':=q_2'/q_1'$. Then
\begin{enumerate}
\item Suppose $q_3,q_3'$ have large prime factors (depending on $\Gamma$). Let $\Gamma[q_1]:=\Gamma\cap \gcal_1(\bbz_{S_0})[q_1]$, and $\gfr:=\gfr_1$. Then 
	\begin{enumerate}
	\item $	
		\lin{q_1}{q_2}:\pi_{q_2}(\Gamma[q_1])\rightarrow \gfr/q_3\gfr
		$
		 is a well-defined, injective additive homomorphism. 
	\item\label{i:Equi} $\lin{q_1}{q_2}$ is $\Gamma$-equivariant, i.e. for any  $\gamma\in \Gamma$ and $g\in \gcal(\prod_{p\in V_f(\bbq)\setminus S_0} \bbz_p)$, there is $x\in \gfr$ such that $\pi_{q_1q_1^-}(g)=\pi_{q_1q_1^-}(1+q_1x)$ and 
\[	
\linValue{q_1}{q_2}{\gamma^{-1} g\gamma}=\pi_{q_3}(\Ad(\gamma)(x)).
\] 
	 \item If $g\in \gcal_1(\bba_{\bbq}(S_0))[q_1]$ and $g'\in \gcal_1(\bba_{\bbq}(S_0))[q_1']$ where $\bba_{\bbq}(S_0):=\prod_{p\in V_f(\bbq)\setminus S_0} \bbz_p$, then $(g,g'):=g^{-1}g'^{-1}gg'\in \gcal_1(\bba_{\bbq}(S_0))[q_1q_1']$. And 
 \[
 \linValue{q_1q_1'}{q}{(g,g')}=\pi_{q/(q_1q_1')}([\linValue{q_1}{q_2}{g}, \linValue{q_1'}{q_2'}{g'}]),
 \]
 where $[x_1,x_2]:=x_1x_2-x_2x_1$.
	\end{enumerate}
\item Suppose $q_i$ are powers of a single prime $p$ and $\log_p q_1^2/q_2,\log_p q_1'^2/q_2'\gg_{\Gamma} 1$. Let $\Gamma_{\pfr}[q_1]:=\Gamma_{\pfr}\cap \gcal_{1,\pfr}(\bbz_p)[q_1]$, and $\gfr:=\gfr_{1,\pfr}$. Then the same properties for $\lin{q_1}{q_2}$ hold for $\Gamma_{\pfr}[q_1]$ and $\gcal_{1,\pfr}(\bbz_p)$ instead of $\Gamma[q_1]$ and $\gcal_1(\bba_{\bbq}(S_0))$, respectively. Moreover $\lin{q_1}{q_2}$ is an additive group isomorphism; in fact by part (\ref{i:Equi}) it is a $\Gamma$-module isomorphism.
\end{enumerate}
\end{lem}
\begin{proof}
Since $\bbz_{S_0}$ and $\bbz_p$ are PIDs, the above discussion shows  why $\lin{q_1}{q_2}$ is well-defined. And from the definition of $\lin{q_1}{q_2}$ it is an embedding. Now one can show the rest of the claims by direct computation except the last part: $\lin{q_1}{q_2}$ is onto if $q_i$ are powers of a single prime $p$ and $\log_p q_1^2/q_2, q_1'^2/q_2'\gg_{\Gamma} 1$.  

Since $\Gamma_{\pfr}$ is an open subgroup of $\gcal_{1,\pfr}(\bbz_p)$, we have $\Gamma_{\pfr}[p^n]=\gcal_{1,\pfr}(\bbz_p)[p^n]$ for $n\gg_{\Gamma} 1$. On the other hand, the exponential map $\exp: p^n\gfr_{1,\pfr}(\bbz_p)\rightarrow \gcal_{1,\pfr}(\bbz_p)[p^n]$ and the logarithmic map $\log:\gcal_{1,\pfr}(\bbz_p)[p^n]\rightarrow p^n\gfr_{1,\pfr}(\bbz_p)$ are well-defined analytic functions and inverse of each other for $n\gg 1$; in particular $\|\exp(x)-I\|_p=\|x\|_p$ and $\|\log g\|_p=\|g-I\|_p$. Therefore for any $x\in \gfr_{1,\pfr}(\bbz_p)$ and large enough $n$ (depending on $\Gamma$) we have that 
\[
\pi_{p^{2n}}(1+p^n x)=\pi_{p^{2n}}(\exp(p^n x))\in \pi_{p^{2n}}(\gcal_{1,\pfr}(\bbz_p)[p^n])=\pi_{p^{2n}}(\Gamma_{\pfr}[p^n]).
\]
 And so $\lin{q_1}{q_2}$ is onto if $\log_p q_1\gg_{\Gamma}1$.

\end{proof}
\section{Expansion, approximate subgroup, and bounded generation.}
\subsection{Statements and notation.}\label{ss:StatementApproximateSubgroupBoundedGeneration} In this section, three properties will be introduced, and we will explore the connections between them. The first property is about the level-$Q$ approximate subgroups, the second property is about the bounded generation of a large congruence subgroup, and the third property is about the bounded generation of a {\em thick top slice}.   

In this section, let $\Omega$ be a symmetric subset of $\GL_{n_0}(\ocal(S))$ given as in Section~\ref{ss:InitialReductions}, and let $\gcal_1$ and $\gcal_{1,\pfr}$ be as in Lemma~\ref{l:PropertiesOfLinearizationMap}: 

Let $S_0:=\{p\in V_f(\bbq)|\h \exists \pfr\in S, \pfr|p\}$, and view $\Gamma$ as a subgroup of $R_{\ocal_k(S)/\bbz_{S_0}}(\GL_{n_0})(\bbz_{S_0})\subseteq \GL_{n_0\dim k}(\bbz_{S_0})$. Let $\bbg_1$ and  $\gcal_1$ be the Zariski-closure of $\Gamma$ in $(\GL_{n_0\dim k})_{\bbq}$ and $(\GL_{n_0\dim k})_{\bbz_{S_0}}$, respectively. For any $\pfr\in V_f(k)\setminus S$, we consider $\Gamma$ as a subgroup of $R_{\ocal_{\pfr}/\bbz_p}(\GL_{n_0})(\bbz_p)\subseteq \GL_{n_0[k:k(\pfr)]}$, where $k(\pfr)$ is as in Lemma~\ref{l:StructurePAdicClosure}. Let $\bbg_{1,\pfr}$ and $\gcal_{1,\pfr}$ be the Zariski-closure of $\Gamma$ in $(\GL_{n_0[k:k(\pfr)]})_{k(\pfr)}$  and $(\GL_{n_0[k:k(\pfr)]})_{\bbz_p}$.  Let $\gfr_1:=\Lie \gcal(\bbz_{S_0}), \h\h \gfr_{1,\pfr}:=\Lie \gcal_{1,\pfr}(\bbz_p).$

In this section, 
$\pi_{p^n}$ is either the residue map from $\gcal_1(\bbz_p)$ to $\gcal_1(\bbz/p^n \bbz)$ if $p\in V_f(\bbq)\setminus S_0$, or the residue map from $\gcal_{1,\pfr}(\bbz_p)$ to $\gcal_{1,\pfr}(\bbz/p^n\bbz)$ if $p\in S_0$ and $\pfr| p$ for some $\pfr\in V_f(k)\setminus S$. And accordingly we define $\Gamma[p^n]$ to be either $\gcal_1(\bbz_p)[p^n]\cap \Gamma$ if $p\in V_f(\bbq)\setminus S_0$, or $\gcal_{1,\pfr}(\bbz_p)[p^n]\cap \Gamma$ if $p\in S_0$ and $\pfr|p$ where $\pfr\in V_f(k)\setminus S$.

Before formulating the precise statement, let us introduce a notation for convenience: for a finite symmetric subset $A$ of $\Gamma$, a positive integer $l$, a prime power $Q=p^n$, and a positive number $\delta$, let $\Pfr_Q(\delta,A,l)$ be the following statement
\be\label{e:P(delta)}
/(\pcal_{\Omega}^{(l)}(A)>Q^{-\delta})\h\wedge\h (l>\frac{1}{\delta} \log Q)\h \wedge\h (|\pi_Q(A\cdot A\cdot A)|\le |\pi_Q(A)|^{1+\delta}).
\ee
\begin{thm}[Approximate subgroups]\label{ReductionTripleProduct}
In the above setting, for any $\vare>0$ there is $\delta>0$ such that 
\begin{center}
$\Pfr_Q(\delta,A,l)$ implies that $|\pi_Q(A)|\ge |\pi_Q(\Gamma)|^{1-\varepsilon}$
\end{center}
if $Q=p^n$ and $Q^{\vare^{\Theta_{\bbg}(1)}}\gg_{\Omega} 1$. 
\end{thm}
Throughout this article for a subset $X$ of a group we let 
\[\textstyle
\prod_C X:=\{x_1\cdot \cdots\cdot x_C|\h x_i\in X\}.
\]
Theorem~\ref{ReductionTripleProduct} is proved using the following:
\begin{thm}[Bounded generation]\label{ReductionBoundedGeneration}
In the above setting, for any $0<\vare\ll_{\Omega} 1$, $0<\delta\ll_{\Omega,\vare} 1$, and positive integer $C\gg_{\Omega,\vare}1$ the following holds for a finite symmetric subset $A$ of $\Gamma$ which contains $1$:
\begin{center}
$\pcal_{\Omega}^{(l)}(A)>Q^{-\delta}$ for some $l>\log Q/\delta$ implies that $\pi_Q(\Gamma[q])\subseteq \prod_C\pi_Q(A)$ for some $q|Q$ such that $q< Q^{\varepsilon}$
\end{center} 
if $Q=p^n$ and $Q^{\vare^{\Theta_{\bbg}(1)}}\gg_{\Omega} 1$.
\end{thm}  
Theorem~\ref{ReductionBoundedGeneration} is proved based on a propagation process and the following proposition.
\begin{prop}[Thick top slice]\label{p:ThickTopSlice}
In the above setting, for any $0<\varepsilon_2\ll_{\Omega} \varepsilon_1\ll_{\Omega} 1$, there are $0<\delta$ and a positive integer $C=C_{\Omega}(\varepsilon_1,\vare_2)$ with the following property:

\begin{center}
$\Pfr_Q(\delta,A,l)$ implies that there are $X\subseteq \prod_C A$ and $q_1|q_2|Q$ such that
\begin{enumerate}
\item $q_1\le Q^{\varepsilon_1}$ and $Q^{\varepsilon_2}\le q_2/q_1$,
\item $\pi_{q_2}(X)=\pi_{q_2}(\Gamma[q_1])$.
\end{enumerate}
\end{center}
if $Q=p^n$ and $n\vare_2\gg_{\Omega} 1$.
\end{prop}  

\subsection{Theorem~\ref{ReductionTripleProduct} (Approximate subgroup) implies Theorem~\ref{t:main} (Spectral gap).}\label{ss:ReductionApproximateSubgroup} We can assume that $\Omega$ satisfies the extra assumptions listed in Section~\ref{ss:InitialReductions}. By the discussion in Section~\ref{ss:StrongApproximation}, we know that $\Gamma_{\pfr}$ is a quotient of the closure $\Gamma_p$ of $\Gamma$ in $\gcal_1(\bbz_p)$ if $p\in V_f(\bbq)\setminus S_0$. Hence $\lambda_1(\pcal_{\Omega};\Gamma_{\pfr})\le \lambda_1(\pcal_{\Omega};\Gamma_p)$. For $\pfr\in V_f(k)\setminus S$ which divides $p\in S_0$, we notice that $\Gamma_{\pfr}[q]:=\gcal_{1,\pfr}(\bbz_p)[q]\cap \Gamma_{\pfr}$ is a neighborhood basis for the identity. So altogether it is enough to prove:
\[
\sup\{\lambda_1(\pcal_{\pi_Q(\Omega)}; \pi_Q(\Gamma))|\h Q=p^n; \exists \pfr\in V_f(k)\setminus S, \pfr|p\}<1, 
\]
where $\pi_Q$ is the residue map defined at the beginning of Section~\ref{ss:StatementApproximateSubgroupBoundedGeneration}.  
 
Let $\pcal=\pcal_{\Omega}$ and let $T_Q:l^2(\pi_Q(\Gamma))\rightarrow l^2(\pi_Q(\Gamma))$, $T_Q(f):=\pi_Q[\pcal]*f$. Let $\lambda_1(Q)$ be the second largest eigenvalue. We have to prove that there is a uniform upper bound for $\lambda_1(Q)$. Since $l^2(\pi_Q(\Gamma))$ is a completely reducible $\pi_Q(\Gamma)$-space, there is a unit eigenfunction $f\in l^2(\pi_Q(\Gamma))$ such that $T_Q(f)=\lambda_1(Q) f$ and the $\pi_Q(\Gamma)$-module $V$ generated by $f$ is irreducible. Changing $Q$ to one of its divisors, if necessary, we can and will assume that the action of $\pi_Q(\Gamma)$ on $V$ does not factor through $\pi_{q}(\Gamma)$ for any proper divisor $q$ of $Q$. Now by Proposition~\ref{p:HighMultiplicity} there is $\varepsilon_0>0$ such that $\dim V\geq |\pi_Q(\Gamma)|^{\varepsilon_0}$. So the multiplicity of $\lambda_1(Q)$ is at least $|\pi_q(\Gamma)|^{\varepsilon_0}$. Hence 
\be
|\pi_Q(\Gamma)|^{\varepsilon_0} \lambda_1(Q)^{2l}\le \Tr(T_Q^{2l})=|\pi_Q(\Gamma)|\h \|\pi_Q(\pcal^{(l)})\|_2^2
\ee
So it is enough to prove that $\|\pi_Q(\pcal^{(l)})\|_2\le |\pi_Q(\Gamma)|^{-\frac{1}{2}+\frac{\varepsilon_0}{4}}$ for some $l\ll \log Q$ where the implied constant just depends on $\Omega$.
\begin{lem}\label{l:flattening}
Let $\Omega$ and $\vare_0$ be as above. Then there is $\delta>0$ with the following property: 

let $Q=p^n$, where some $\pfr\in V_f(k)\setminus S$ divides $p$, and $l_0\gg_{\Omega} \log Q$.
\begin{enumerate}
\item $\|\pi_Q[\pcal^{(l_0)}]\|_2\le |\pi_Q(\Gamma)|^{-\delta}$,
\item If $l\ge \frac{1}{\delta}\log Q$ and $\|\pi_Q[\pcal^{(l)}]\|_2\ge |\pi_Q(\Gamma)|^{-\frac{1}{2}+\frac{\varepsilon_0}{4}}$, then $\|\pi_Q[\pcal^{(2l)}]\|_2\le \|\pi_Q[\pcal^{(l)}]\|_2^{1+\delta}$.
\end{enumerate}   
\end{lem}
Repeated use of Lemma~\ref{l:flattening}, in finitely many steps (the number of steps is independent of $Q$), would give us $l\ll \log Q$ such that $\|\pi_Q[\pcal^{(l)}]\|_2\le |\pi_Q(\Gamma)|^{-\frac{1}{2}+\frac{\varepsilon_0}{4}}$ as we desired. So in order to complete this step, it is enough to prove Lemma~\ref{l:flattening}.
\begin{proof}[Proof of Lemma~\ref{l:flattening}]
First we notice that, since ${\rm Cay}(\Gamma,\Omega)$ is a regular tree and $\pi_Q$ induces a bijection on $\prod_{\Theta_{\Omega}(\log Q)}\Omega$, by Kesten's bound, there is $l_0\gg_{\Omega} \log Q$ and $\delta_1>0$ such that $\|\pi_Q[\pcal^{(l)}]\|_2\le |\pi_Q(\Gamma)|^{-\delta_1}$ for $l\ge l_0$.

Now to get the second part, we proceed by contradiction. So assume for small enough $\delta>0$ (specified later), we have  $\|\mu\|_2^2\ge |\pi_{Q_{\delta}}(\Gamma)|^{-1+\vare_0/2}$ and  $\|\mu*\mu\|_2\ge \|\mu\|_2^{1+\delta}$,
where $\mu=\pi_{Q_{\delta}}[\pcal^{(l_{\delta})}]$ for some $l_{\delta}\ge \frac{1}{\delta}\log Q_{\delta}$. Thus by Lemma~\ref{l:BSG} there is a subset $\overline{A}_{\delta}$ of $\pi_{Q_\delta}(\Gamma)$ with the following properties:
\begin{enumerate}
	\item $\|\mu\|_2^{-2+R\delta}\le |\overline{A}_{\delta}|\le \|\mu\|_2^{-2-R\delta}$,
	\item $|\prod_3 \overline{A}_{\delta}|\le |\overline{A}_{\delta}|\cdot \|\mu\|_2^{-R\delta}$,
	\item For any $\overline{h}\in \overline{A}_{\delta}$, $(\mu*\mu)(\overline{h})\ge \frac{\|\mu\|_2^{R\delta}}{|\overline{A}_{\delta}|}$.
\end{enumerate}
These conditions imply that there is $\delta':=\delta'(\delta)$ with the following properties:
\begin{enumerate}
	\item $\lim_{\delta\rightarrow 0^+}\delta'(\delta)=0$ and $\delta'(\delta)\ge \delta$,
	\item $|\prod_3 \overline{A}_{\delta}|\le |\overline{A}_{\delta}|^{1+\delta'}$,
	\item $\mu^{(2)}(\overline{A}_{\delta})\ge Q_\delta^{-\delta'}$.
\end{enumerate}
Now let $A_{\delta}=\pi_{Q_\delta}^{-1}(\overline{A}_{\delta})\cap \supp(\pcal^{(2l_{\delta})})$. By the definition, $\pcal^{(2l_{\delta})}(A_{\delta})=\pi_{Q_\delta}[\pcal^{(2l_{\delta})}](\overline{A}_{\delta})$ and $\pi_{Q_\delta}(A_{\delta})=\overline{A}_{\delta}$ (equality holds as $\pi_Q[\pcal^{(2l_{\delta})}](\overline{h})>0$ for any $\overline{h}\in \overline{A}_{\delta}$); moreover since $\pcal$ and $\overline{A}_{\delta}$ are symmetric, $A_{\delta}$ is also symmetric. Hence for $0<\delta\ll 1$, we end up with a symmetric finite subset $A_{\delta}$ of $\Gamma$ with the following properties:
\begin{enumerate}
	\item $\pcal^{(2l_{\delta})}(A_{\delta})\ge Q_{\delta}^{-\delta'}$ and $2l_{\delta}>\frac{1}{\delta'}\log Q_{\delta}$, 
	\item $|\pi_{Q_{\delta}}(A_{\delta}\cdot A_{\delta}\cdot A_{\delta})|\le |\pi_{Q_{\delta}}(A_{\delta})|^{1+\delta'}$,
	\item $|\pi_{Q_{\delta}}(A_{\delta})|\le |\pi_{Q_{\delta}}(\Gamma)|^{1-\vare_0/4}$.
\end{enumerate} 
Since $l_{\delta}>\log Q_{\delta}/\delta$, we have $\lim_{\delta\to 0} l_{\delta}=\infty$. Now we claim that $Q_{\delta}\rightarrow \infty$ as $\delta\rightarrow 0$. If not, for infinitely many $\delta$ we have $Q_{\delta}=Q$, and $\pi_Q(A_{\delta})=\overline{A}$ is independent of $\delta$. And so
\[
\lim_{\delta\rightarrow 0} \pi_Q[\pcal]^{(l_{\delta})}(\pi_Q(A_{\delta}))=\frac{|\overline{A}|}{|\pi_Q(\Gamma)|}<\frac{1}{|\pi_Q(\Gamma)|^{\vare_0/4}}.
\]
On the other hand, we have
\[
\pi_Q[\pcal]^{(l_{\delta})}(\pi_Q(A_{\delta}))\ge \pcal^{(l_{\delta})}(A_{\delta})> Q_{\delta}^{-\delta}=Q^{-\delta}.
\]
Hence as $\delta\rightarrow 0$, we get $|\pi_Q(\Gamma)|^{-\vare_0/4}>1$, which is a contradiction. Since $Q_{\delta}$ gets arbitrarily large, Theorem~\ref{ReductionTripleProduct} gives us the desired contradiction.
\end{proof}

\subsection{Theorem~\ref{ReductionBoundedGeneration} (Bounded generation) implies Theorem~\ref{ReductionTripleProduct} (Approximate subgroup).}

By the contrary assumption, there is $\varepsilon_0>0$ such that for any $\delta>0$ there are a finite symmetric subset $A_{\delta}$, a positive integer $l_{\delta}$ and ${Q_{\delta}}\in \{p^n|\h \exists \pfr\in V_f(k)\setminus S, \pfr|p\}$ such that $\Pfr_{Q_{\delta}}(\delta,A_{\delta},l_{\delta})$ holds and at the same time $|\pi_{Q_{\delta}}(A_{\delta})|< |\pi_{Q_{\delta}}(\Gamma)|^{1-\varepsilon_0}$. 

By a similar argument as in the end of Section~\ref{ss:ReductionApproximateSubgroup}, we have $Q_{\delta}\rightarrow \infty$ as $\delta\rightarrow 0$. 

By Theorem~\ref{ReductionBoundedGeneration}, for any $0<\vare'\ll 1$ there is $C=C_{\Omega}(\vare')$ such that for any $0<\delta\ll_{\vare',\Omega} 1$ 
\be\label{e:CongruenceSubgroup}
\pi_{Q_{\delta}}(\Gamma[q_{\delta}])\subseteq \textstyle\prod_C\pi_{Q_{\delta}}(A_{\delta}),
\ee
for some $q_{\delta}|Q_{\delta}$ such that $q_{\delta}<Q_{\delta}^{\vare'}$. Hence
\be\label{e:LowerBound}
|\pi_{Q_{\delta}}(\textstyle \prod_C A_{\delta})|\ge |\pi_{Q_{\delta}}(\Gamma)|^{1-\Theta_{\Omega}(\vare')}, 
\ee
for small enough $\delta$ (depending on $\vare'$) as $\lim_{\delta\to 0}Q_{\delta}=\infty$. On the other hand, since $|\pi_{Q_{\delta}}(\prod_3A_{\delta})|\le |\pi_{Q_{\delta}}(A_{\delta})|^{1+\delta}$ by the Ruzsa inequality (see~\cite{Hel1}), we have
\be\label{e:UpperBound}
|\pi_{Q_{\delta}}(\textstyle\prod_C A_{\delta})|\le |\pi_{Q_{\delta}}(A_{\delta})|^{1+(C-2)\delta}\le |\pi_{Q_{\delta}}(\Gamma)|^{(1-\vare_0)(1+(C-2)\delta)}.
\ee
By (\ref{e:LowerBound}) and (\ref{e:UpperBound}), for any $\vare'$ and small enough $\delta$ (in particular it can approach zero), we have
\[
1-\Theta_{\Omega}(\vare')\le (1-\vare_0)(1+(C(\vare')-2)\delta),
\] 
which is a contradiction.

\subsection{Proposition~\ref{p:ThickTopSlice} (Thick top slice) implies Theorem~\ref{ReductionBoundedGeneration} (Bounded generation).}
First we discuss that changing $A$ to $\prod_{O_{\Omega}(1)}A$, if needed, we can and will assume $\Pfr(\delta,A,l)$ holds. Next we deal with a {\em given large} $N$. Then we finish the proof using Proposition~\ref{p:ThickTopSlice} and a propagation process based on taking commutators (a similar approach is used in~\cite{BG2,BG3}).

\begin{lem}\label{l:BoundedGenerationOfApproximateSubgroup}
Let $\Omega$ and $Q$ be as in Proposition~\ref{ReductionBoundedGeneration}. For any $0<\delta_0\ll_{\Omega} 1$ and a finite symmetric subset $A$ of $\Gamma$ which contains $1$ the following holds
\begin{center}
$\pcal^{(l)}(A)>Q^{-\delta_0}$ for some positive integer $l>\frac{1}{\delta_0}\log Q$ implies $\Pfr(\delta_0,\prod_{O_{\Omega}(1)}A,l)$ holds.
\end{center}
\end{lem}
\begin{proof}
For small enough $c:=c(\Omega)$, $\pi_Q$ induces an injection from $B_{\lceil c \log Q\rceil}$ into $\pi_Q(\Gamma)$. In particular, $|\pi_Q(A\cdot A)|\ge |A\cdot A\cap B_{\lceil c \log Q\rceil}|$. By Lemma~\ref{l:BGRemark}, we have $\pcal^{({\lceil c \log Q\rceil})}(A\cdot A)\ge Q^{-2\delta_0}$ as ${\lceil c \log Q\rceil}/2< \frac{1}{\delta_0}\log Q< l$. Hence by Kesten's bound we have
\[
Q^{-2\delta_0}\le \pcal^{({\lceil c \log Q\rceil})}(A\cdot A)\le |A\cdot A\cap B_{\lceil c \log Q\rceil}| Q^{-\Theta_{\Omega}(1)}.
\]
Therefore $|\pi_Q(A\cdot A)|\ge Q^{\Theta_{\Omega}(1)-2\delta_0}=Q^{\Theta_{\Omega}(1)}$ for $0<\delta_0\ll_{\Omega}1$. Hence for some $C':=C'(\Omega)$ we have $|\pi_Q(\prod_{C'} A)\cdot \pi_Q(\prod_{C'} A)|\le \pi_Q(\prod_{C'} A)|^{1+\delta_0}$. Since $1\in A$, $\prod_{C'}A\supseteq A$. And so $\pcal^{(l)}(\prod_{C'}A)\ge \pcal^{(l)}(A)>Q^{-\delta_0}$, which implies that $\Pfr(\delta_0,\prod_{C'}A,l)$ holds.
\end{proof}
  
\begin{lem}\label{l:BoundedPower}
Let $\Omega$ be as in Section~\ref{ss:InitialReductions}. Then for any positive integer $N$ there are $\delta>0$ and a positive integer $C$ which depends on $\Omega$ and $N$ such that 
\begin{center}
$\Pfr_Q(\delta,A,l)$ implies that $\pi_Q(\Gamma)=\prod_C \pi_Q(A)$ 
\end{center}
if $Q=p^N$ and $p\gg_{\Omega} 1$. 
\end{lem}
\begin{proof}
Let's recall that $\gcal_1$ is the Zariski-closure of $\Gamma$ in $(\GL_{n_0\dim k})_{\bbz_{S_0}}$  (see Lemma~\ref{l:PropertiesOfLinearizationMap}). Based on the discussion in Section~\ref{ss:StrongApproximation}, for $p\gg_{\Omega} 1$, we have that $\Gamma$ is dense in $\gcal_1(\bbz_p)$. By Theorem~\ref{t:ExpanderSquareFree}, Lemma~\ref{l:EquidistributionErrorRate}, and $l>\frac{N}{\delta}\log p$, we have that
\[
\left|\pi_p[\pcal]^{(l)}(\pi_p(A))-\frac{|\pi_p(A)|}{|\pi_p(\Gamma)|}\right|\le \frac{1}{|\pi_p(\Gamma)|}
\]
for small enough $\delta$. And so we have 
\[
p^{-N\delta}=Q^{-\delta}\le \pi_p[\pcal]^{(l)}(\pi_p(A))\le \frac{|\pi_p(A)|+1}{|\pi_p(\Gamma)|},
\]
which implies that $|\pi_p(A)|\ge |\pi_p(\Gamma)|^{1-\Theta_{\Omega,N}(\delta)}$. On the other hand, for $p\gg_{\Omega} 1$, $\pi_p(\Gamma)=\gcal_1(\bbz/p\bbz)$ is a product of almost simple groups and so it is a quasi-random group (see \cite[Corollary 14]{SGV} and \cite{LS}). Hence by a result of Gowers~\cite{Gow} (see \cite{NP}), for small enough $\delta$, we have 
\be\label{e:Modp}
\pi_p(A\cdot A\cdot A)=\pi_p(\Gamma).
\ee
If $p$ is large enough, $\Gamma$ is dense in $\gcal_1(\bbz_p)$, $\gcal_1(\bbz_p)$ is a hyperspecial parahoric subgroup~\cite[3.8, 3.9.1]{Tit},
\be\label{e:ExactSequence}
1\rightarrow \Lie(\gcal_1)(\bbz/p\bbz) \rightarrow \gcal_1(\bbz/p^2\bbz) \xrightarrow{\pi_p} \gcal_1(\bbz/p\bbz) \rightarrow 1
\ee
is a short exact sequence, and $\Lie \gcal_1(\bbz/p\bbz)=\underline{\gfr}_1(\bbz/p\bbz)$. Moreover $\underline{\gfr}_1(\bbz/p\bbz)=\oplus_j \gfr_j$ where $\gfr_j$'s are simple Lie algebras and simple $\gcal_1(\bbz/p\bbz)$-modules under the adjoint representation. The last assertion is a consequence of the fact $\gcal_1(\bbz_p)$ is a hyperspecial parahoric subgroup.

By (\ref{e:Modp}), there is a section $\psi:\gcal_1(\bbz/p\bbz)\rightarrow \prod_3 A\subseteq \Gamma$ of $\pi_p:\Gamma\rightarrow \gcal_1(\bbz/p\bbz)$. Let $\psi_p:=\pi_{p^2}\circ \psi:\gcal_1(\bbz/p\bbz)\rightarrow \gcal_1(\bbz/p^2\bbz)$.

{\bf Claim.} For any $j$, the projection of ${\rm Im}(\psi_p)\cdot {\rm Im}(\psi_p) \cdot {\rm Im}(\psi_p)^{-1}\cap \underline{\gfr}_1(\bbz/p\bbz)$ to $\gfr_j$ is non-trivial. 

{\em Proof of Claim.} For any $x,y\in \gcal_1(\bbz/p\bbz)$, $\psi_p(x)\psi_p(y)\psi_p(xy)^{-1}\in \ker \pi_p=\oplus_j \gfr_j$. Now suppose to the contrary that the projection of ${\rm Im}(\psi_p)\cdot {\rm Im}(\psi_p) \cdot {\rm Im}(\psi_p)^{-1}\cap \underline{\gfr}_1(\bbz/p\bbz)$ to $\gfr_{j_0}$ is zero. Let's consider
\[
\phi_p: \gcal_1(\bbz/p\bbz)\rightarrow \gcal_1(\bbz/p^2\bbz)/(\oplus_{j\neq j_0} \gfr_j), \phi_p(x):=\psi_p(x) (\oplus_{j\neq j_0} \gfr_j).
\]
By the contrary assumption, $\phi_p$ is a group homomorphism. And so it is a group embedding. On the other hand, by~\cite[Theorem 7.2]{Wei}, $\langle {\rm Im}(\psi_p) \rangle=\pi_{p^2}(\langle {\rm Im}(\psi)\rangle)=\gcal_1(\bbz/p^2\bbz)$ if $p$ is large enough. Since $\phi_p$ is a group homomorphism, ${\rm Im}(\phi_p)={\rm Pr}(\langle {\rm Im}(\psi_p) \rangle)$, where ${\rm Pr}:\gcal_1(\bbz/p^2\bbz)\rightarrow \gcal_1(\bbz/p^2\bbz)/(\oplus_{j\neq j_0} \gfr_j)$. Therefore $\phi_p$ is onto, which contradicts the fact that $|\gcal_1(\bbz/p\bbz)|<|\gcal_1(\bbz/p^2\bbz)/(\oplus_{j\neq j_0} \gfr_j)|$.  

The above claim implies that $\pi_{p^2}(\prod_9 A)\cap \underline{\gfr}_1(\bbz/p\bbz)$ generates $\underline{\gfr}_1(\bbz/p\bbz)$ as a $\gcal_1(\bbz/p\bbz)$-module. Hence by \cite[Corollary 31]{SGV} we have that 
\be\label{e:ModP2}
\textstyle \pi_{p^2}(\prod_{O_{\dim \bbg}(1)} A)=\pi_{p^2}(\Gamma),
\ee
where $\bbg$ is the generic fiber of $\gcal_1$. 

For large enough $p$ and any positive integer $i$, $\gcal_1(\bbz_p)[p^i]/\gcal_1(\bbz_p)[p^{i+1}]\simeq \underline{\gfr}_1(\bbz/p\bbz)$ and $[\underline{\gfr}_1(\bbz/p\bbz),\underline{\gfr}_1(\bbz/p\bbz)]=\underline{\gfr}_1(\bbz/p\bbz)$. Hence by Lemma~\ref{l:PropertiesOfLinearizationMap} and Equation (\ref{e:ModP2}), we get the desired result. 
\end{proof}

To understand the propagation process, the following definition is helpful. 
\begin{definition}\label{d:ThickLayer}
Let $\Omega$, $A$ and $\vare$ be as above. For positive real numbers $L$ (level) and $T$ (thickness), we say that $\bcal(L,T):=\bcal_{Q,\vare}(L,T)$ (bounded generation at the level $L$ and of thickness $T$) holds if for some divisors $q$ and $q'$ of $Q$ we have $q\le L$, $LT\le qq'$ and
\[
\pi_{qq'}(\Gamma[q])\subseteq \pi_{qq'}(\textstyle\prod_{O(1)} A),
\]
where the implied constant depends only on $\Omega$ and $\varepsilon$.
\end{definition}
To understand Definition~\ref{d:ThickLayer} better, it is helpful to think about the rooted tree $\tcal(p,n)$ attached to $\pi_{p^n}(\Gamma)$: 


\begin{itemize}
\item Vertices of the $m^{\rm th}$-layer of $\tcal(p,n)$ are elements of $\pi_{p^m}(\Gamma)$.
\item  For any $m$ and $\gamma\in \Gamma$, the ``parent" vertex of $\pi_{p^m}(\gamma)$ is $\pi_{p^{m-1}}(\gamma)$. 
\end{itemize}

Any subset $X$ of $\Gamma$ gives us a subtree $\tcal(X;p,n)$ of $\tcal(p,n)$ whose $m^{\rm th}$-layer vertices are elements of $\pi_{p^m}(X)$. 

The combinatorial interpretation of $\bcal(L,T)$ is that all the ``descendants" of $\pi_{p^m}(1)$ for at least $k$ ``generations" are in $\tcal(\prod_{O_{\Omega,\vare}(1)}A);p,n)$, where $m\le \log_p L$ ({\em level}) and $m+k\ge \log_p L+\log_p T$ (which gives us a {\em thickness}) (see Figure~\ref{fig:BLT} for a schematic picture). 

We present a propagation process which implies that if $\bcal(Q^{\varepsilon_1},Q^{\varepsilon_2})$ holds, then $\bcal(Q^{c\varepsilon_1},Q^{1-c\vare_1})$ holds for a constant $c$ which depends on $\varepsilon_1/\vare_2$ (see Figure~\ref{fig:ThickLayerToBG}). To get such a process, we prove the following lemma.

\begin{figure}[h]
\centering
\begin{minipage}[scale=.3]{.3\textwidth}
  \centering
  \includegraphics{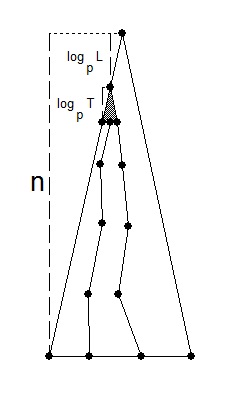}
  \vspace{-20pt}
  \caption{$\bcal(L,T)$}
  \label{fig:BLT}
\end{minipage}%
\begin{minipage}{.7\textwidth}
  \centering
  \includegraphics[scale=.75]{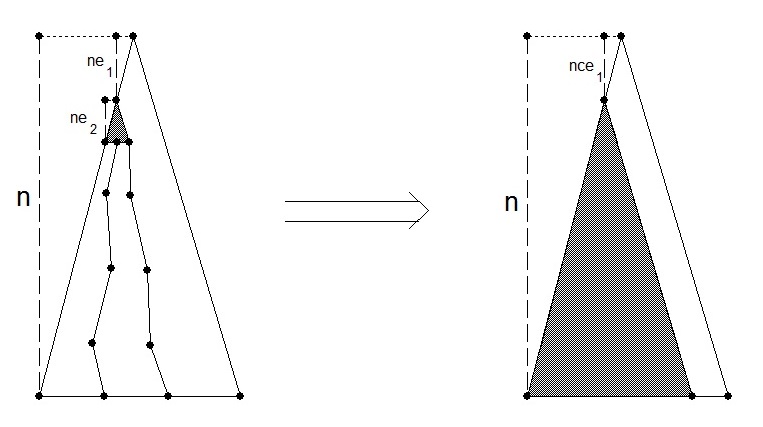}
  \caption{Having a thick top layer implies bounded generation.}
  \label{fig:ThickLayerToBG}
\end{minipage}
\end{figure}

\begin{lem}\label{l:OneStepPropagation}
There is a constant $q_0$ which depends on $\Omega$ with the following property: if $\bcal(L,T)$ and $\bcal(L',q_0p)$ hold, $L,L'\ge T$ and $\log_p T\gg_{\Omega} 1$, then $\bcal(q_0LL',T/q_0)$ holds. 
\end{lem}
\begin{cor}\label{c:OneStepPropagation}
Let $c$ be a fixed positive number, and $q_0$ as in Lemma~\ref{l:OneStepPropagation}. Then, if $L,L'\geq T\ge Q^{c\vare}$, $c\vare\log_p Q\gg_{\Omega} 1$, and $\bcal(L,T)$ and $\bcal(L',T')$ hold, then $\bcal(q_0LL',TT'/pq_0^2)$ holds.
\end{cor}
\begin{proof}
Since $\bcal(L',T')$ holds, $\bcal(L'(T/(q_0p))^i,q_0p)$ holds for any $0\le i <\lfloor \log T'/\log(T/q_0p) \rfloor <1/(c\vare)$. Using Lemma~\ref{l:OneStepPropagation} for $\bcal(L,T)$ and $\bcal(L'(T/(q_0p))^i,q_0p)$, we conclude that $\bcal(q_0LL'(T/(q_0p))^i, T/q_0)$ holds for any $0\le i <1/(c\vare)$. This implies that $\bcal(q_0LL',TT'/pq_0^2)$ holds.  
\end{proof}
\begin{cor}\label{c:CompletionPropagation}
Let $c$ be a fixed positive number, and $q_0$ as in Lemma~\ref{l:OneStepPropagation}. Then, if $L'/q_0p\ge L\ge Q^{c\varepsilon}$, $c\vare\log_p Q\gg_{\Omega} 1$, and $T\ge q_0^2p$ and $\bcal(L,T)$ and $\bcal(L',q_0L)$ hold, then $\bcal(L',Q/L')$ holds.
\end{cor}
\begin{proof}
By induction on $i$ and Corollary~\ref{c:OneStepPropagation} we can show that $\bcal((q_0L)^iL',q_0L)$ for any $i\ll_{\Omega,\vare}1$. This easily implies the claim.  
\end{proof}
Now we have enough to prove Theorem~\ref{ReductionBoundedGeneration}.
\begin{proof}[Proof of Theorem~\ref{ReductionBoundedGeneration}]
Let $c_1$ and $c_2$ be two positive constants such that for any $0<\vare<1$, $\vare_1:=c_1\vare$ and $\vare_2:=c_2\vare$ satisfy the required inequalities of Proposition~\ref{p:ThickTopSlice}.

By assumption, $Q^{\vare^{\Theta(1)}}\gg_{\Omega} 1$. So by choosing the power of $\vare$ larger than $2$, we can assume that $n\vare^2 \log p\ge  C_1$ where $C_1$ is a large enough number depending on $\Omega$ to be specified later. Let's assume that $\vare\le c_2^2$ and $C_1=(1/c_2)^2$. We will determine how small $c_2$ should be. Hence $(c_2 \vare n)(c_2\log p)\ge 1$. Let $C_2:=C_2(\Omega)$ be the maximum of the implied constants in Proposition~\ref{p:ThickTopSlice}, Corollary~\ref{c:OneStepPropagation}, Corollary~\ref{c:CompletionPropagation}, and $d_0:=2v_p(q_0)+1$, where $q_0$ is the constant given in Lemma~\ref{l:OneStepPropagation}. And let $C_3:=C_3(\Omega)$ be the implied constant in Lemma~\ref{l:BoundedPower}.

If $c_2\vare n<2C_2$, then $\log p\ge\frac{1}{2c_2C_2}$ and $n<2C_2/c_2$. Let $c_2= 1/(2C_2\log C_3)$. So, if $c_2\vare n<2C_2$, then $p\ge C_3$. And so by Lemma~\ref{l:BoundedPower} we are done.

Hence we can and will assume that $c_2\vare n\ge 2C_2$. Hence by Proposition~\ref{p:ThickTopSlice} there are $0<\delta$, $C:=C(c_1\vare,c_2\vare)$, and $q_1,q_2|Q$ such that 
\begin{enumerate}
\item $q_1\le Q^{\varepsilon_1}$ and $Q^{\varepsilon_2}\le q_2/q_1$,
\item $\pi_{q_2}(\Gamma[q_1])\subseteq \pi_{q_2}(\prod_C A)$.
\end{enumerate}

 We also notice that we can assume $q_1|q_2|q_1^2$ (by increasing $q_1$ or decreasing $q_2$ if 
necessary). So we can assume that $\bcal(Q^{c_1\varepsilon},Q^{c_2\varepsilon})$ holds.  Let $a_0:=nc_1\vare$, $b_0:=nc_2\vare$, and 
\[
a_m:=2 a_{m-1}+d_0,\h b_m:=2 b_{m-1}-d_0,
\]
for $m\leq \log_2 \lceil 2c_1/c_2 \rceil$ (recall that $d_0:=2v_p(q_0)+1$, where $q_0$ is the constant given in Lemma~\ref{l:OneStepPropagation}). Applying Corollary~\ref{c:OneStepPropagation} (for $L=L'=p^{a_i}$ and $T=T'=p^{b_i}$), inductively we get that $\bcal(p^{a_m},p^{b_m})$ holds for $m\leq \log_2 \lceil 2c_1/c_2 \rceil$. Notice that 
\be\label{e:ab}
a_m=(a_0+d_0)2^m-d_0, b_m=(b_0-d_0)2^m+d_0.
\ee 
 In particular, since $c_2n\vare\ge C_2$, we have that $b_m\ge b_0\ge 2C_2$, and therefore we {\em are} allowed to apply Corollary~\ref{c:OneStepPropagation}. Moreover by Equation(\ref{e:ab}) we have $b_{m_0}> a_0+d_0$ for $m_0:= \log_2 \lceil 2c_1/c_2 \rceil$.
 
Hence $\bcal(Q^{c\varepsilon},q_0Q^{c_1\varepsilon})$ 
holds for some constant $c=c(c_1,c_2)$. Thus, by Corollary~\ref{c:CompletionPropagation}, $\bcal(Q^{c\varepsilon},Q^{1-c\varepsilon})$ holds as $\bcal(Q^{c_1\varepsilon},Q^{c_2\varepsilon})$ and $\bcal(Q^{c\varepsilon},q_0Q^{c_1\varepsilon})$ hold.
\end{proof}

So it is enough to prove Lemma \ref{l:OneStepPropagation}. As in the previous related works (see \cite{BG2, BG3, BV, Din, GS}), we use the connection between congruence subgroups and the Lie algebra. 

\begin{proof}[Proof of Lemma~\ref{l:OneStepPropagation}]
Finite logarithmic maps are key in this proof (see Lemma~\ref{l:PropertiesOfLinearizationMap}). The definition of these maps rely on fixing an embedding of $\Gamma$ into either $\GL_{n_0\dim k}(\bbz_{S_0})$ or $\GL_{n_0\dim k(\pfr)}(\bbz_p)$. Now let $\gcal_1$ and $\gcal_{1,\pfr}$ be as in Lemma~\ref{l:PropertiesOfLinearizationMap}, and $\gfr_1:=\Lie(\gcal_1)(\bbz_{S_0})$ and $\gfr_{1,\pfr}:=\Lie(\gcal_{1,\pfr})(\bbz_p)$. 

Since we treat all the primes at the same time, let
\[
\gfr:=\begin{cases}
		\underline{\gfr}_1(\bbz_p)& \mbox{ if } p\in V_f(\bbq)\setminus S_0,\\
		\underline{\gfr}_{1,\pfr}(\bbz_p) & \mbox{ if } \exists \pfr\in V_f(k)\setminus S, \pfr|p.
	 \end{cases}
\]
We also note that $\gfr\subseteq \gl_n(\bbz_p)$ where $n$ is either $n_0\dim k$ or $n_0\dim k(\pfr)$. Since $\bbz_{S_0}$ and $\bbz_p$ are PID, there are positive integers $m_1,\ldots, m_d$ with the following properties: for any prime $p$ which is divisible by some $\pfr\in V_f(k)\setminus S$ we have that $\gl_n(\bbz_p)$ has a $\bbz_p$-basis $\{\tilde{x}_1,\ldots,\tilde{x}_{n^2}\}$ such that $\{x_i=m_{j_i} \tilde{x}_i\}_{i=1}^{{\rm rank}_{\bbz_p} \gfr}$ forms a $\bbz_p$-basis of $\gfr$. 

Since $\gfr$ is perfect and $\gfr_1$ is a $\bbz_{S_0}$-Lie algebra, there is a positive integer $q_0$ such that $q_0 \gfr\subseteq [\gfr,\gfr]$. We also assume that $q_0$ is a common multiple of $m_i$'s. 

We claim that $q_0$ has the desired property. By strong approximation, we can choose the implied constant in $\log_p L'\ge \log_p T\gg 1$ large enough so that, for some $q'\le L'$, $\lin{q'}{q'q_0p}$ is  an isomorphism. Since $\bcal(L',q_0p)$ holds, there are $\gamma_i\in \prod_{O_{\Omega,\vare}(1)} A$ such that $\linValue{q'}{q'q_0p}{\gamma_i}=\pi_{q_0p}(x_i)$ for any $i$. On the other hand, since $\bcal(L,T)$ holds, there are $q_1\le L$, $q_2\ge LT$ and a subset $Y$ of $\prod_{O_{\Omega,\varepsilon}(1)}A$ such that $\linValue{q_1}{q_2}{Y}=\pi_{q_3}(\gfr)$, where $q_3=q_2/q_1$ (since $\log_p T\gg_{\Omega} 1$, we can talk about $\linValue{q_1}{q_2}$). Thus by Lemma~\ref{l:PropertiesOfLinearizationMap}, parts (2-a) and (2-c), we have that 
\begin{align*}
\linValue{q_1q'}{q_2q'}{(\textstyle\prod_{O_{\Omega,\varepsilon}(1)} A)\cap \Gamma[q_1q']}&\supseteq \sum_i \linValue{q_1q'}{q_2q'}{(\gamma_i,Y)}\\
&\supseteq \sum_i [\linValue{q'}{q'q_3}{\gamma_i},\linValue{q_1}{q_2}{Y}].
\end{align*}
So there are $y_i$'s in $\gfr$ such that $\pi_{q_0p}(y_i)=\pi_{q_0p}(x_i)$ for any $i$ and 
\be\label{e:y_is}
\pi_{q_3}(\sum_i \ad(y_i)(\gfr))\subseteq \linValue{q_1q'}{q_2q'}{(\textstyle\prod_{O_{\Omega,\varepsilon}(1)} A)\cap \Gamma[q_1q']}.
\ee
Since $y_i\in \gfr$, there are $m_{ij}\in \bbz_p$ such that $y_i=\sum_{j=1}^{{\rm rank}_{\bbz_p} \gfr} m_{ij}x_j=\sum_{j=1}^{{\rm rank}_{\bbz_p} \gfr}m_{ij}m_j\tilde{x}_j$. Now as $\pi_{q_0p}(y_i)=\pi_{q_0p}(x_i)$, we have that $m_{ij}m_j=\delta_{ij}m_j \pmod{q_0p}$ for any $i$ and $j$, where $\delta_{ij}$ is one if $i=j$ and zero otherwise. Therefore $[m_{ij}]=I \pmod{p}$ which implies that $[m_{ij}]\in \GL_d(\bbz_p)$ and so $\pi_{q_3}(\sum_i\bbz_p y_i)=\pi_{q_3}(\gfr)$. Hence by Equation (\ref{e:y_is}) we have that $q_0\pi_{q_3}(\gfr)\subseteq \pi_{q_3}([\gfr,\gfr])\subseteq \linValue{q_1q'}{q_2q'}{(\textstyle\prod_{O_{\Omega,\varepsilon}(1)} A)\cap \Gamma[q_1q']}.$ This means that
\[
\pi_{q_3/p^{v_p(q_0)}}(\gfr)\subseteq \linValue{q_1q'p^{v_p(q_0)}}{q_2q'}{(\textstyle\prod_{O_{\Omega,\varepsilon}(1)} A)\cap \Gamma[q_1q'p^{v_p(q_0)}]},
\]
which implies that $\bcal(LL'q_0,T/q_0)$ holds.
\end{proof}

\section{Proof of Proposition~\ref{p:ThickTopSlice} (Thick top slice).}
\subsection{Finding a basis consisting of vectors with small height.}\label{ss:SpanSmallElements}
In this section, we prove the following proposition.
\begin{prop}\label{p:SpanSmallElements}
Let $\Omega$ and $\Gamma$ be as in Section~\ref{ss:InitialReductions}.
Let $\bbg_1$ be as in Section~\ref{ss:StrongApproximation}. Let $k'$ be a number field, and $S'\subseteq V_f(k')$ be a finite subset. Let $\rho:\bbg_1\rightarrow \mathbb{GL}(\bbv)$ be a $k'$-representation such that $\rho(\Gamma)\subseteq \GL_d(\ocal_{k'}(S'))$. 

Then for any $0<\delta\ll_{\rho,\Omega} \vare\ll_{\rho,\Omega} 1$ the following holds: 

If $Q$ is a positive integer, $A\subseteq \Gamma$ such that $\pcal_{\Omega}^{(l)}(A)>Q^{-\delta}$ for some positive integer $l>\frac{1}{\delta}\log Q$, then there is $\{a_i\}_{i=1}^{O_{\Omega}(1)}\subseteq \prod_{O_{\rho, \Omega}(1)} A$ such that 
\begin{enumerate}
\item for any $i$, $\|\rho(a_i)\|_{S'}\le Q^{\Theta_{\Omega,\rho}(\vare)}$, and 
\item  
$
\prod_{\pfr\in V_f(k')\setminus S'} [\ocal_{\pfr}[\rho(\Gamma)]:\sum_j \ocal_{\pfr}\rho(a_i)]\le Q^{\vare}.
$
where $\ocal_{\pfr}[\rho(\Gamma)]$ is the $\ocal_{\pfr}$-span of $\rho(\Gamma)$,
\end{enumerate}
where $\ocal_{\pfr}$ is the ring of integers of $k'_{\pfr}$ and  $\|v\|_{S'}:=\max_{\pfr\in V_{\infty}(k')\cup S'}\|v\|_{\pfr}$ (see Appendix, Section~\ref{ss:Norms} for the definition of $\|v\|_{\pfr}$ and the set $V_{\infty}(k')$ of all Archimedean places of $k'$). 
\end{prop}

For the rest of this section, let $\Omega$ and $\Gamma$ be as in Section~\ref{ss:InitialReductions}, and let $\bbg_1, \bbg_{1,\pfr}$ be as in Section~\ref{ss:StrongApproximation}. 

We need to work with normalized norms on number fields in order to have the product formula. In the Appendix, the needed definition of the $\pfr$-norm on a number field $k'$ is recalled. 

\begin{lem}\label{l:IndexControl}
Let $v_1,\ldots,v_m\in \ocal_{k'}(S')^d\setminus \{0\}$. Then
\[
\prod_{\pfr\in V_f(k')\setminus S'} |(\ocal_{\pfr}^d\cap \sum_i k'_{\pfr} v_i)/ \sum_i \ocal_{\pfr} v_i| \ll_{d,[k':\bbq]} \prod_i \|v_i\|_{S'}^{(|S'|+[k':\bbq])[k':\bbq]}.
\]
\end{lem}
\begin{proof}
For an $\ocal_{k'}(S')$-submodule $M$ of $\ocal_{k'}(S')^d$, let 
\[
c_{S'}(M):=\prod_{\pfr\in V_f(k')\setminus S'} |(\ocal_{\pfr}^d \cap k'[M])/ \ocal_{\pfr}[M]|.
\]
So if $k'[M_1]=k'[M_2]$ and $M_1\subseteq M_2$, then $c_{S'}(M_1)\ge c_{S'}(M_2)$. Hence, since $\|v\|_{S'}\ge 1$ for any $v\in \ocal_{k'}(S')\setminus \{0\}$, without loss of generality we can and will assume that $v_1,\ldots, v_m$ are linearly independent over $k'$. Let $X$ be the $d$-by-$m$ matrix having $v_1,\ldots,v_m$ on its columns. So it is a full-rank matrix. Without loss of generality we can and will assume that the first $m$ rows are linearly independent. Let ${\rm Pr}$ be the projection onto the first $m$ components. So the restriction of ${\rm Pr}$ to $\sum_i k'v_i$ is injective, which implies that 
\[
|(\ocal_{\pfr}^d\cap \sum_i k'_{\pfr}v_i)/\sum_i \ocal_{\pfr} v_i|\le 
|
\ocal_{\pfr}^m/ \sum_i \ocal_{\pfr} {\rm Pr}(v_i)|,
\]
and so $c_{S'}(\sum_i \ocal_{k'}(S') v_i)\le c_{S'}(\sum_i \ocal_{k'}(S') {\rm Pr}(v_i))$. Hence, since $\|{\rm Pr}(v)\|_{S'}\le \|v\|_{S'}$ for any $v\in \ocal_{k'}(S')^d$, without loss of generality we can and will assume that $m=d$, i.e. $v_i$'s span $k'^d$. Thus $X\in \GL_d(k')\cap M_d(\ocal_{k'}(S'))$. Therefore, for any $\pfr\in V_f(k')\setminus S'$, there are matrices $K_1,K_2\in \GL_d(\ocal_{\pfr})$ and non-negative  integers $n_1,\ldots,n_d$ such that 
\[
X=K_1 {\rm diag}(\pfr^{n_1},\ldots,\pfr^{n_d}) K_2.
\]     
(Here we are abusing the notation and using $\pfr$ as a uniformizing element of $\ocal_{\pfr}$, as well.) Hence we have 
\begin{align}
\notag [\ocal_{\pfr}^d:X(\ocal_{\pfr}^d)]&=[K_1^{-1}(\ocal_{\pfr}^d):{\rm diag}(\pfr^{n_1},\ldots,\pfr^{n_d}) K_2(\ocal_{\pfr}^d)]\\
\notag &=[\ocal_{\pfr}^d:{\rm diag}(\pfr^{n_1},\ldots,\pfr^{n_d})(\ocal_{\pfr}^d)]\\
\label{e:Index} &=|\f_{\pfr}|^{\sum_i n_i}.
\end{align}
On the other hand, 
\begin{align}
\notag |\det(X)|_{\pfr}&=|\det(K_1 {\rm diag}(\pfr^{n_1},\ldots,\pfr^{n_d}) K_2)|_{\pfr}\\
\notag &=|\det({\rm diag}(\pfr^{n_1},\ldots,\pfr^{n_d}))|_{\pfr}\\
\label{e:NormDet} &=(|\f_{\pfr}|^{-\sum_i n_i})^{1/[k':\bbq]}. 
\end{align}
Therefore by Equations (\ref{e:Index}) and (\ref{e:NormDet}), we have
\[
c_{S'}(\sum_i \ocal_{k'}(S')v_i)=\prod_{\pfr\in V_f(k')\setminus S'} [\ocal_{\pfr}^d:X(\ocal_{\pfr}^d)]=\prod_{\pfr\in V_f(k')\setminus S'} |\det(X)|_{\pfr}^{-[k':\bbq]}.
\]
Thus by the product formula we have
\begin{align*}
c_{S'}(\sum_i \ocal_{k'}(S')v_i)&=\prod_{\pfr\in V_{\infty}(k')\cup S'} |\det(X)|_{\pfr}^{[k':\bbq]}\\
& \le (\prod_{\pfr\in V_{\infty}(k')} (d!) \prod_i \|v_i\|_{\pfr})^{[k':\bbq]}(\prod_{\pfr\in S'} \prod_i \|v_i\|_{\pfr})^{[k':\bbq]}\\
&\le (d!)^{[k':\bbq]^2} \prod_{i} \|v_i\|_{S'}^{(|S'|+[k':\bbq])[k':\bbq]}.
\end{align*}
\end{proof}


The next lemma shows that, if the probability of hitting $A$ in $l$-steps is at least $e^{-\delta l}$, then elements with {\em small} logarithmic height in $A.A$ generate a Zariski-dense subgroup of $\bbg_1$. 

\begin{lem}\label{l:SmallElementsGenerateZariskiDense}
Let $0<\delta\ll_{\Omega} \vare< 1$. Let $Q$ be a positive integer. If 
$\pcal_{\Omega}^{(l)}(A)>Q^{-\delta}$ for some positive integer $l>\frac{1}{\delta}\log Q$ and $Q^{\vare}\gg_{\Omega} 1$, then the group generated by $A\cdot A \cap B_{\lfloor \vare \log Q\rfloor}$ is Zariski-dense in $\bbg_1$, where $B_r$ is the ball of radius $r$ in $\Gamma$ with respect to the $\Omega$-word metric. 
\end{lem}

\begin{proof}
If $\frac{1}{\delta}\log Q\ge \vare' \log Q$, then by Lemma~\ref{l:BGRemark} and $\pcal_{\Omega}^{(l)}(A)>Q^{-\delta}$ we have that
\be\label{e:SmallBallLargeMeasure}
\pcal_{\Omega}^{\lceil \vare' \log Q\rceil}(A.A)\ge e^{-2\delta \log Q}.
\ee
Now if $\delta_0:=\delta_0(\Omega)$ and $l_0:=l_0(\Omega)$ are as in Proposition~\ref{p:EscapeSubgroups}, $\delta\le \delta_0 \vare'$ and $\lceil \vare' \log Q\rceil\ge l_0$, then by Equation (\ref{e:SmallBallLargeMeasure}) and Proposition~\ref{p:EscapeSubgroups} we have that the group generated by 
\[
A\cdot A \cap B_{\lceil \vare' \log Q\rceil}
\]
is Zariski-dense in $\bbg_1$.
\end{proof}


\begin{lem}\label{l:ZariskiDenseGroupRing}
Let $F$ be a field, $G$ be a subgroup of $\GL_n(F)$, and $1\in A\subseteq G$ be a symmetric subset which generates a Zariski-dense subgroup of $G$. Then 
\[
\sum_{a\in \prod_{n^2} A}Fa=F[G],
\] 
where $F[G]$ is the $F$-span of $G$ in $M_n(F)$.
\end{lem}
\begin{proof}
Since $A$ generates a Zariski-dense subgroup, the $F$-algebra generated by $A$ is equal to $F[G]$. As $\dim_FF[G]\le n^2$, the $F$-span of $\prod_{n^2}A$ is equal to $F[G]$.
\end{proof}

\begin{proof}[Proof of Proposition~\ref{p:SpanSmallElements}]
By Lemma~\ref{l:SmallElementsGenerateZariskiDense} and Lemma~\ref{l:ZariskiDenseGroupRing}, there are $\{\gamma_i\}_{i=1}^{O_{\dim \rho}(1)}\subseteq \prod_{\dim \rho^2} (A.A\cap B_{\lfloor \vare \log Q\rfloor})$ such that $k'[\rho(\Gamma)]=\sum_i k'\rho(\gamma_i)$. So in particular  $\log \|\rho(\gamma_i)\|_S \ll_{\rho,\Omega} \vare \log Q$, and
\[
M_d(\ocal_{\pfr})\cap k'_{\pfr}[\rho(\Gamma)]=M_d(\ocal_{\pfr})\cap \sum_i k'_{\pfr}\rho(\gamma_i)\supseteq \ocal_{\pfr}[\rho(\Gamma)],
\] 
for any $\pfr\in V_f(k')\setminus S$. Therefore we have
\begin{align*}
\log \prod_{\pfr\in V_f(k')\setminus S} [\ocal_{\pfr}[\rho(\Gamma)]:\sum_i \ocal_{\pfr} \rho(\gamma_i)]
=&\sum_{\pfr\in V_f(k')\setminus S}\log [(\ocal_{\pfr}[\rho(\Gamma)]:\sum_i \ocal_{\pfr} \rho(\gamma_i)]\\
&\le \sum_{\pfr\in V_f(k')\setminus S}\log [M_d(\ocal_{\pfr})\cap \sum_i k'_{\pfr}\rho(\gamma_i):\sum_i \ocal_{\pfr} \rho(\gamma_i)]\\
\text{(By Lemma~\ref{l:IndexControl})}\h\h\h &\ll_{\dim \rho,[k':\bbq],|S|} \sum_{i}\log \|\rho(\gamma_i)\|_S \\
(\text{for }Q^{\vare}\gg_{\Omega,\rho} 1)\h\h\h &\ll_{\Omega,\rho} \vare \log Q. 
\end{align*}
\end{proof}
\subsection{Large number of conjugacy classes.}\label{ss:LargeNumberOfConjugacyClasses} In this section, by the virtue of the proof of Burnside's theorem, we prove the following proposition.
\begin{prop}\label{p:LargeNumberOfConjugacyClasses}
Let $\Omega$, $\Gamma$ (we might pass to a subgroup of finite-index using Lemma~\ref{l:RepIntegrality}), $\bbg_1$, and $\bbg_{1,\pfr}$ be as before, and let $\pi_Q$ be as in the beginning of Section~\ref{ss:StatementApproximateSubgroupBoundedGeneration}. And suppose $0<\delta\ll_{\Omega}\vare\ll_{\Omega} 1$ and $1\ll_{\Omega} \vare \log_p Q$.  

If $\pcal_{\Omega}^{(l)}(A)>Q^{-\delta}$ for some positive integer $l>\frac{1}{\delta}\log Q$, then there is $\{a_i\}_{i=1}^{O_{\Omega}(1)}\subseteq \prod_{O_{\Omega}(1)}A$ such that for any positive integer 
$n\gg_{\Omega} \vare \log_p Q$ there is a positive integer $n'$ with the following properties:
\begin{enumerate}
	\item $n'\le n\le n'+\vare \log_p Q$,
	\item For any $X\subseteq \Gamma$, there is an $i$ such that
	\[
	|\pi_{q'}(X)|^{\Theta_{\Omega}(1)}\le |\conj(\pi_q(Xa_i))|,
	\]
	where $q:=p^n$, $q':=p^{n'}$, and $\conj(Y)=\{[y]|\h y\in Y\}$ where $[y]$ is the intersection of the conjugacy class of $y$ and $Y$.
\end{enumerate}
\end{prop}
One can deduce Proposition~\ref{p:LargeNumberOfConjugacyClasses} from  the following proposition.
\begin{prop}\label{p:ElementsToConjugacyClasses}
In the above setting, suppose $0<\delta\ll_{\Omega}\vare\ll_{\Omega} 1$ and $1\ll_{\Omega} \vare \log_p Q$. Suppose $p\in V_f(\bbq)$ and $\pfr|p$ for some $\pfr\in V_f(k)\setminus S$. 

If $\pcal_{\Omega}^{(l)}(A)>Q^{-\delta}$ for some positive integer $l>\log Q/\delta$, then there is $\{a_i\}_{i=1}^{O_{\Omega}(1)}\subseteq \prod_{O_{\Omega}(1)}A$ such that for any positive integer $n\gg_{\Omega} \vare \log_p Q$, there is a positive integer $n'$ with the following properties:
\begin{enumerate}
	\item $n'\le n\le n'+\vare \log_p Q$,
	\item For any $\gamma\in \Gamma$, let $\Xi_q(\gamma):=(\ccal(\pi_q(\gamma a_i)))_i$, where $q:=p^n$ and $\ccal(y)$ is the conjugacy class of $y$. If $\Xi_q(\gamma_1)=\Xi_q(\gamma_2)$, then $\pi_{q'}(\gamma_1)=\pi_{q'}(\gamma_2)$ where $q':=p^{n'}$.
\end{enumerate}
\end{prop} 
\begin{proof}
Recall from Section~\ref{ss:InitialReductions} that $\Gamma\subseteq \GL_{n_0}(\ocal_k(S))$, and $S_0:=\{p\in V_f(\bbq)|\h \exists \pfr\in S, \pfr|p\}$. Moreover $\bbg_1$ is a $\bbq$-semisimple group, and $\bbg_{1,\pfr}$ is a factor of $\bbg_1$ that is defined over a subfield $k(\pfr)$ of $k$. So there are absolutely irreducible representations $\rho_i:\bbg\rightarrow \mathbb{GL}_{n_i}$ that are defined over a Galois extension $k'$ of $k$, and $\rho=\oplus_{i\in I} \rho_i$ is a faithful representation. Furthermore we can and will assume that for any $\pfr\in \{\pfr\in V_f(k)\setminus S|\h \exists p\in S_0, \pfr|p\}$ there is a subset $I_{\pfr}$ of $I$ such that $\rho_i$ factors through $\bbg_{1,\pfr}$ for any $i\in I_{\pfr}$ and $\oplus_{i\in I_{\pfr}}\rho_i$ induces a faithful representation of $\bbg_{1,\pfr}$. We consider $\rho_i$'s fixed and do not mention the dependence of constants to the choice of such homomorphisms. Let
\[
\ccal_i:=\{\wt{\pfr}\in V_f(k')|\h \rho_i(\Gamma) \text{ is a bounded subgroup of } \GL_{n_i}(k'_{\wt{\pfr}})\}.
\]
So we have that $\{\wt{\pfr}\in V_f(k')|\h \exists p\in V_f(\bbq)\setminus S_0, \wt{\pfr}|p\}\subseteq \ccal_i$ for any $i\in I$. And for any $\pfr\in V_f(k)\setminus S$ and $i\in I_{\pfr}$ we have $\{\wt{\pfr}\in V_f(k')|\h \wt{\pfr}|\pfr\}\subseteq \ccal_i$.

Let $S_i:=V_f(k')\setminus \ccal_i$. By Lemma~\ref{l:RepIntegrality} and Lemma~\ref{l:FiniteIndexSubgroup}, we can assume that 
\[
\rho_i(\Gamma)\subseteq \GL_{n_i}(\ocal_{k'}(S_i)).
\] 
Therefore by Proposition~\ref{p:SpanSmallElements}, there is $\{a_i\}_{i=1}^{O_{\Omega}(1)}\subseteq \prod_{O_{\Omega}(1)}A$ such that for any $i\in I$ and $j$
\be\label{e:SmallHeight}
\|\rho_i(a_j)\|_{S_i}\le Q^{\Theta_{\Omega}(\vare)},
\ee
and 
\be\label{e:SpanSmallElements}
\prod_{\wt{\pfr}\in \ccal_i}|\ocal_{\wt{\pfr}}[\rho_i(\Gamma)]/\sum_j \ocal_{\wt{\pfr}} \rho_i(a_j)|\le Q^{\vare}.
\ee

Since $\rho_i$ is absolutely irreducible, we have $k'[\rho_i(\Gamma)]=M_{n_i}(k')$. By (\ref{e:SpanSmallElements}), there is a subset $J_i$ of indexes such that $\{\rho_i(a_j)\}_{j\in J_i}$ is a $k'$-basis of $M_{n_i}(k')$. Let $T_i:M_{n_i}(k')\rightarrow k'^{n_i^2}$ be
\[
T_i(x):=(\Tr(x\rho_i(a_j)))_{j\in J_i}.
\] 
Since $f_i:M_{n_i}(k')\times M_{n_i}(k')\rightarrow k', f_i(x,y):=\Tr(xy)$ is a $k'$-bilinear non-degenerate map, $T_i$ is an invertible $k'$-linear map. Moreover the matrix representation of $T_i$ in the standard basis of $M_{n_i}(k')$ has entries in $\ocal_{k'}(S_i)$, and by (\ref{e:SmallHeight}), the $S_i$-norm of its entries are at most $Q^{\Theta_{\Omega}(\vare)}$. Hence by the product formula we have
\be\label{e:LargePAdicDet}
|\det(T_i)|_{\wt{\pfr}}\le 1, \text{ and } \prod_{\wt{\pfr}\in \ccal_i}|\det(T_i)|_{\wt{\pfr}}\ge Q^{-\Theta_{\Omega}(\vare)}.
\ee
Notice that (\ref{e:LargePAdicDet}) implies that for any $\wt{\pfr}\in\ccal_i$ we have
\be\label{e:UpperBoundInverseOperator}
\|T_i^{-1}\|_{\pfr}\le Q^{\Theta_{\Omega}(\vare)}.
\ee

For any $\pfr\in V_f(k)\setminus S$ and $\wt{\pfr}\in V_f(k')$ which divides $\pfr$ we have 
\be\label{e:AlmostIsometry}
\|g_1-g_2\|_{\wt{\pfr}} \ll \max_{i\in I_{\pfr}} \|\rho_i(g_1)-\rho_i(g_2)\|_{\wt{\pfr}}\ll \|g_1-g_2\|_{\wt{\pfr}},
\ee
for any $g_1,g_2\in \bbg_{1,\pfr}(\bbz_p)$, since $\oplus_{i\in I_{\pfr}} \rho_i$ induces an injective homomorphism on $\bbg_{1,\pfr}$. 

Suppose $\pfr\in V_f(k)\setminus S$ divides $p\in V_f(\bbq)$, $q=p^n$, and $\Xi_q(\gamma_1)=\Xi_q(\gamma_2)$. Then for any $j$ we have that $\|\gamma_1 a_j-g_j \gamma_2 a_j g_j^{-1}\|_{\pfr}\le |q|_{\pfr}$ for some $g_j\in \Gamma$. Therefore for any $i$ and $j$ and any $\wt{\pfr}\in V_f(k')$ that divides $\pfr$ we have
\begin{align}
\notag |\Tr(\rho_i(\gamma_1 a_j))-\Tr(\rho_i(\gamma_2 a_j))|_{\wt{\pfr}}&=
|\Tr(\rho_i(\gamma_1 a_j)-\rho_i(g_j \gamma_2 a_j g_j^{-1}))|_{\wt{\pfr}}
\\
\label{e:FirstEstimate} &\ll
\|\gamma_1 a_j-g_j \gamma_2a_j g_j^{-1}\|_{\wt{\pfr}}\le  |q|_{\wt{\pfr}}. 
\end{align}
So by (\ref{e:FirstEstimate}) for any $i$ we have 
\be\label{e:SingleRep1}
\|T_i(\rho_i(\gamma_1))-T_i(\rho_i(\gamma_2))\|_{\wt{\pfr}}\ll |q|_{\wt{\pfr}}.
\ee
Hence by (\ref{e:UpperBoundInverseOperator}) and (\ref{e:SingleRep1}) we have
\be\label{e:SingleRep2}
\|\rho_i(\gamma_1)-\rho_i(\gamma_2)\|_{\wt{\pfr}}\le Q^{\Theta_{\Omega}(\vare)} |q|_{\wt{\pfr}}.
\ee
Therefore by (\ref{e:AlmostIsometry}) we have
\[
\|\gamma_1-\gamma_2\|_{\wt{\pfr}}\le Q^{\Theta_{\Omega}(\vare)}|q|_{\wt{\pfr}},
\]
which implies that $\pi_{q'}(\gamma_1)=\pi_{q'}(\gamma_2)$ for some $q'|q$ and $q'\ge qQ^{-\Theta_{\Omega}(\vare)}$. 
\end{proof}


\subsection{Helfgott's trick to get large centralizer.}\label{ss:HelfgottLargeCentralizer}
In this section, we recall Helfgott's method of getting a {\em large} centralizer and prove the following lemma.
\begin{lem}\label{l:HelfgottLargeCentralizer}
Let $\Omega$ and $\Gamma$ be as in Section~\ref{ss:InitialReductions}, and let $\pi_{p^n}$ be as in Section~\ref{ss:StatementApproximateSubgroupBoundedGeneration}. Assume $0<\delta\ll_{\Omega}\vare\ll_{\Omega} 1$ and $1\ll_{\Omega} \vare N$ where $N$ is a positive integer. Suppose $p$ is divisible  by some $\pfr\in V_f(k)\setminus S$. 

Suppose $\Pfr_Q(\delta,A,l)$ holds, where $Q:=p^N$, and $\{a_i\}_{i=1}^{O_{\Omega}(1)}\subseteq \prod_{O_{\Omega}(1)} A$ is as in Proposition~\ref{p:LargeNumberOfConjugacyClasses}. Then for any $X\subseteq \prod_{O_{\Omega,\varepsilon}(1)} A$ and for any integer $N\vare \ll_{\Omega} n \le N$ there are $i$, $x\in Xa_i$, and a positive integer $n'$ such that
\begin{enumerate}
\item $n'\le n\le n'+\vare N$,
\item $|C_{\pi_q(\Gamma)}(\pi_q(x))\cap \pi_q(A\cdot A)|\geq |\pi_Q(A)|^{-\Theta_{\Omega,\vare}(\delta)}|\pi_{q'}(X)|^{\Theta_{\Omega}(1)},$ where $q:=p^n$ and $q':=p^{n'}$,
\end{enumerate} 
where $C_{\pi_q(\Gamma)}(\pi_q(x))$ is the centralizer of $\pi_q(x)$ in $\pi_q(\Gamma)$.
\end{lem}
One can deduce Lemma~\ref{l:HelfgottLargeCentralizer} from Proposition~\ref{p:LargeNumberOfConjugacyClasses} and the following lemma.
\begin{lem}\label{l:HelfgottLargeCentralizerLargeNumberOfConjugacy}
In the above setting. Assume $0<\delta\ll_{\Omega}\vare\ll_{\Omega} 1$ and $1\ll_{\Omega}\vare N$ where $N$ is a positive integer. Suppose $p$ is divisible  by some $\pfr\in V_f(k)\setminus S$. 

Suppose that $\Pfr_Q(\delta,A,l)$ holds, where $Q:=p^N$, and $\{a_i\}_{i=1}^{O_{\Omega}(1)}\subseteq \prod_{O_{\Omega}(1)} A$ is as in Proposition~\ref{p:LargeNumberOfConjugacyClasses}. Then for any $i$ and any $X\subseteq \prod_{O_{\Omega,\varepsilon}(1)} A$ and for any $q|Q$ there is $x_0\in X$ such that
\[
|C_{\pi_q(\Gamma)}(\pi_q(x_0a_i))\cap \pi_q(A\cdot A)|\geq |\pi_Q(A)|^{-\Theta_{\Omega,\vare}(\delta)} |\conj(\pi_q(Xa_i))|.
\]
\end{lem}
\begin{proof}
Since $|\pi_Q(A)|^{1+\delta}\geq |\pi_Q(\prod_3 A)|$, for any $q|Q$ we have that 
\be\label{e:AlmostSubgroup}
|\pi_q(A)||\pi_Q(A)|^{\delta}\geq |\pi_q(\textstyle{\prod_3} A)|. 
\ee
Therefore by the Ruzsa inequality we have that $|\pi_q(\prod_{O_{\Omega,\vare}(1)}A)|\leq |\pi_Q(A)|^{\Theta_{\Omega,\vare}(\delta)}|\pi_q(A)|$. 

For any $X\subseteq \prod_{O_{\Omega,\vare}(1)}A$ and $i$ we have that
\begin{align}
|\pi_q(\textstyle{\prod_{O_{\Omega,\vare}(1)}}A)|&\geq |\pi_q(AXa_iA)|\\
&\ge\sum_{[\pi_q(x)]\in \conj(\pi_q(Xa_i))} |\{\pi_q(axa^{-1})|\h a\in A\}|\\
&\ge |\pi_q(A)|\cdot |\conj(\pi_q(Xa_i))|\cdot \bbe_{[\pi_q(x)]}\left(\frac{1}{|C_{\pi_q(\Gamma)}(\pi_q(x))\cap \pi_q(A\cdot A)|}\right).
\end{align}
So for some $x_0\in X$ we have that 
\be\label{e:LargeCentralizer}
|C_{\pi_q(\Gamma)}(\pi_q(x_0a_i))\cap \pi_q(A\cdot A)|\ge \frac{|\pi_q(A)|}{|\pi_q(\textstyle{\prod_{O_{\Omega,\vare}(1)}}A)|}\cdot |\conj(\pi_q(Xa_i))|.
\ee

\end{proof}

\subsection{Measuring the regularity.}\label{ss:q-regularSemisimple}
In this section, we recall what a {\em regular semisimple} element is, and we describe a way to measure {\em how much} regular semisimple an element is. Let $K$ be a local non-archimedean field, $|\cdot|$ its norm, $\ocal$ its ring of integers, and $\pfr$ a uniformizing element. Let $\bbh$ be a connected semisimple $K$-group, and let $r$ be the absolute rank of $\bbh$. Steinberg~\cite[Corollary 6.8]{Ste} proved that there is $h_s\in K[\bbh]$ such that 
\begin{enumerate}
\item $h_s(g)=F_g(1)$, where $F_g(x)(x-1)^r:=\det(\Ad(g)-xI)$,
\item $g\in \bbg(\overline{K})$ is {\em regular semisimple} if and only if $h_s(g)\not=0$.
\end{enumerate}
 
\begin{definition}
Let $\cal$ be an $\ocal$-subscheme of $(\mathcal{GL}_n)_{\ocal}$. Suppose the generic fiber $\bbh$ of $\cal$ is a connected semisimple $K$-group. Let $f_s\in \ocal[\cal]$ be
\[
f_s(g)=R\left((x-1)F_g(x),\frac{d}{dx}((x-1)F_g(x))\right),
\]
where $R(f_1,f_2)$ is the resultant of $f_1$ and $f_2$. For $\eta>0$, we say $g\in\cal(\ocal)$ is $\eta$-regular semisimple if $|f_s(g)|>\eta$. Equivalently $g\in \cal(\ocal)$ is $|q|$-regular semisimple if  and only if $\pi_q(f_s(g))\neq 0$.
\end{definition}

\begin{remark}
If $R$ is any integral domain and $\cal$ is an $R$-subscheme of $(\mathcal{GL}_n)_R$ with a connected semisimple generic fiber $\bbh$, then $f_s\in R[\cal]$ can be defined in the same way.  
\end{remark}

Before summarizing the basic properties of $\eta$-regular semisimple elements, an elementary lemma on certain diagonalizable elements of $\GL_n(\ocal)$ is proved. 

\begin{lem}\label{l:q-regularGL}
Suppose $a\in \GL_n(\ocal)$ is diagonalizable over $K$ and let $\lambda_i$ be its distinct eigenvalues (which are not necessarily of multiplicity one). Assume that $\prod_{i\neq j}(\lambda_i-\lambda_j)\in \pfr^m\ocal\setminus \pfr^{m+1}\ocal$. Then  
\[
\pfr^{m}\ocal^n\subseteq \bigoplus_{i}(V_{\lambda_i}\cap \ocal^n),
\]
where $V_{\lambda_i}:=\{\vbf \in K^n|\h a \vbf=\lambda_i \vbf\}$.
\end{lem}
\begin{proof}
Since $a$ is diagonalizable, we have that $K^n=\bigoplus_i V_{\lambda_i}$. So for any $\xbf\in \ocal^n$ there are $\xbf_i\in V_{\lambda_i}$ such that $\xbf=\sum_i \xbf_i$. Therefore for any integer $j$ we have $a^j \xbf=\sum_i \lambda_i^j \xbf_i\in \ocal^n$. Since $a\in\GL_n(\ocal)$ and it is diagonalizable over $K$, all the eigenvalues $\lambda_i$ are in $\ocal$. Hence by the Vandermonde equality we have that $\prod_{i\neq j}(\lambda_i-\lambda_j) \xbf_k\in \ocal^n$ for any $k$, which implies that 
\[
\pfr^{m}\ocal^n\subseteq \bigoplus_{i}(V_{\lambda_i}\cap \ocal^n).
\]
\end{proof}

\begin{lem}\label{l:q-regularBasics}
In the above setting, let $g\in \cal(\ocal)$, $q\in \pfr \ocal$, and $\eta:=|q|$. Suppose $g$ is $\eta$-regular semisimple. Then
\begin{enumerate}
\item\label{item1} $\bbt:=C_{\bbg}(g)^{\circ}$ is a maximal $K$-torus of $\bbh$.
\item\label{item2} $g\in\bbt(K)$.
\item\label{item3} There is a Galois extension $K'$ of $K$ such that $[K':K]\le n!$, and $\Ad(\bbt)$ splits over $K'$. In particular,
\[
\hfr(K')=\tfr(K')\oplus\bigoplus_{\phi\in \Phi(\bbh,\bbt)}\hfr_{\phi}(K'),
\] 
where $\hfr:=\Lie(\cal)$, $\bbt=\Lie(\bbt)$, $\Phi(\bbh,\bbt)$ is the set of roots with respect to $\bbt$ and $\hfr_{\phi}$ are the corresponding root spaces. In this setting, we have 
\begin{align*}
h_s(g)&=\prod_{\phi\in\Phi(\bbh,\bbt)}(\phi(g)-1)\\
f_s(g)&=\pm h_s(g)^2\cdot\prod_{\phi_1\neq \phi_2\in\Phi(\bbh,\bbt)}(\phi_1(g)-\phi_2(g))^2\not\in q \ocal',
\end{align*}
where $\ocal'$ is the ring of integers of $K'$. 
\item\label{item4} $g'\in \cal(\ocal)$ is $|q|^{O_{n}(1)}$-regular semisimple if and only if for any distinct roots $\phi$ and $\phi'$
\[
|\phi(g')-1|\ge |q|^{O_n(1)},\h \text{and}\h\h\h |\phi(g')-\phi'(g')|\ge |q|^{O_n(1)},
\]
for suitable choice of the implied constants.
\item\label{item} We have 
\[
q\hfr(\ocal')\subseteq (\hfr(\ocal')\cap \tfr(K'))\oplus \bigoplus_{\phi\in\Phi(\bbh,\bbt)} (\hfr(\ocal')\cap \hfr_{\phi}(K')).
\]
And so, for $q'\in q^2\ocal$ and $x\in\hfr(\ocal')$, if $\Ad(g)(x)=x\pmod{q'}$, then 
\[
x\equiv h\pmod{q'/q^2},
\]
for some $h\in \hfr(\ocal')\cap\tfr(K')$.
\end{enumerate}
\end{lem}
\begin{proof}
Since $F_{g}(x)$ is a monic polynomial, $f_s(x)$ is, up to sign, the discriminant of $(x-1)F_{g}(x)$. Therefore, if $f_s(g)$ is non-zero, then $h_s(g)$ is non-zero. And so $g$ is a regular semisimple element. This implies Part~(\ref{item1}). Part (\ref{item2}) is a well-known consequence of Part (\ref{item1}) (see~\cite[Chapter 22.3]{Hum}). Parts (\ref{item3}) and (\ref{item4}) can easily be deduced from the definitions and the earlier argument. The first claim of Part (\ref{item}) is a direct consequence of Lemma~\ref{l:q-regularGL} and Part (\ref{item3}). For any $x\in\hfr(\ocal')$ there are $h'\in \hfr(\ocal')\cap \tfr(K')$ and $x_{\phi}\in \hfr(\ocal')\cap \gfr_{\phi}(K')$ such that 
$
qx=h'+\sum_{\phi\in\Phi(\bbg,\bbt)} x_{\phi}.
$
Hence 
\[
\Ad(g)(qx)-(qx)=\sum_{\phi\in\Phi(\bbh,\bbt)}(\phi(g)-1) x_{\phi}= qq'y,
\]
for some $y\in \hfr(\ocal')$. Again using the first claim of Part (\ref{item}) we know that there are $y_{\phi}\in \hfr(\ocal')\cap \hfr_{\phi}(K')$ such that $qy=\sum_{\phi} y_{\phi}$. Hence
for any $\phi\in\Phi(\bbh,\bbt)$ we have $(\phi(g)-1) x_{\phi}=q'y_{\phi}$. On other hand, by Part (\ref{item3}), we have that $|\phi(g)-1|>|q|$. Therefore $x_{\phi}\in (q'/q) \hfr(\ocal')$, which implies the second claim of Part (\ref{item}).
\end{proof}
\begin{lem}\label{l:q-RegularSemisimpleCentralizer}
Let $\Omega$ and $\Gamma$ be as in Section~\ref{ss:InitialReductions}, and $\gcal_1$, $\gcal_{1,\pfr}$, and $\pi_{p^n}$ be as in Section   \ref{ss:StatementApproximateSubgroupBoundedGeneration}. Let $\pfr\in V_f(k)\setminus S$ which divides $p\in V_f(\bbq)$. Let the triple of a group scheme $\cal$, a number field $k'$, and a norm $|\cdot|_\nu$ of $k'$  be either $(\gcal_1,\bbq,|\cdot|_p)$ if $p\not\in S_0$, or $(\gcal_{1,\pfr},k(\pfr),|\cdot|_{\pfr})$ if $p\in S_0$. 

Suppose $q, q_0, q_1, q_2$ are powers of $p$, $q_1|q_2$, $q_0q^2q_2|q_1^2$, and $\log_p q_0\gg_n 1$ and $\log_p (q_1^2/(q_0q^2 q_2))\gg_{\Gamma} 1$. If $\gamma\in \Gamma\subseteq \cal(\bbz_p)$ is $|q|_{\nu}$-regular semisimple, then  
\[
\lin{q_1}{q_2}(\pi_{q_2}(C_{\pi_{q_0q^2q_2}(\Gamma)}(\pi_{q_0q^2q_2}(\gamma)))\cap\pi_{q_2}(\Gamma[q_1]))\subseteq \pi_{q_2/q_1}(\Lie(\cal)(\bbz_p)\cap\tfr(k')),
\]
where $\tfr:=\Lie C_{\bbh}(\gamma)^{\circ}$ and $\bbh$ is the generic fiber of $\cal$.
\end{lem}
\begin{proof}
Suppose $\pi_{q_0q^2 q_2}(\gamma')\in C_{\pi_{q_0q^2 q_2}(\Gamma)}(\pi_{q_0q^2q_2}(\gamma))$ such that $\pi_{q_2}(\gamma')\in \pi_{q_2}(\Gamma[q_1])$. Therefore $\gamma'\in \Gamma[q_1]$. Let $q_3':=q_0q^2 q_2/q_1$ and $\hfr:=\Lie(\cal)(\bbz_p)$. 

Since $\log_p (q_1^2/(q_0q^2 q_2))\gg_{\Omega} 1$, by Part 2 of Lemma~\ref{l:PropertiesOfLinearizationMap} we have that 
\[
\lin{q_1}{q_0q^2 q_2}:\pi_{q_0q^2 q_2}(\Gamma[q_1])\rightarrow \hfr/q_3'\hfr
\]
 is a $\Gamma$-module isomorphism. Hence there is $x\in \hfr$ such that 
\[
 \pi_{q_3'}(x)=\lin{q_1}{q_0 q^2 q_2}(\pi_{q_0 q^2 q_2}(\gamma')),\hspace{1cm} \Ad(\gamma)(x)\equiv x \pmod{q_3'}.
\]
Thus by Part 5 of Lemma~\ref{l:q-regularBasics} there is a Galois extension $k''$ of $k'$ and $h'\in \Lie(\cal)(\ocal)\cap \Lie(\bbt)(k'')$, where $\bbt:=C_{\bbh}(\gamma)^{\circ}$ and  $\ocal$ is the ring of integers of a composite field of $k''$ and $\bbq_p$, such that
\be\label{e:GoingToLieAlgebra}
[k'':k']\ll_n 1, \text{ and } x\equiv h' \pmod{q_3'/q^2}, \text{ i.e. } x\equiv h \pmod{q_0q_2/q_1}.
\ee
(Let us recall that $k(\pfr)$ is essentially the intersection of the copy of $\bbq_p$ in $k_{\pfr}$ and $k$.) So there is $x'\in \Lie(\cal)(\ocal)$ such that 
\[
x=h'+q_0q_3 x',
\]
where $q_3:=q_2/q_1$. Since $x\in \Lie(\cal)(k')$, for any $\sigma\in \gal(k''/k')$ we have 
\[
h'-\sigma(h')=q_0q_3(\sigma(x')-x'). 
\]
Therefore we have 
\[
h'= \frac{1}{[k'':k']}\sum_{\sigma\in \gal(k''/k')} \sigma(h')+ \frac{q_0q_3}{[k'':k']} \sum_{\sigma\in \gal(k''/k')} (\sigma(x')-x').
\]
Since $\log_p q_0\gg_n 1$, we have 
\[
x-h\in q_3 \Lie(\cal)(\ocal),
\]
where $h:=\frac{1}{[k'':k']}\sum_{\sigma\in \gal(k''/k')} \sigma(h')$. In particular, $h$ is invariant under the Galois action. Since $\bbt$ is defined over $k'$, we have  $h\in \Lie(\bbt)(k')$. Altogether we have that
\[
x-h\in q_3 \hfr, \h\h h\in \Lie(\cal)(\bbz_p)\cap \Lie(\bbt)(k').
\]
\end{proof}

\subsection{Hitting shifts of regular semisimple elements.}\label{ss:LotsRegularSemisimpleModq}
The main result of this Section is Proposition~\ref{p:q-nonRegularSemisimple}. It is essentially proved, in a {\em quantitative way}, that after certain {\em number of steps} in the random walk there is only a {\em small chance} of hitting a translation by a {\em small} element of a not-{\em sufficiently regular} element. Of course the main issue is gaining control on the relation between the four parameters: the needed number of steps in the random walk; the size of translation; the size of regularity; and an upper bound on the considered probability.   

For the rest of this section, let $\Omega$ and $\Gamma$ be as in Section~\ref{ss:InitialReductions}, and $\gcal_1$, $\gcal_{1,\pfr}$, and $\pi_{p^n}$ be as in Section   \ref{ss:StatementApproximateSubgroupBoundedGeneration}. Let $\pfr\in V_f(k)\setminus S$ which divides $p\in V_f(\bbq)$. Let the quadruple $(\cal,k',\nu,S')$ of a group scheme $\cal$, a number field $k'$, a place $\nu\in V_f(k')$, and a finite subset $S'$ of $V_f(k')$ be either $(\gcal_1,\bbq,p,S_0)$ if $p\not\in S_0$, or $(\gcal_{1,\pfr},k(\pfr),\pfr,S)$ if $p\in S_0$. In particular, $\Gamma\subseteq \cal(\ocal_{k'}(S'))$.

\begin{prop}~\label{p:q-nonRegularSemisimple}
Suppose $0<\vare\ll_{\Omega} 1$, $q:=p^n$, $Q:=p^N$ where $n,N\in \bbz^+$ and $n\ge N\vare$. Let $a\in\Gamma$ such that $Q^{\vare}\ge \|a\|_{S'}$ and $l\gg_{\Omega} \vare \log Q$. Then
\[
\pcal_{\Omega}^{(l)}(\{\gamma\in \Gamma|\h \gamma a \text{ is not } |q|_{\nu}-\text{regular semisimple}\})\le Q^{-\Theta_{\Omega}(\vare)}.
\]
\end{prop}
To simplify the presentation, let 
\[
W_{a,q}:=\{\gamma\in \Gamma|\h \gamma \text{ is not } |q|_{\nu}-\text{regular semisimple}\}.
\]
By the definition of $|q|_{\nu}$-regular semisimple elements, we know that 
\[
W_{a,q}:=\{\gamma\in \Gamma|\h \pi_q(f_s(\gamma a))=0\}.
\]
And, for $\gamma_1,\gamma_2\in \Gamma$ and $q_1,q_2$, powers of $p$, let 
\begin{align*}
V_{q_1,q_2}(\gamma_1,\gamma_2)&:=\{\gamma\in \Gamma|\h \gamma_1\gamma\gamma_2\in W_{1,q_1} \h{\rm and}\h \|\gamma\|_{S'}\le q_2\}\\
&=\{\gamma\in \Gamma|\h \pi_{q_1}(f_s(\gamma_1 \gamma \gamma_2))=0,\h \|\gamma\|_{S'}\le q_2\}.
\end{align*}

For $g\in\bbh(\overline{k'})$ and $\gamma_0\in \Gamma$, let $f_{g,\gamma_0}(g'):=f_s(gg'\gamma_0)$ and $V(f_{g,\gamma_0})\subseteq \bbg$ be the zeros of $f_{g,\gamma_0}$, where $\bbh$ is the generic fiber of $\cal$. The following is the key lemma in the proof of Proposition~\ref{p:q-nonRegularSemisimple}.
\begin{lem}\label{l:SmallLiftq-nonRegularSemisimpleProperVariety}
Let $\gamma_1,\gamma_2\in\Gamma$. Suppose $\|\gamma_2\|_{S'}\le q_2\le q_1^{O_{\Omega}(1)}$ and $q_2\gg_{\Omega} 1$. Then there is $g\in \bbh(\overline{k'})$ such that 
 $h(g)\ll_{\Omega} \log q_2$, where $h(g)$ is the logarithmic height of $g$, and
$V_{q_1,q_2}(\gamma_1,\gamma_2)\subseteq V(f_{g,\gamma_2})$.
\end{lem}
\begin{proof}
Let $Y:=\{g\in\bbh(\overline{k'})|\h f_s(g\gamma\gamma_2)=0\h\text{for any}\h\gamma\in V_{q_1,q_2}(\gamma_1,\gamma_2)\}$. We would like to prove that there is $g\in Y$ such that $h(g)\ll_{\Omega} \log q_2$. First we prove that $Y\not=\varnothing$. 

Let $Q_1(\underline{X}),\ldots,Q_m(\underline{X})\in \ocal_{k'}(S')[X_{11},\ldots,X_{nn}]$ be such that $\ocal_{k'}(S')[\cal]=\ocal_{k'}(S')[X_{11},\ldots,X_{nn}]/\langle Q_i\rangle$. If $Y=\varnothing$, then by the effective Nullstellensatz~\cite[Theorem IV]{MW} and the fact that the logarithmic height of $f_s(\underline{X}\gamma \gamma_2)$ is $O_{\dim \bbh}(\log q_2)$ we have that there are 
\[
P_{\gamma}(\underline{X}), P_1(\underline{X}),\ldots,P_m(\underline{X})\in \ocal_{k'}[X_{11},\ldots,X_{nn}] \text{ and } d \in \ocal_{k'}
\]
such that
\begin{enumerate}
	\item $d=\sum_{\gamma\in V_{q_1,q_2}(\gamma_1,\gamma_2)}P_{\gamma}(\underline{X})f_s(\underline{X}\gamma\gamma_2)+\sum_iP_i(\underline{X})Q_i(\underline{X})$,
	\item $\deg P_{\gamma},\deg P_i\ll O_{\Omega}(1)$,
	\item $h(d) \ll_{\dim \bbh} \log q_2$. 
\end{enumerate}
Therefore $d=\sum_{\gamma\in V_{q_1,q_2}(\gamma_1,\gamma_2)} P_{\gamma}(\gamma_1)f_s(\gamma_1\gamma\gamma_2)$ and so $d\in q_1 \ocal_{k'}(S')$. Therefore 
\[
\log |d|_{\nu}\le \log |q_1|_{\nu}\ll_{\deg k'} - \log q_1
\]
 as $\nu\not\in S'$. Thus we have 
\begin{align*}
\log q_2\gg_{\dim \bbh} h(d)&=\sum_{\nu'\in V_{\infty}(k')} \log^+ |d|_{\nu'}
\ge \sum_{\nu'\in V_{\infty}(k')} \log |d|_{\nu'}\\
\text{ (the product formula) }&=-\sum_{\nu'\in V_f(k')} \log |d|_{\nu'}\ge -\log |d|_{\nu} \gg_{\deg k'} \log q_1.
\end{align*}

This contradicts the assumption $q_2\le q_1^{O_{\Omega}(1)}$. And so $Y\neq \varnothing$.

Now by Proposition~\ref{p:SmallSolution} we know $Y$ has a point with {\em small} height, i.e. there is $g\in \bbh(\qbar)$ such that $V_{q_1,q_2}(\gamma_1,\gamma_2)\subseteq  V(f_{g,\gamma_2})$ and $H(g)\le q_2^{O_{\Omega}(1)}$.
\end{proof}

\begin{cor}\label{c:AuxqSingular}
Let $\gamma_1,\gamma_2\in \Gamma$. Suppose $\|\gamma_2\|_{S'}\le q_2\le q_1^{O_{\Omega}(1)}$ and $q_2\gg_{\Omega} 1$. Then 
\[
\pcal_{\Omega}^{(l)}(V_{q_1,q_2}(\gamma_1,\gamma_2))\ll q_2^{-O_{\Omega}(1)},
\]
where $l\gg_{\Omega} \log q_2$. 
\end{cor}
\begin{proof}
By Lemma~\ref{l:SmallLiftq-nonRegularSemisimpleProperVariety}, there is $g\in \bbh(\overline{k'})$ such that $V_{q_1,q_2}(\gamma_1,\gamma_2)\subseteq V(f_{g,\gamma_2})$ and $h(g)\ll_{\Omega} \log q_2$. By Corollary~\ref{c:RamifiedPrime} we have $p_0(V(f_{g,\gamma_2}))\ll_{\Omega} \log q_2$. Therefore by Proposition~\ref{p:EscapeSubvariety} we have 
\[
\pcal_{\Omega}^{(l)}(V_{q_1,q_2}(\gamma_1,\gamma_2))\le \pcal_{\Omega}^{(l)}(V(f_{g,\gamma_2}))\ll q_2^{-O_{\Omega}(1)}.
\] 
\end{proof}
\begin{proof}[Proof of Proposition~\ref{p:q-nonRegularSemisimple}] 
We have that 
\begin{align}
\pcal_{\Omega}^{(l)}(W_{a,q})&=\pi_q[\pcal_{\Omega}^{(l)}](\overline{W}_{a,q})=\sum_{\overline{\gamma}\in \pi_q(\Gamma)}(\pi_q[\pcal_{\Omega}]^{(l-l_0)}(\overline{\gamma}))(\pi_q[\pcal_{\Omega}]^{(l_0)}(\overline{\gamma}^{-1}\overline{W}_{a,q}))\\
& \le \max_{\overline{\gamma}\in \pi_q(\Gamma)}\pi_q[\pcal_{\Omega}]^{(l_0)}(\overline{\gamma}^{-1}\overline{W}_{a,q})
\end{align}
if $l_0\le l$. So it is enough to prove 
\be\label{e:q-regular}
\pcal_{\Omega}^{(l_0)}(\gamma^{-1}W_{a,q})\le Q^{-O_{\Omega}(\varepsilon)}
\ee
for $\vare \log Q\ll l_0\ll \vare \log Q$. 

Since $\gamma^{-1}W_{a,q}\cap \supp(\pcal^{(l_0)})\subseteq V_{q,\lfloor Q^{O_{\Omega}(\varepsilon)}\rfloor}(\gamma,a)$, Corollary~\ref{c:AuxqSingular} implies the desired result.
\end{proof}

\subsection{Finding a torus with lots of $p$-adically large elements in $A.A$.}\label{ss:LotsSimultanouslyAdDiagModq1}

For the rest of this section, let $\Omega$ and $\Gamma$ be as in Section~\ref{ss:InitialReductions}, and $\gcal_1$, $\gcal_{1,\pfr}$, and $\pi_{p^n}$ bez as in Section   \ref{ss:StatementApproximateSubgroupBoundedGeneration}. Let $\pfr\in V_f(k)\setminus S$ which divides $p\in V_f(\bbq)$. Let the quadruple  of a group scheme $\cal$, a number field $k'$, a place $\nu\in V_f(k')$, and a finite subset $S'$ of $V_f(k')$ be either $(\gcal_1,\bbq,p,S_0)$ if $p\not\in S_0$, or $(\gcal_{1,\pfr},k(\pfr),\pfr,S_{\pfr})$ where $S_{\pfr}$ is the set of restrictions of $S$ to $k(\pfr)$ if $p\in S_0$. Notice that in either case we have $k'_{\nu}=\bbq_p$ and $\Gamma\subseteq \cal(\ocal_{k'}(S'))$. Let $\bbh$ be the generic fiber of $\cal$. So $\Ad(\bbh)=\oplus_i \bbh_i$ where $\bbh_i$ is a $k'$-simple $k'$-group, and $\Lie(\bbh)=\oplus_i \Lie(\bbh_i)$. Let ${\rm pr}_i$ be the projection either from $\Ad(\bbh)$ onto $\bbh_i$, or from $\Lie(\bbh)$ onto $\Lie(\bbh_i)$. 

The main goal of this section is to show Proposition~\ref{p:LargeIntersectionMaximalTorus}. 

\begin{prop}\label{p:LargeIntersectionMaximalTorus}
In the above setting, suppose $0<\vare_2\ll_{\Omega} \vare_1\ll_{\Omega} 1$. Let $N$ be a positive integer such that $1\ll_{\Omega} N\vare_2$, and $Q=p^N$. Fix a simple factor $\bbh':=\bbh_{i_0}$ and let ${\rm pr}:={\rm pr}_{i_0}$. Then there are a positive real number $\delta$, positive integers $n_1=\Theta_{\Omega, i_0}(\vare_1N)$ (level) and $n_2=\Theta_{\Omega, i_0}(\vare_1N)$ (thickness) where the following holds:

If $\Pfr_Q(\delta,A,l)$ holds, then there are a maximal $k'$-torus $\bbt$ of $\bbh$ and $\bfr\subseteq \Lie(\cal)(\bbz_p)\cap \Lie(\bbt)(k')$ such that
\begin{enumerate}
\item $\pi_{p^{n_2}}({\rm pr}(\bfr))\subseteq \lin{p^{n_1}}{p^{n_1+n_2}}(\pi_{p^{n_1+n_2}}({\rm pr}(\prod_8 A))\cap \pi_{p^{n_1+n_2}}(\Gamma[p^{n_1}])),$
\item $Q^{\Theta_{\Omega}(\varepsilon_1)}\le |{\rm pr}(\bfr)|=|{\rm pr}(\pi_{p^{n_2}}(\bfr))|$,
\item For any $0\le m\le n_2,\h p^{\Theta_{\Omega}(m)}\le |{\rm pr}(\pi_{p^m}(\bfr))|$.
\end{enumerate}
 Moreover, let $n_3=\Theta_{\Omega, i_0}(\vare_2N)$ (auxiliary number to get a $p^{-n_3}$-regular semisimple element). Then there is a Galois extension $k''$ of $k'$ of degree at most $n!$ such that $\bbt$ splits over $k''$ and for any $\nu'\in V_f(k'')$ that divides $\nu$ there are $x_{\phi}\in \hfr(\ocal_{\nu'})\cap \hfr_{\phi}((k'')_{\nu'})$ such that
\begin{enumerate}
\item $\ocal_{\nu'}x_{\phi}=\hfr(\ocal_{\nu'})\cap \hfr_{\phi}(k''_{\nu'})$,
\item $p^{n_3}\hfr(\ocal_{\nu'})\subseteq (\hfr(\ocal_{\nu'})\cap \tfr(k''_{\nu'}))\oplus \bigoplus_{\phi\in\Phi(\bbh,\bbt)} \ocal_{\nu'} x_{\phi},$
\end{enumerate}
where $\ocal_{\nu'}$ is the ring of integers of $k''_{\nu'}$.
\end{prop}

To this end, first we find a $q^{-1}$-regular semisimple element with {\em small} $q$ and {\em large} centralizer in Proposition~\ref{p:q-regularLargeCentralizer}. 

\begin{lem}\label{l:LargeSetNonsingular}
In the above setting, suppose $0<\delta\ll\vare\ll 1$ and $n, N$ are  positive integers such that $1\ll_{\Omega} N\vare\ll_{\Omega} n$. Let $Q:=p^N$ and $q:=p^n$. Also assume that $\Pfr_Q(\delta,A,l)$ holds and $\{a_i\}_{i=1}^{O_{\Omega}(1)}\subseteq \prod_{O_{\Omega}(1)} A$ is as in Proposition~\ref{p:LargeNumberOfConjugacyClasses}. Then 
\[
\pcal_{\Omega}^{(2l_0)}(A.A\setminus \bigcup_{i=1}^{O_{\Omega}(1)} W_{a_i,q})\ge Q^{-2\delta}/2,
\]
for any $\vare\log Q\ll_{\Omega} l_0\le l.$
\end{lem}
\begin{proof}
Since $\pcal_{\Omega}^{(l)}(A)\ge Q^{-\delta}$, by Lemma~\ref{l:BGRemark} we have $\pcal_{\Omega}^{(2l_0)}(A.A)\ge Q^{-2\delta}$. Hence by Proposition~\ref{p:q-nonRegularSemisimple}, for $\delta\ll_{\Omega} \vare$ and $1\ll_{\Omega} Q^{\vare}$ we have
\[
\pcal_{\Omega}^{(2l_0)}(A.A\setminus \bigcup_{i=1}^{O_{\Omega}(1)} W_{a_i,q})\ge Q^{-2\delta}-O_{\Omega}(1)Q^{-O_{\Omega}(\vare)}\ge Q^{-2\delta}/2.
\] 
\end{proof}
\begin{prop}\label{p:q-regularLargeCentralizer}
In the above setting, suppose $0<\vare_2\ll_{\Omega} \vare_1\ll_{\Omega} 1$ and $N$ is a positive integer such that $1\ll_{\Omega} \vare_2 N$. Let $Q:=p^N$. Then there is a positive real number $\delta$ where the following holds: 

Let $q_1=p^{m_1}$ and $q_2=p^{m_1+m_2}$ such that $m_1=\Theta_{\Omega}(\vare_1 N)$ and $m_2=\Theta_{\Omega}(\vare_2 N)$. If $\Pfr_Q(\delta,A,l)$ holds, then there is $x\in \prod_{O_{\Omega,\vare_1}(1)}A$ such that
\begin{enumerate}
	\item $x$ is $|p^{m_2}|_{\nu}$-regular semisimple.
	\item $|C_{\pi_{q_2}(\Gamma)}(\pi_{q_2}(x))\cap \pi_{q_2}(A.A)|\ge Q^{\Theta_{\Omega}(\vare_1)}$.
\end{enumerate}
\end{prop}
\begin{proof}
We can and will assume that $\delta\ll_{\Omega} \varepsilon_2$ and so there is $\{a_i\}_{i=1}^{O_{\Omega}(1)}\subseteq \prod_{O_{\Omega}(1)} A$ as in Proposition~\ref{p:LargeNumberOfConjugacyClasses}.
Let $X=(A.A\setminus \bigcup_i W_{a_i,p^{m_2}})\cap B_{2\lceil  \vare_1 \log Q\rceil}$. Now we know the following properties of $X$:
\begin{enumerate}
	\item By Lemma~\ref{l:LargeSetNonsingular} we have $\pcal_{\Omega}^{(2\lceil\vare_1\log Q\rceil)}(X)\ge Q^{-2\delta}/2$. So by the Kesten bound we have $|X|\ge Q^{\Theta_{\Omega}(\vare_1)}$.
	\item The $S$-norm of any element of $X$ is at most $Q^{\Theta_{\Omega}(\vare_1)}$. So $|\pi_{q'}(X)|=|X|$ if $\log q'\gg \vare_1 \log Q$. 
	\item By Lemma \ref{l:HelfgottLargeCentralizer}, for some $x'\in X$, some index $i$, and some $q'\ge q_2 Q^{-\vare_1}$, we have that
\[
|C_{\pi_{q_2}(\Gamma)}(\pi_{q_2}(x'a_i))\cap \pi_{q_2}(A.A)|\ge |\pi_Q(A)|^{-\Theta_{\Omega,\vare_1}(\delta)}|\pi_{q'}(X)|^{\Theta_{\Omega}(1)}.
\]
\end{enumerate}
Hence if we choose $\delta$ small enough depending on $\varepsilon_1$ and $\varepsilon_2$, then we have
\[
|C_{\pi_{q_2}(\Gamma)}(\pi_{q_2}(x'a_i))\cap \pi_{q_2}(A.A)|\ge Q^{\Theta_{\Omega}(\varepsilon_1)},
\]
for some $x'\in X$. Since $x'\in X$, $x'a_i$ is a $|p^{m_2}|_{\nu}$-regular semisimple element.
\end{proof} 
 
\begin{lem}[Going to $k'$-simple factors]\label{l:ProjectionSimpleFactors}
In the above setting, if $\Pfr_Q(\delta,A,l)$ holds for a positive integer $Q$ and $1\gg_{\Omega} \delta>0$, then $\Pfr_Q(\Theta_{\Omega}(\delta),{\rm pr}_i(\Ad(A\cdot A)),l)$ holds for any $i$.
\end{lem}
\begin{proof}
As it was discussed in Section~\ref{ss:StrongApproximation}, $R_{k(\pfr)/\bbq}(\bbg_{1,\pfr})$ is naturally isomorphic to $\bbg_1$.
 And so, if $\bbg'$ is a $k(\pfr)$-simple factor of $\Ad(\bbg_{1,\pfr})$, then $R_{k(\pfr)/\bbq}(\bbg')$ is naturally isomorphic to a $\bbq$-simple factor of $\bbg_1$. 
 Therefore, by Proposition~\ref{p:EscapeSubgroups}, we have that the ${\rm pr'}(\Ad(\overline{\Omega}))$ freely generates a Zariski-dense subgroup of $R_{k(\pfr)/\bbq}(\bbg')(\bbq)$. 
 Thus ${\rm pr}'(\Ad(\overline{\Omega}))$ freely generates a Zariski-dense subgroup of $\bbg'(k(\pfr))$. 
 Hence, for any index $i$, $\overline{\Omega}_i:={\rm pr}_i(\Ad(\overline{\Omega}))$ freely generates a Zariski-dense subgroup of $\bbh_i(k')$. 
 Let $\Omega_i:=\overline{\Omega}_i\cup \overline{\Omega}_i^{-1}$. Since $\pcal_{\Omega}^{(l)}(A)\ge Q^{-\delta}$ and the restriction of ${\rm pr}_i\circ \Ad$ to $\Gamma$ is one-to-one, we have $\pcal_{\Omega_i}^{(l)}(A_i)\ge Q^{-\delta}$. 
 So by Lemma~\ref{l:BGRemark} we have that $\pcal_{\Omega_i}^{(\lfloor \Theta_{\Omega}(\log Q)\rfloor)}(A_i\cdot A_i)\ge Q^{-2\delta}$. Hence by the Kesten bound we have $|A_i\cdot A_i\cap B_{\lfloor \Theta_{\Omega}(\log Q)\rfloor}|\ge Q^{\Theta_{\Omega}(1)-2\delta}\ge Q^{\Theta_{\Omega}(1)}$. 
 For a suitable (implied) constant we have also $|A_i\cdot A_i\cap B_{\lfloor \Theta_{\Omega}(\log Q)\rfloor}|=|\pi_Q(A_i\cdot A_i\cap B_{\lfloor \Theta_{\Omega}(\log Q)\rfloor})|$. 
 Thus overall we have that
\begin{enumerate}
	\item $\pcal_{\Omega_i}^{(l)}(A_i)\ge Q^{-\delta}$, and
	\item $|\pi_Q(A_i\cdot A_i)|\ge Q^{\Theta_{\Omega}(1)}\ge |\pi_Q(A)|^{\Theta_{\Omega}(1)}$.
\end{enumerate}
Hence by the Ruzsa inequality we have that $|\pi_Q(\prod_3 (A_i\cdot A_i))|\le |\pi_Q(A_i\cdot A_i)|^{1+\Theta_{\Omega}(\delta)}$, which implies that $\Pfr_Q(\Theta_{\Omega}(\delta),A_i\cdot A_i,l)$ holds for any $i$.
\end{proof}
\begin{proof}[Proof of Proposition~\ref{p:LargeIntersectionMaximalTorus}]
By Lemma~\ref{l:ProjectionSimpleFactors} and the virtue of the proof of Proposition~\ref{p:q-regularLargeCentralizer}, we have that there is $x_i\in \prod_{O_{\Omega,\vare_2}(1)}A$ such that
\begin{enumerate}
	\item $x_i$ is $|p^{n_3}|_{\nu}$-regular semisimple.
	\item $|C_{\pi_{q'_2}(\rho_i(\Gamma))}(\pi_{q'_2}(\rho_i(x_i)))\cap \pi_{q'_2}(\rho_i(\prod_2 (A.A)))|\ge Q^{\Theta_{\Omega}(\vare_1)}$, where $\rho_i={\rm pr}_i\circ\Ad$ and $q'_2=Q^{\Theta_{\Omega}(\vare_1)}$.
\end{enumerate}
Using the $p$-adic topology on $\rho_i(\Gamma)$ we can view $\pi_{q'_2}(\rho_i(\Gamma))$ as a subset of a rooted regular tree $T_{|\f_{\nu}|^{\dim\bbh},\log_p q'_2}$. Hence by regularization, see  Corollary~\ref{c:Regularization}, there are a positive integer $n_1=\Theta_{\Omega}(\vare_1 N)$ and 
\[\textstyle
\overline{B}_i\subseteq C_{\pi_{q'_2}(\rho_i(\Gamma))}(\pi_{q'_2}(\rho_i(x_i)))\cap \pi_{q' _2}(\rho_i(\prod_2 (A.A)))
\]
 such that $|\pi_{p^{n_1}}(\overline{B}_i)|=1$ and $|\pi_{p^l}(\overline{B}_i)|\gg_{\Omega} p^{\Theta_{\Omega}(l-n_1)}$ for any $n_1\le l\le \log_p q'_2$; in particular $|\overline{B}_i|\gg Q^{\Theta_{\Omega}(\vare_1)}$. 
 
Since $x_i$ is $|p^{n_3}|_{\nu}$-regular semisimple, $\rho_i(x_i)$ is also $|p^{n_3}|_{\nu}$-regular semisimple. Now, let's apply Lemma~\ref{l:q-RegularSemisimpleCentralizer} to $q_0:=p^{n_3}$, $q:=p^{n_3}$, $q_1:=p^{n_1}$, and $q_2:=p^{n_1+n_2}$ where 
\[
n_2:=\min\{n_1-4n_3,\log_p q_2'-n_1\}.
\]
We notice that 
\[
\log_p q_0=n_3=\Theta_{\Omega}(N \vare_2)\gg_{\Omega} 1,
\]
 and,
\[
\log_p (q_1^2/(q_0 q^2 q_2))= 2n_1-(n_3+2n_3+n_1+n_2)\ge n_3
\gg_{\Omega} 1.
\]
So the hypotheses of Lemma~\ref{l:q-RegularSemisimpleCentralizer} hold. Therefore we have
\[
\Psi^{q_2}_{q_1}(\pi_{q_2}(C_{\pi_{q_0 q^2 q_2}(\rho_i(\Gamma))}(\rho_i(x_i))\cap \pi_{q_0 q^2 q_2}(\rho_i(\Gamma)[q_1])))\subseteq \pi_{q_2/q_1}(\hfr_i(\bbz_p)\cap \overline{\tfr}_i(k')),
\]
 where $\overline{\tfr}_i=\Lie C_{\bbh_i}(\rho_i(x_i))$ and $\hfr_i:=\Lie(\cal_i)$ where $\cal_i$ is the closure of $\bbh_i$ in $\cal$. Since $|\pi_{q_1}(\overline{B}_i)|=1$, $\overline{B}_i\cdot \overline{B}_i^{-1}\subseteq \rho_i(\Gamma)[q_1]$. Thus we have
\[
 \lin{q_1}{q_2}(\pi_{q_2}(\overline{B}_i\cdot \overline{B}_i^{-1}))\subseteq \pi_{q_2/q_1}(\hfr_i(\bbz_p)\cap \overline{\tfr}_i(k')).
\]
 On the other hand, since $\rho_i(x_i)$ is a regular semisimple element, we have ${\rm pr}_i(\Lie C_{\bbh}(x_i))=\Lie C_{\bbh_i}(\rho_i(x_i))$. So we have
 \be\label{e:SetPart}
  \lin{q_1}{q_2}(\pi_{q_2}(\overline{B}_i\cdot \overline{B}_i^{-1}))\subseteq \pi_{q_2/q_1}(\hfr_i(\bbz_p)\cap {\rm pr}_i(\tfr_i(k')),
 \ee
 where $\tfr_i:=\Lie C_{\bbh}(x_i)$, and $C_{\bbh}(x_i)^{\circ}$ is a $k'$-torus.
 
Since $n_1\le n_1+n_2\le \log_p q_2'$, we have $|\pi_{q_2}(\overline{B}_i)|\gg_{\Omega} p^{\Theta_{\Omega}(n_2)}$. On the other hand, since $n_1=\Theta_{\Omega}(N\vare_1)$ and $\log_p q_2'-n_1=\Theta_{\Omega}(N\vare_1)$, we have
 \[
n_2=\min\{\log_p q_2'-n_1,n_1-4n_3\}=\Theta_{\Omega}(N\vare_1)
 \]
if $\vare_2\ll_{\Omega} \vare_1$ for suitable constant. 

The way we chose $B_i$ implies that  
\begin{enumerate}
\item For any $q_1|q|q_2$ we have $(q/q_1)^{\Theta_{\Omega}(1)}\ll_{\Omega} |\pi_q(\overline{B}_i\cdot \overline{B}_i^{-1})|=|\Psi_{q_1}^{q}(\pi_q(\overline{B}_i\cdot \overline{B}_i^{-1}))|$.
\item $Q^{\Theta_{\Omega}(\varepsilon_1)}\ll |\Psi_{q_1}^{q_2}(\pi_{q_2}(\overline{B}_i\cdot \overline{B}_i^{-1}))|$.
\end{enumerate}
Let $\bfr\subseteq \hfr_i\cap {\rm pr}_i(\tfr_i(k'))$ be such that $\pi_{q_2/q_1}$ induces a bijection between $\bfr$ and $\Psi_{q_1}^{q_2}(\pi_{q_2}(\overline{B}_i\cdot \overline{B}_i^{-1}))$. By (\ref{e:SetPart}), we know that there is such $\bfr$, and, by the discussed properties of $B_i$, $\bfr$ satisfies the claimed properties.

Now one can find the desired $k''$ and $x_{\phi}$'s by Lemma~\ref{l:q-regularBasics}, Part (\ref{item}).
\end{proof}

\subsection{Proof of Proposition \ref{p:ThickTopSlice}}\label{ss:ThickTopLevel}
As in the previous section, let $\Omega$ and $\Gamma$ be as in Section~\ref{ss:InitialReductions}, and $\gcal_1$, $\gcal_{1,\pfr}$, and $\pi_{p^n}$ be as in Section   \ref{ss:StatementApproximateSubgroupBoundedGeneration}. Let $\pfr\in V_f(k)\setminus S$ which divides $p\in V_f(\bbq)$. Let the quadruple  of a group scheme $\cal$, a number field $k'$, a place $\nu\in V_f(k')$, and a finite subset $S'$ of $V_f(k')$ be either $(\gcal_1,\bbq,p,S_0)$ if $p\not\in S_0$, or $(\gcal_{1,\pfr},k(\pfr),\pfr,S_{\pfr})$ if $p\in S_0$. In particular, $\Gamma\subseteq \cal(\ocal_{k'}(S'))$. Let $\bbh$ be the generic fiber of $\cal$. So $\Ad(\bbh)=\oplus_i \bbh_i$ where $\bbh_i$ is a $k'$-simple $k'$-group, and $\Lie(\bbh)=\oplus_i \Lie(\bbh_i)$. Let ${\rm pr}_i$ be the projection either from $\Ad(\bbh)$ onto $\bbh_i$, or from $\Lie(\bbh)$ onto $\Lie(\bbh_i)$.

For a given $\varepsilon_1$ and $\varepsilon_2$ (we consider $\vare_1$ and $\vare_2$ up to constants that depend only on $\Omega$ and it will be clear what the conditions of the implied constants are), let $\delta$ be the real number given by Proposition~\ref{p:LargeIntersectionMaximalTorus}. We use Proposition~\ref{p:LargeIntersectionMaximalTorus} for any $k'$-simple factor $\bbh_i$ of $\Ad(\bbh)$, and we get sets $\bfr_i$, tori $\bbt_i$, finite Galois extensions $k''_i$ of $k'$, divisors $\nu'_i$ of $\nu$, and elements $x_{\phi}^{(i)}$ that satisfy the properties of Proposition~\ref{p:LargeIntersectionMaximalTorus}. For simplicity let $K_i:=(k''_i)_{\nu_i'}$ and $\ocal_i$ be the ring of integers of $K
_i$. Let $\Phi_i$ be the set of roots of $\tfr_i$ in $d\rho_i={\rm pr}_i\circ{\rm ad}$. Let $D:\tfr_i(K_i)\rightarrow K_i^{|\Phi_i|}$, $
D(x):=(\phi(x))_{\phi\in \Phi_i},$
and $X_i:=D(\bfr_i)\subseteq {\ocal_i}^{|\Phi_i|}$.

If $D(b)\equiv D(b') \pmod{p^m {\ocal_i}^{|\Phi_i|}}$ for $b,b'\in \bfr_i$, then $\ad(b)(x_{\phi}^{(i)})\equiv \ad(b')(x_{\phi}^{(i)}) \pmod{p^m \hfr_i(\ocal_i)}$ for any $\phi\in \Phi_i$. So for any $x\in \hfr_i(\ocal_i)$ we have $p^{n_3} \ad(b)(x)\equiv p^{n_3} \ad(b')(x) \pmod{p^m\hfr_i(\ocal_i)}$, where $n_3$ is given by Proposition~\ref{p:LargeIntersectionMaximalTorus}. Hence $\pi_{p^{m'}}(b)=\pi_{p^{m'}}(b')$, where $m'=m-n_3-\Theta_{\Omega}(1)$. In particular, $|\pi_{p^{n_2'}}(X_i)|\ge Q^{\Theta_{\Omega}(\varepsilon_1)}$, where $n_2\ll_{\Omega}n_2'\le n_2$ and $n_2=\Theta_{\Omega}(\vare_1 N)$ is given in Proposition~\ref{p:LargeIntersectionMaximalTorus}. Hence by \cite[Corollary 3]{SG:SumProduct}, there are integers $m_1'$ and $m_2'=\Theta_{\Omega}(\vare_1 N)$ and ${\bf d}\in \ocal_i^{|\Phi_i|}\setminus p\ocal_i^{|\Phi_i|}$ such that $m_1'+m_2'\ll_{\Omega} n_2'$ and
\[
\textstyle\pi_{p^{m_1'+m_2'}}(p^{m_1'}\bbz\h {\bf d})\subseteq \pi_{p^{m_1'+m_2'}}(\sum_{O_{\Omega}(1)}\prod_{O_{\Omega}(1)}X_i-\sum_{O_{\Omega}(1)}\prod_{O_{\Omega}(1)}X_i).
\]
Let $\Bfr$ be a $K_i$-basis of $\hfr_i$ which consists of $\{x_{\phi}^{(i)}\}_{\phi\in \Phi_i}$ and a $K_i$-basis of ${\rm pr}_i(\tfr_i(K_i))$. For a linear map $L$ in ${\rm End}_{K_i}(\hfr_i(K_i))$, let $[L]_{\Bfr}$ be the corresponding matrix with respect to $\Bfr$. Therefore  we have:
\[
\textstyle\pi_{p^{m_1'+m_2'}}(p^{m_1'}\bbz\h \diag({\bf d}))\subseteq \pi_{p^{m_1'+m_2'}}(\sum_{O_{\Omega}(1)}\prod_{O_{\Omega}(1)}[\ad(\bfr_i)|_{\hfr_i(K_i)}]_{\Bfr}-\sum_{O_{\Omega}(1)}\prod_{O_{\Omega}(1)}[\ad(\bfr_i)|_{\hfr_i(K_i)}]_{\Bfr}).
\]
On the other hand, we know that $\hfr_i(K_i)$ has a basis $\Bfr'$ such that 
\[
[\ad(\bfr_i)|_{\hfr_i(K_i)}]_{\Bfr'}\in M_{d_i}(\bbz_p),
\]
where $d_i:=\dim_{K_i}(\hfr_i(K_i))$. Moreover by (the second part of)  Proposition~\ref{p:LargeIntersectionMaximalTorus}, by slightly enlarging $n_3$, we have that  $\pi_{p^s}([L_1]_{\Bfr}) =\pi_{p^s}([L_2]_{\Bfr})$ for $L_1,L_2\in {\rm End}_{\bbz_p}(\hfr_i(\bbz_p))$ implies $\pi_{p^{s-n_3}}([L_1]_{\Bfr'})=\pi_{p^{s-n_3}}([L_2]_{\Bfr'})$. Hence there is 
$b_i\in {\rm End}_{\bbz_p}(\hfr_i(\bbz_p))\setminus p {\rm End}_{\bbz_p}(\hfr_i(\bbz_p))$ which is in the $\bbq_p$-algebra generated by $\ad(\bfr_i)|_{\hfr_i(K_i)}$ such that 
\be\label{e:UseSumProduct}
\textstyle\pi_{p^{m_1'+m_2'-n_3}}(p^{m_1'}\bbz_p b_i)\subseteq \pi_{p^{m_1'+m_2'-n_3}}(\sum_{O_{\Omega}(1)}\prod_{O_{\Omega}(1)}d \rho_i(\bfr_i)-\sum_{O_{\Omega}(1)}\prod_{O_{\Omega}(1)}d\rho_i(\bfr_i)).
\ee
Since $b_i$ is in the $\bbq_p$-algebra generated by $\ad(\bfr_i)|_{\hfr_i(K_i)}$, we have that $b_i(x_{\phi}^{(i)})=b_{i,\phi} x_{\phi}^{(i)}$ for any $\phi\in \Phi_i$. 

\begin{lem}\label{l:UseSumProduct}
If $t_j\equiv 1+p^my_j \pmod{p^{2m}}$ for $1\le j\le N$, $\gamma\equiv 1+p^k\xi\pmod{p^{2k}}$ and $k\ge m$, then
\[
C_{t_1}\circ C_{t_2} \circ \cdots \circ C_{t_N} (\gamma)\equiv 1+p^{k+Nm} \ad(y_1)\ad(y_2)\cdots\ad(y_N)\xi \pmod{p^{k+(N+1)m}},
\]
where $C_t(\gamma):=t\gamma t^{-1} \gamma^{-1}$.
\end{lem}
\begin{proof}
We proceed by induction on $N$. For $N=1$, we have 
\begin{align*}
C_t(\gamma)&\equiv 1+[t,\gamma] \pmod{p^{m+k+\min\{m,k\}}}\\
&\equiv 1+((1+p^my')(1+p^k\xi')-(1+p^k\xi')(1+p^my')) \pmod{p^{2m+k}}\\
&\equiv 1+p^{m+k}[y',\xi'] \pmod{p^{2m+k}}\\
&\equiv 1+p^{m+k}[y,\xi] \pmod{p^{2m+k}}\hspace{2cm}\text{(since $y'\equiv y \pmod{p^m}$, $\xi'\equiv\xi \pmod{p^k}$.)}\\
&\equiv 1+p^{m+k}\ad(y)\xi \pmod{p^{2m+k}}.
\end{align*}
By a similar argument the inductive step can be proved.
\end{proof}
\begin{cor}[Thick segment, conditional]\label{c:UseSumProduct}
In the above setting, suppose $0< \vare_2\ll_{\Omega}\vare_1\ll_{\Omega}1$. Let $N$ be a positive integer such that $1\ll_{\Omega} N\vare_2$, and $Q=P^N$. Then there are $\delta>0$, positive integers $n_1=\Theta_{\Omega,i}(\vare_1N)$ and $n_2=\Theta_{\Omega,i}(\vare_1N)$, and a positive integer $m_1$ such that the following holds:

If $\Pfr_Q(\delta,A,l)$ holds, then there is $b_i\in {\rm End}_{\bbz_p}(\hfr_i(\bbz_p))\setminus p\h {\rm End}_{\bbz_p}(\hfr_i(\bbz_p))$  such that for any $s\ge n_2$ and $\xi_0\in \Psi^{p^{2s}}_{p^s}(\prod_{O_{\Omega}(1)}A)$ we have:
\begin{enumerate}
\item $n_2-m_1=\Theta_{\Omega}(\varepsilon_1 N)$,
\item $\pi_{p^{n_2}}(p^{m_1}\bbz_p b_i({\rm pr}_i(\xi_0)))\subseteq {\rm pr}_i\left(\linValue{p^{s+n_1'}}{p^{s+n_1'+n_2}}{\Gamma[p^{s+n_1'}]\cap \prod_{O_{\Omega}(1)}A}\right)$, where $n_1'=\Theta_{\Omega}(n_1)$.
\end{enumerate}
\end{cor}
\begin{remark}\label{re:OK}
	The parameter $n_1$ will be the same as in Proposition~\ref{p:LargeIntersectionMaximalTorus}. The parameter $n_2$ is going to be smaller than the one in Proposition~\ref{p:LargeIntersectionMaximalTorus}. Hence the {\em new} $n_2$ satisfies the claims of Proposition~\ref{p:LargeIntersectionMaximalTorus}, too.
\end{remark}
\begin{proof}[Proof of Corollary~\ref{c:UseSumProduct}]
For now suppose $n_1,n_2$ and $\bfr_i$ are the ones given by 	Proposition~\ref{p:LargeIntersectionMaximalTorus}. Hence we have 
\[
\textstyle\pi_{p^{n_1+n_2}}(1+p^{n_1}\bfr_i) \subseteq \pi_{p^{n_1+n_2}}(\prod_{O_{\Omega}(1)}A).
\]
Hence by Lemma~\ref{l:UseSumProduct} (applied to the parameters $m=\lfloor(n_1+n_2)/2\rfloor$, $k=s$, $y_j\in p^{\lceil (n_1-n_2)/2\rceil}\bfr_i$, $\xi=\xi_0$, and $N=\Theta_{\Omega}(1)$) we have that
\be\label{e:SumProd}
\textstyle\pi_{p^{s+n_1'+n_2}}\left(1+p^{s+n_1'}\left(\sum_{O_{\Omega}(1)}\prod_{O_{\Omega}(1)}\ad(\bfr_i)-\sum_{O_{\Omega}(1)}\prod_{O_{\Omega}(1)}\ad(\bfr_i)\right)(\xi_0)\right)\subseteq \pi_{p^{s+n_1'+n_2}}(\prod_{O_{\Omega}(1)}A),
\ee
where $n_1'=\Theta_{\Omega}(n_1)$. Therefore 
\[\textstyle
\pi_{p^{n_2}}\left(\sum_{O_{\Omega}(1)}\prod_{O_{\Omega}(1)}d\rho_i(\bfr_i)-\sum_{O_{\Omega}(1)}\prod_{O_{\Omega}(1)}d\rho_i(\bfr_i)\right)(\xi_0)\subseteq
{\rm pr}_i\left(\Psi^{p^{s+n_1'+n_2}}_{p^{s+n_1'}}(\Gamma[p^{s+n_1'}]\cap \prod_{O_{\Omega}(1)}A)\right).
\]

On the other hand, by (\ref{e:UseSumProduct}), there are  $b_i\in {\rm End}_{\bbz_p}(\hfr_i(\bbz_p))\setminus p\h {\rm End}_{\bbz_p}(\hfr_i(\bbz_p))$ and positive integers $m_1',m_2',n_3$ such that 
\begin{enumerate}
	\item $m_1'+m_2'\ll_{\Omega} n_2$, $m_2'=\Theta_{\Omega}(n_2)$, and $n_3=\Theta_{\Omega}(\varepsilon_2 N)$,
	\item $\textstyle\pi_{p^{m_1'+m_2'-n_3}}(p^{m_1'}\bbz_p b_i)\subseteq \pi_{p^{m_1'+m_2'-n_3}}(\sum_{O_{\Omega}(1)}\prod_{O_{\Omega}(1)}d \rho_i(\bfr_i)-\sum_{O_{\Omega}(1)}\prod_{O_{\Omega}(1)}d\rho_i(\bfr_i)).$
\end{enumerate}
Hence the quadruple $(n_1,m_1'+m_2'-n_3,m_1',b_i)$ satisfies the claimed properties for $(n_1,n_2,m_1,b_i)$.
\end{proof}
Corollary~\ref{c:UseSumProduct} gives us a {\em thick segment} only when $b_i({\rm pr}_i(\xi_0))$ has a large $p$-adic norm, which happens if $\xi_0$ is away from $\bigcup_{\phi\in \Phi_i} W_{\phi}$ where $W_{\phi}:=\tfr_i\oplus\bigoplus_{\phi'\neq\phi} \hfr_{\phi'}$.

\begin{lem}[Thick segment]\label{l:ThickSegment}
Let $n_1=\Theta_{\Omega}(N\vare_1)$ and $n_2=\Theta_{\Omega}(N\vare_1)$ be as in Corollary~\ref{c:UseSumProduct} (see Remark~\ref{re:OK})~\ref{p:LargeIntersectionMaximalTorus}. There are integers $m_1$ and $n_1''=\Theta_{\Omega}(n_1)$ and $\xi_0\in \hfr_i(\bbz_p)\setminus p\hfr_i(\bbz_p)$ such that $n_2-m_1=\Theta_{\Omega}(N\vare_1)$ and 
\[
\textstyle\pi_{p^{n_2}}(p^{m_1}\bbz_p \xi_0)\subseteq {\rm pr}_i\left(
\linValue{p^{n_1''}}{p^{n_1''+n_2}}{\Gamma[p^{n_1''}]\cap \prod_{O_{\Omega}(1)}A}\right).
\]
\end{lem}
\begin{proof}
Applying Proposition~\ref{p:SpanSmallElements} for $\vare:=\vare_2$, representation $\rho:=\rho_i$ (recall that $\rho_i:=\Ad\circ \Pr_i$), number field $k'$, and $S'$, we get $\{a_j\}_{j=1}^{O_{\Omega}(1)}\subseteq \prod_{O_{\Omega}(1)}A$ such that 
\be\label{e:LargeSpan}
[\bbz_p[\rho_i(\Gamma)]:\sum_j \bbz_p\rho_i(a_j)]\le Q^{\vare_2}.
\ee
 On the other hand, $\hfr_i(k')$ is a simple $\rho_i(\Gamma)$-module, and $k'_{\nu}=\bbq_p$. Hence for any $x\in \hfr_i(\bbq_p)\setminus \{0\}$ we have 
\be\label{e:Simple} 
 \hfr_i(\bbq_p)=\bbq_p[\rho_i(\Gamma)](x).
\ee
Moreover for large enough $p$ we have $\bbz_p[\rho_i(\Gamma)]$ is a maximal order of the central simple algebra $\bbq_p[\rho_i(\Gamma)]$. Thus for large enough $p$ and $x\in \hfr_i(\bbz_p)\setminus p\hfr_i(\bbz_p)$ we have $\hfr_i(\bbz_p)=\bbz_p[\rho_i(\Gamma)](x)$. For arbitrary $p$, by a compactness argument similar to~\cite[Lemma 3.5]{SGChar}, we have:

{\bf Claim 1.} For $x\in \hfr_i(\bbz_p)\setminus p \hfr_i(\bbz_p)$,  $[\hfr_i(\bbz_p):\bbz_p[\rho_i(\Gamma)(x)]]\ll_{\Omega} 1$.

{\em Proof of Claim.} Suppose to the contrary that there are $x_j\in \hfr_i(\bbz_p)\setminus p \hfr_i(\bbz_p)$ such that $[\hfr_i(\bbz_p):\bbz_p[\rho_i(\Gamma)(x_j)]]$ goes to infinity. Since $\hfr_i(\bbz_p)\setminus p \hfr_i(\bbz_p)$ is a compact set, passing to a subsequence, if needed, we can assume that $x_j$ converges to $x\in\hfr_i(\bbz_p)\setminus p \hfr_i(\bbz_p)$ as $j$ goes to infinity. Since $\hfr_i(\bbq_p)=\bbq_p[\rho_i(\Gamma)](x)$, $\bbz_p[\rho_i(\Gamma)](x)$ is a neighborhood of $x$. As an additive group $\bbz_p[\rho_i(\Gamma)](x)$ is a finitely generated pro-$p$ group, and its Frattini subgroup is $p\bbz_p[\rho_i(\Gamma)](x)$. Suppose $\{\rho_i(\gamma_1)(x),\ldots,\rho_i(\gamma_{d_i})(x)\}$ is a set of coset representatives of the Frattini subgroup in $\bbz_p[\rho_i(\Gamma)](x)$. Since $p\bbz_p[\rho_i(\Gamma)](x)$ is an open set, for any large enough $j$ and $1\le j'\le d_i$, we have
\[
\rho_i(\gamma_{j'})(x_j)-\rho_i(\gamma_{j'})(x)\in p\bbz_p[\rho_i(\Gamma)](x).
\]
Hence, for large enough $j$, $\bbz_p[\rho_i(\Gamma)](x_j)\supseteq \bbz_p[\rho_i(\Gamma)](x)$ which contradicts the contrary assumption.

Therefore, by (\ref{e:LargeSpan}) and Claim 1, for $x\in \hfr_i(\bbz_p)\setminus p \hfr_i(\bbz_p)$, we have
\be\label{e:SmallIndex}
[\hfr_i(\bbz_p):\sum_j \bbz_p \rho_i(a_j)(x)]=[\hfr_i(\bbz_p):\bbz_p[\rho_i(\Gamma)](x)][\bbz_p[\rho_i(\Gamma)](x):\sum_j \bbz_p \rho_i(a_j)(x)]\ll_{\Omega} Q^{\vare_2}.
\ee
{\bf Claim 2.} We have $\max_j \|b_i(\rho_i(a_j)(x))\|_p \ge Q^{-\Theta_{\Omega}(\vare_2)}$ for any $x\in \hfr_i(\bbz_p)\setminus p\hfr_i(\bbz_p)$.

{\em Proof of Claim.} Suppose $\max_j \|b_i(\rho_i(a_j)(x))\|_p=|q|_p$ where $q$ is a power of $p$. So for any $j$ we have $b_i(\rho_i(a_j)(x))\in q\hfr_i(\bbz_p)$. Hence
\be\label{e:Small}
\sum_j \bbz_p b_i(\rho_i(a_j)(x)) \subseteq q \hfr_i(\bbz_p).
\ee
Therefore by (\ref{e:SmallIndex}) and (\ref{e:Small}) we have 
\[
p^{\Theta_{\Omega}(N\vare_2)} \hfr_i(\bbz_p) \subseteq \sum_j \bbz_p b_i(\rho_i(a_j)(x)) \subseteq q \hfr_i(\bbz_p),
\]
which implies that $q\le Q^{\Theta_{\Omega}(\vare_2)}$, and it gives us the claim. 

By Proposition~\ref{p:LargeIntersectionMaximalTorus}, there is $\gamma\in \prod_8 A\cap (\Gamma[p^{n_1}]\setminus \Gamma[p^{n_1+n_2}])$. Hence
there are  $x\in \hfr_i(\bbz_p)\setminus p\hfr_i(\bbz_p)$ and $n_1\le k< 2n_1$ such that $\pi_{p^k}(x)=\linValue{p^k}{p^{2k}}{\gamma}$. Therefore by Lemma~\ref{l:PropertiesOfLinearizationMap} and our choice of $\{a_j\}_{j=1}^{O_{\Omega}(1)}$ we have 
\[\textstyle
\pi_{p^{2k}}(\rho_i(a_j)(x))\in {\rm pr}_i\left(\linValue{p^k}{p^{2k}}{\Gamma[p^{k}]\cap \prod_{O_{\Omega}(1)}A}\right), 
\]
for any $i,j$, and $\rho_i(a_j)(x)\in \hfr_i(\bbz_p)\setminus p\hfr_i(\bbz_p)$. Hence by Corollary~\ref{c:UseSumProduct} applied to $\xi_0:=\rho_i(a_j)(x)$ we have
\[\textstyle
\pi_{p^{n_2}}(p^{m_1}\bbz_p b_i(\rho_i(a_j)(x))) \subseteq {\rm pr}_i\left(\linValue{p^{n_1'}}{p^{n_1'+n_2}}{\Gamma[p^{n_1'}]\cap \prod_{O_{\Omega}(1)}A}\right),
\]
where $n_1'=\Theta_{\Omega}(n_1)$ and $n_2-m_1=\Theta_{\Omega}(\vare_1 N)$. Thus by the above claim for some $j$ we have $b_i(\rho_i(a_j)(x))=p^{m_1'}\xi_0$ for some $\xi_0\in \hfr_i(\bbz_p)\setminus p\hfr_i(\bbz_p)$ and $m_1'\ll_{\Omega} \vare_2 N$. Since $n_2-m_1-m_1'=\Theta_{\Omega}(\vare_1 N)-\Theta_{\Omega}(\vare_2 N)$, for a suitable choice of the implied constant in $\vare_2\ll_{\Omega} \vare_1$ we have $n_2-m_1-m_1'=\Theta_{\Omega}(\vare_1 N)$. This completes the proof. 
\end{proof}
\begin{lem}[Thick top slice: simple factors]\label{l:ThickTopLevelProj}
For any $i$ there are positive integers $n_1^{(i)}$ (level) and $n_2^{(i)}$ (thickness) such that 
\begin{enumerate}
\item $n_1^{(i)}=\Theta_{\Omega}(\varepsilon_1n)$ and $n_2^{(i)}=\Theta_{\Omega}(\varepsilon_1n)$.
\item $\pi_{p^{n_2^{(i)}}}(\hfr_i(\bbz_p))\subseteq {\rm pr}_i\left(
\linValue{p^{n_1^{(i)}}}{p^{n_1^{(i)}+n_2^{(i)}}}{\Gamma[p^{n_1^{(i)}}]\cap \prod_{O_{\Omega}(1)}A}\right)$.

\end{enumerate}
\end{lem}
\begin{proof}
Without loss of generality, we fix $i$. Let $\{a_j\}_{j=1}^{O_{\Omega}(1)}\subseteq \prod_{O_{\Omega,\rho_i}(1)}A$ be as in Proposition~\ref{p:SpanSmallElements} and $\xi_0\in \hfr_i(\bbz_p)\setminus p\hfr_i(\bbz_p)$ be as in Lemma~\ref{l:ThickSegment}. Then after changing $\vare_2$ by a constant which just depends on $\Omega$ we have $|\hfr_i(\bbz_p)/\sum_j\bbz_p \rho_i(a_j)(\xi_0)|\le Q^{\varepsilon_2}$. Hence for some positive integer $m=\Theta_{\Omega}(\varepsilon_2 n)$ we have $p^m\hfr_i(\bbz_p)\subseteq \sum_j\bbz_p \rho_i(a_j)(\xi_0)$. Therefore by Lemma~\ref{l:ThickSegment} 
\[
\textstyle\pi_{p^{n_2}}(p^{m_1+m}\hfr_i(\bbz_p))\subseteq {\rm pr}_i\left(
\linValue{p^{n_1''}}{p^{n_1''+n_2}}{\Gamma[p^{n_1''}]\cap \prod_{O_{\Omega}(1)}A}\right),
\]
where $n_2=\Theta_{\Omega}(\varepsilon_1n)$, $n_1''=\Theta_{\Omega}(\varepsilon_1n)$, $m_1$ are as in Lemma~\ref{l:ThickSegment}. Notice that $n_2-m_1-m=\Theta_{\Omega}(\varepsilon_1 n)-\Theta_{\Omega}(\varepsilon_2 n)$ and so $n_2-m_1-m=\Theta_{\Omega}(\varepsilon_1n)$ as $\varepsilon_2\ll_{\Omega}\varepsilon_1$.
\end{proof}
Inspired by Definition~\ref{d:ThickLayer}, let us say $\bcal_i(L,T)$ holds if for some $q_1|q_2|q_1^2$ that are powers of $p$ we have 
\[
\textstyle\pi_{q_2}(\hfr_i(\bbz_p))\subseteq \linValue{q_2}{q_1}{\Gamma[q_1]\cap \prod_{O_{\Omega,\vare_1,\vare_2}(1)}A},
\]
and $q_1\le L$ and $LT\le q_2$. By the virtue of proofs of Lemma~\ref{l:OneStepPropagation} and Corollary~\ref{c:OneStepPropagation}, we have that 
\begin{lem}\label{l:OneStepPropagationSimple}
There is a constant $q_0$ depending only on $\Omega$ such that if $\bcal_i(L,T)$ and $\bcal_i(L',T')$ hold, $L,L'\ge T$, $L,L'\gg_{\Omega} 1$ and $\log T'=\Theta_{\Omega}(\log T)$, then $\bcal_i(q_0LL',TT'/q_0)$ holds.
\end{lem}
\begin{lem}[Leveling top slices]\label{l:Leveling}
There are positive integers $n_1$ and $n_2$ such that
\begin{enumerate}
\item$n_1=\Theta_{\Omega}(\varepsilon_1n)$ and $n_2=\Theta_{\Omega}(\varepsilon_1n)$.
\item $\pi_{p^{n_2}}(\hfr_i(\bbz_p))\subseteq {\rm pr}_i\left(\linValue{p^{n_1}}{p^{n_1+n_2}}{\Gamma[p^{n_1}]\cap \prod_{O_{\Omega,\vare_1,\vare_2}(1)}A}\right)$ for any $i$.
\end{enumerate}
In particular, we have $\pi_{p^{n_2}}(\hfr(\bbz_p))\subseteq \linValue{p^{n_1}}{p^{n_1+n_2}}{\Gamma[p^{n_1}]\cap \prod_{O_{\Omega}(1)}A}$.
\end{lem}
\begin{proof}
By Lemma~\ref{l:ThickTopLevelProj}, we know that $\bcal_i(p^{n_1^{(i)}},p^{n_2^{(i)}})$ holds for any $i$, for some integers $n_1^{(i)}=\Theta_{\Omega}(\varepsilon_1n)$ and $n_2^{(i)}=\Theta_{\Omega}(\varepsilon_1n)$. Hence by repeated use of Lemma~\ref{l:OneStepPropagationSimple} we have that $\bcal_i(p^{ln_1^{(i)}+ln_0},p^{ln_2^{(i)}-ln_0})$ holds for any $l=O_{\Omega}(1)$ and any $i$, where $|q_0|_p=p^{-n_0}$. Therefore, since $\lceil (n_1^{(i)}+n_0)/(n_2^{(i)}-n_0)\rceil^2=O_{\Omega}(1)$, for any 
\[
\left\lceil \frac{n_1^{(i)}+n_0}{n_2^{(i)}-n_0}\right\rceil \le l\le \left\lceil \frac{n_1^{(i)}+n_0}{n_2^{(i)}-n_0}\right\rceil^2,
\]
we have $\bcal_i(p^{ln_1^{(i)}+ln_0},p^{ln_2^{(i)}-ln_0})$ holds . These properties together imply that 
\[
\bcal_i(p^{l_0n_1^{(i)}+l_0n_0},p^{l_0n_1^{(i)}+l_0n_0})
\]
 holds for $l_0=\lceil (n_1^{(i)}+n_0)/(n_2^{(i)}-n_0)\rceil=O_{\Omega}(1)$. So without loss of generality we can and will assume that $n_1^{(i)}=n_2^{(i)}$. 
 
For $i\neq i'$, we level the top thick layers of the $i$-th and the $i'$-th factors. Without loss of generality, let's assume that $n_1^{(i)}\leq n_1^{(i')}$. Since $n_1^{(i)}=\Theta_{\Omega}(N\vare_1)$ and $n_1^{(i')}=\Theta_{\Omega}(N\vare_1)$, there is a positive integer $l_0\ll_{\Omega} 1$ such that $n_1^{(i')}\le 2^{l_0} n_1^{(i)}$. Hence
\be\label{e:IntervalCovering}
[n_1^{(i')},2n_1^{(i')}]\subseteq \bigcup_{j=0}^{l_0-1}[2^j n_1^{(i)}+2^{j-1} n_0, 2^{j+1} n_1^{(i)}] \cup \bigcup_{j=1}^{l_0-1}(2^j n_1^{(i)}, 2^j n_1^{(i)}+2^{j-1} n_0).
\ee
Since $N\vare_1\gg_{\Omega} 1$, $n_1^{(i)}=\Theta_{\Omega}(N\vare_1)$, and $l_0\ll_{\Omega} 1$, there is $j$ such that
\be\label{e:LargeSubinterval}
|[n_1^{(i')},2n_1^{(i')}]\cap [2^j n_1^{(i)}+2^{j-1} n_0, 2^{j+1} n_1^{(i)}]| \gg_{\Omega} N\vare_1.
\ee
On the other hand, by Lemma~\ref{l:OneStepPropagationSimple}, for any integer $0\le j\le l_0-1$ we have that 
\[
\bcal_i(p^{2^j n_1^{(i)}+2^{j-1} n_0},p^{2^j n_1^{(i)}-2^{j-1} n_0})
\]
holds. Hence if $[t_j,t_j+t'_j]:=[n_1^{(i')},2n_1^{(i')}]\cap [2^j n_1^{(i)}+2^{j-1} n_0, 2^{j+1} n_1^{(i)}]$, then 
\[
\bcal_i(p^{t_j},p^{t_j'}),\h \text{ and } \bcal_{i'}(p^{t_j},p^{t_j'})
\] 
hold. By (\ref{e:LargeSubinterval}) for some $j$ the above closed interval $[t_j, t_j+t_j']$ is thick enough, and works for both the $i$-th and the $i'$-th factors.  Now by induction on the number of simple factors, we can find $n_1$ and $n_2$ that work for all the simple factors at the same time. And therefore we get 
\[
\textstyle \pi_{p^{n_2}}(\hfr(\bbz_p))\subseteq \linValue{p^{n_1}}{p^{n_1+n_2}}{\Gamma[p^{n_1}]\cap \prod_{O_{\Omega}(1)}A}.
\] 
\end{proof}
Lemma~\ref{l:Leveling} and Lemma~\ref{l:PropertiesOfLinearizationMap} imply Proposition~\ref{p:ThickTopSlice}.

\section{Appendix A: a small solution.}
\subsection{Logarithmic height, and the statement of the main result.}\label{ss:Norms}
In this appendix, we recall the (logarithmic and multiplicative) height of a point and prove that any closed subscheme of $(\bba_n)_\bbq$ with a closed geometric point has a {\em small} closed geometric point.

For any prime $p$, a place $|.|_p$ on $\bbq$ is fixed such that $|p|_p=1/p$. For $p=\infty$, the Archimedean place $|.|_p$ of $\bbq$ is the ordinary absolute value on $\bbq$. If $k$ is a Galois number field, then a place $\pfr|p$ is normalized such that $|x|_{\pfr}:=|N_{k_{\pfr}/\bbq_p}(x)|_p^{1/[k:\bbq]}$ (this is the case for all the places, including the Archimedean place). A place $\pfr$ of $k$ is called an Archimedean place if $\pfr|\infty$. The set of all the Archimedean places is denoted by $V_{\infty}(k)$.

 and in particular for any $x\in k$ we have  the product formula $\prod_{\pfr\in \pl{k}}|x|_{\pfr}=1$  where $\pl{k}$ is the set of inequivalence places of $k$ (within the article the set of finite places is denoted by $V_f(k)$). Let $|x|_{\pfr}^+:=\max\{1,|x|_{\pfr}\}$, and for a finite place let $\|{\bf x}\|_{\pfr}:=\max_i\{|x_i|_{\pfr}\}$ and $\|{\bf x}\|_{\pfr}^+:=\max_i\{|x_i|_{\pfr}^+\}$.

Any point $v\in \bbp^n(\qbar)$ can be represented by $\xbf:=(x_0,x_1,\ldots,x_n)\in k^{n+1}$ where $k$ is a Galois number field. We define the height $h(v)$ of $v$ to be
\[
h(v):=\sum_{\pfr\in\pl{k}}\log \|\xbf\|_{\pfr}.
\]
It is well-known that $h(v)$ is independent of the choice of $k$ and the choice of a representative (e.g. see \cite[Chapter 1]{BGBook}). The height $h(\xbf)$ of $\xbf=(x_1,\ldots,x_n)\in \qbar^n$ is defined to be $h(1\!:\!x_1\!:\!\cdots\!:\!x_n)=\sum_{\pfr\in\pl{k}}\log \|\xbf\|^+_{\pfr}$. The multiplicative height of a point $\xbf$ in $\bbp(\qbar)$ or $\bba^n(\qbar)$ is defined to be $e^{h(\xbf)}$. Considering $(\mathbb{GL}_d)_{\bbq}$ as an open subset of $\underline{\rm End}_{\bbq^d}$, we can talk about the height $h(X)$ of $X\in \GL_d(\qbar)$. The height $h(f)$ of a polynomial $f\in\bbq[\underline{X}]=\bbq[X_1,\ldots,X_n]$ is defined to be 
\[
h(\sum_I a_I \underline{X}^I):=\sum_{\pfr\in\pl{k}} \log (\max_I |a_I|_{\pfr}),
\]
where $k$ is a Galois number field such that $f\in k[\underline{X}]$ and $\underline{X}^I=\prod_{i=1}^n X_i^{m_i}$ for $I=(m_1,\ldots,m_n)$.

The main result of this appendix is the following Proposition.
\begin{prop}\label{p:SmallSolution}
Let $k$ be a number field, and $f_1,\ldots,f_m\in k[X_1,\ldots,X_n]$. If there is $\xbf\in \qbar^n$ such that $f_1(\xbf)=\cdots=f_m(\xbf)=0$, then there is $\xbf_0\in \qbar^n$ such that

\begin{enumerate}
\item $f_1(\xbf_0)=\cdots=f_m(\xbf_0)=0$,
\item $h(\xbf_0)\ll \max_i h(f_i),$ where the implied constant depends on $n$, $\max_i \deg f_i$, and $k$.
\end{enumerate}
\end{prop}
\subsection{Reduction to the geometrically zero-dimensional case.} Let $V=\Spec(k[\underline{X}]/\langle f_1,\ldots,f_m\rangle)$. In this section, we reduce the  dimension of $V$ by proving that $V$ intersects a hyperplane with {\em small} height (see Lemma~\ref{l:CuttingByHyperplane}). Providing such a process implies that it suffices to prove Proposition~\ref{p:SmallSolution} under the additional assumption $\dim V=0$. 

\begin{lem}\label{l:SmallRepresentatives}
Let $k$ be a number field, $\ocal_k$ be its ring of integers, and $\afr$ be a non-zero ideal of $\ocal_k$. Then there is a set $X\subseteq \ocal_k$ such that 
\begin{enumerate}
\item $X$ is a set of coset representatives of $\afr$ in $\ocal_k$, i.e. $\ocal_k$ is a disjoint union of $x+\afr$ for $x\in X$,
\item for any $x\in X$ and a complex embedding $\sigma:k\rightarrow \bbc$ we have $|\sigma(x)| \ll_k |\ocal/\afr|$. 
\end{enumerate}
\end{lem}
\begin{proof}
Let $\sigma_1,\ldots,\sigma_{r+2s}:k\rightarrow \bbc$ be all the distinct embeddings of $k$ into $\bbc$. Suppose $\sigma_i(k)\subseteq \bbr$ for $i\le r$, and $\sigma_{r+j}$ is the complex conjugate of $\sigma_{r+s+j}$ for any $1\le j\le s$ (with the understanding that either $r$ or $s$ can be zero). It is well-known that $r+2s=d:=[k:\bbq]$, 
\[
\sigma: k \rightarrow L_{rs}:=\bbr^r\oplus \bbc^s,\h \sigma(a):=(\sigma_1(a),\ldots,\sigma_{r+s}(a))
\] 
embeds $k$ as a lattice in $L_{rs}$ and its covolume with respect to the natural Euclidean structure of $L_{rs}$ is $\sqrt{|\Delta_k|}$ where $\Delta_k$ is the discriminant of $k$ (e.g. see~\cite[Corollary 8.3]{ST}). As any non-zero ideal $\afr$ is of finite index in $\ocal_k$, we have that $\sigma(\afr)$ is a lattice in $L_{rs}$ and ${\rm vol}(L_{rs}/\sigma(\afr))=N(\afr)\sqrt{|\Delta_k|}$, where $N(\afr):=|\ocal_k/\afr|$.

Let $\lambda_1(\afr),\ldots, \lambda_d(\afr)$ be Minkowski's successive minima of $\sigma(\afr)$ in $L_{rs}$, i.e. 
\[
\lambda_i(\afr):=\min\{r\in \bbr^+ |\h B_r\cap \sigma(\afr) \text{ has $i$ elements that are }  \bbr-\text{linearly independent elements}\}, 
\]
where $B_r$ is the Euclidean ball of radius $r$ centered at the origin. 

By Minkowski's second theorem on successive minima (e.g. see \cite[Chapter 8]{Cas}) we have
\[
\prod_{i=1}^{d} \lambda_i(\afr) \le \frac{2^d}{{\rm vol}(B_1)} {\rm vol}(L_{rs}/\sigma(\afr))\ll_k N(\afr).
\]

On the other hand, $\lambda_i(\ocal_k)\le \lambda_i(\afr)$, and $\lambda_1(\afr)\le \cdots \le \lambda_d(\afr)$. Hence we have
$
\lambda_d(\afr) \ll_k N(\afr).
$ Hence $\afr$ has a $\bbz$-basis in $B_{\Theta_k(N(\afr))}$. Therefore there is a set $X$ of coset representatives of $\afr$ in $\ocal_k$ which is a subset of $B_{\Theta_k(N(\afr))}$.
\end{proof}

\begin{lem}\label{l:CuttingByHyperplane}
Let $V=\Spec(k[\underline{X}]/\langle f_1,\ldots,f_m\rangle)$. If $\dim V\ge 1$, then there is $\xbf\in k^n$ such that
\begin{enumerate}
\item $V(\qbar)\cap \underline{\ker(l_{\xbf}})(\qbar) \neq \varnothing$, where $\underline{\ker(l_{\xbf}})(\qbar)=\{\ybf\in\qbar^n|\h l_{\xbf}(\ybf):=\sum_i x_i y_i=0\}$,
\item $V(\qbar)\not\subseteq \underline{\ker(l_{\xbf}})(\qbar)$,
\item $h(\xbf)\ll \max_i h(f_i)$, where the implied constant depends on $n$ and $\max_i \deg f_i$.
\end{enumerate}
\end{lem}
\begin{proof}
Let $W(F):=\{[\xbf]\in \bbp^{n-1}(F)|\h V(\qbar)\cap \underline{\ker( l_{\xbf})}(\qbar)=\varnothing\}$, where $F$ is a subfield of $\qbar$. By the nullstellensatz theorem, $[\xbf]\in W(k)$ if and only if there are $q_0,q_i\in k[\underline{X}]$ such that 
\be\label{e:NonGenericHyperplane}
1=q_0(\underline{X}) l_{\xbf}(\underline{X})+\sum_i q_i(\underline{X}) f_i(\underline{X}).
\ee
So by the effective B\'{e}zout~\cite[Theorem IV]{MW} we can assume further that 
\be\label{e:EffectiveNonGenericDegree}
\max_i \deg(q_i)\ll 1,
\ee
and 
\be\label{e:EffectiveNonGenericHieght}
\max_i h(q_i)\ll \max\{h([\xbf]), h(f_1),\ldots, h(f_m)\},
\ee
where the implied constants depend on $n, k$ and $\max_j \deg f_j$. By (\ref{e:EffectiveNonGenericDegree}), the $q_i$ are in a finite-dimensional subspace of $k[\underline{X}]$ and so the vector $v_i$ of their coefficients are in a fixed finite-dimensional vector space. Comparing the coefficients of $\underline{X}^I$ in the both sides of Equation (\ref{e:NonGenericHyperplane}), we get a system of equations 
\be\label{e:BilinearNonGeneric}
L_i(\xbf,v_j)=n_i,
\ee
where $L_i$ are linear in $\xbf$ and $v_j$ for any $j$, $h(L_i)\ll \max_i h(f_i)$ where the implied constant depends on $n$ and $\max_i \deg(f_i)$, and $n_i$ are either 0 or 1 (they are zero except once). We consider (\ref{e:BilinearNonGeneric}) as a system of equations for the unknown $\vbf=\{v_j\}$, with parameter $\xbf$. So $[\xbf]\in W(F)$ if and only if Equation (\ref{e:BilinearNonGeneric}) has a solution with entries in $F$. Considering (\ref{e:BilinearNonGeneric}) as a system of linear equations over $k[\underline{X}]$, we get a (rectangular) matrix $A$ with entries in $k[\underline{X}]$ such that (\ref{e:BilinearNonGeneric}) has a solution if and only if $A\vbf = \ebf_1$ has a solution, where all the components of $\ebf_1$ are zero except the first one. It is well-known that there are invertible matrices $P_1$ and $P_2$ such that $A=P_1 D P_2$ where the only non-zero entries of $D$ are in the $i,i$ position for some $i$. Let $s(\underline{X})$ be the product of the denominators of the entries of $P_1^{-1}$ and $P_2$. Thus, if $s(\xbf)\neq 0$, then (\ref{e:BilinearNonGeneric}) has a solution if and only if $D(\xbf) \vbf=P_1(\xbf)^{-1}\ebf_1$ has a solution. And the latter happens if and only if $r_1(\xbf)=\cdots=r_c(\xbf)=0$ for certain polynomials $r_i(\underline{X})\in k[\underline{X}]$. Hence altogether there are polynomials $s, r_i$ such that 
\be\label{e:NonGenericHyperplaneUpper}
W(F)\setminus \{[\xbf]|\h \xbf \in V(s)(F)\}=\{[\xbf]|\h\ \xbf\in V(r_1,\ldots,r_c)(F)\setminus V(s)(F)\}.
\ee
Since $\dim V\ge 1$, $\dim W<n-1$. Hence in (\ref{e:NonGenericHyperplaneUpper}) we can assume that $r_1$ is a non-zero homogeneous polynomial. It is also easy to see that the logarithmic height of the non-zero entries of $P_1$, $P_2$ and $D$ are $O(\max_i h(f_i))$ where the implied constant depends on the size of matrices (which is a function of the number of variables $n$ and $\max_i\deg f_i$) and their degree depends just on $n$ and $\max_i\deg f_i$. Thus $h(r_1)\ll \max_i h(f_i)$ where the implied constant depends only on $n$ and $\max_i \deg f_i$. 

By Landau's prime ideal theorem, there is $\pfr\in \Spec(\ocal_k)$ such that $\max\{h(r_1),h(s)\}< \log N(\pfr)\ll \max_i h(f_i)$ for a suitable constant depending on $n, k$ and $\max_i \deg f_i$.

Let $X$ be a set of coset representatives of $\pfr$ in $\ocal_k$ which Lemma~\ref{l:SmallRepresentatives} gives us. 
Hence for any $x\in X$ we have $\log \max_{\nu\in V_{\infty}(k)} |x|_{\nu}\ll_k \log N(\pfr)=|\f_{\pfr}|$, and $\pi_{\pfr}:\ocal_k\rightarrow \f_{\pfr}$ induces a bijection between $X$ and $\pi_{\pfr}(X)$.

Then
\be\label{e:UpperBoundNumberOfBadPoints}
|\{\xbf\in X^n|\h [\xbf]\in W(k)\}|\le |\{ \xbf|\h \xbf\in V(r_1)(\f_{\pfr})\cup V(s)(\f_{\pfr})\}|\le C |\f_{\pfr}|^{n-1}. 
\ee
where the constant $C$ depends on the degrees of $r_1$ and $s$. We can further assume that $C< \log N(\pfr)$. Hence we can find a point $[\xbf]\in \bbp^n(k)$ such that
\begin{enumerate}
\item $h([\xbf])\ll \max_i h(f_i)$ where the implied constant depends only on $n$ and $\max_i\deg f_i$.
\item $[\xbf]$ is not in $W(k)$, which means $V(\qbar)\cap \underline{\ker(l_{\xbf})}(\qbar)\neq \varnothing$.
\end{enumerate}
If $V(\qbar)$ is not a subset of $\underline{\ker(l_{\xbf})}(\qbar)$, we are done. Otherwise, we can repeat the above argument after taking a basis $\vbf_1,\ldots, \vbf_{n-1}$ for $\ker l_{\xbf}$ such that $h(\vbf_i)\ll \max_i h(f_i)$, and rewriting all the functions in this coordinate systems. Now by a dimensional argument, in finitely many steps we get the desired $\xbf$.    
\end{proof}
\subsection{Reduction to the complete intersection, geometrically zero-dimensional case.} By Lemma~\ref{l:CuttingByHyperplane}, Proposition~\ref{p:SmallSolution} can be reduced to the (geometrically) zero-dimensional case. In this section, we make further reduction to the case of a complete intersection, geometrically zero-dimensional case. And then invoking the arithmetic B\'{e}zout theorem~\cite[Theorem 5.5.1]{BGiS} Propoductsition~\ref{p:SmallSolution} is proved. 

\begin{lem}\label{l:CompleteIntersection}
Let $V=\Spec(\qbar[\underline{X}]/\langle f_1,\ldots,f_m\rangle)$. Suppose that $\dim V=d$.
Then for $1\le l\le n-d$ there are $\tilde{f}_1,\ldots,\tilde{f}_l\in \qbar[\underline{X}]$ 
\begin{enumerate}
	\item For any $i$, $\tilde{f}_i\in \bbz f_1+\bbz f_2+\cdots+\bbz f_m$.
	\item For any $i$, $h(\tilde{f}_i)\ll \max_j h(f_j)$ where the implied constant depends on $n$ and the degree of $f_j$.
	\item $\dim (\qbar[\underline{X}]/\langle \tilde{f}_1,\ldots,\tilde{f}_l\rangle)=n-l.$
\end{enumerate}
\end{lem}
\begin{proof}
We proceed by induction on $l$. The base of the induction is clear. Assume that we have found $\tilde{f}_1,\ldots,\tilde{f}_l$ with the desired properties and $l<n-d$. Now by Krull's height theorem \cite[Theorem 13.5]{Mat} and the fact that the ring of polynomials is a catenary ring~\cite[Chapter 5, Theorem 17.7, Theorem 17.9]{Mat} we have that ${\rm ht}(\afr)=l$ where $\afr=\langle \tilde{f}_1,\ldots,\tilde{f}_l\rangle$. So by the unmixedness theorem~\cite[Theorem 17.6, Theorem 17.7]{Mat} we have ${\rm ht}(\pfr)=l$ for any $\pfr\in {\rm Ass}(\afr)$ (see \cite[Chapter 5]{Mat} for the definition of height ${\rm ht}(\afr)$ of $\afr$ and \cite[Chapter 6]{Mat} for the definition of the set ${\rm Ass}(\afr)$ of associated prime ideals of $\afr$). Since ${\rm ht}(\langle f_1,\ldots,f_m\rangle)=n-d>l$, for any $\pfr\in{\rm Ass}(\afr)$ there is an $i$ such that $f_i\not\in \pfr$. This implies that for any $m$ distinct integers $n_1,\ldots,n_m$ and any $\pfr\in{\rm Ass}(\afr)$ we have that
\[
\{\sum_{i=1}^m {n_j}^i f_i|\h j=1,\ldots, m\}\not\subseteq \pfr.
\] 
Therefore
\be\label{e:AvoidingZeroDivisors}
\{\sum_{i=1}^m j^i f_i|\h j=1,\ldots,l'\}\not\subseteq \bigcup_{\pfr\in {\rm Ass}(\afr)} \pfr,
\ee
where $l'=| {\rm Ass}(\afr)|(m-1)+1$.

By the generalized B\'{e}zout theorem and the induction hypothesis $|{\rm Ass}(\afr)|\ll 1$ where the implied constant depends on $n$ and the degree of $f_i$. It is also well-known that the set of zero-divisors of $\qbar[\underline{X}]/\afr$ is equal to $\bigcup_{\pfr\in {\rm Ass}(\afr)}\pfr/\afr$. So for some $j\ll 1$, $\tilde{f}_{l+1}=\sum_{i=1}^m j^i f_i$ modulo $\afr$ is not a zero-divisor. Hence by Krull's principal ideal theorem~\cite[Theorem 13.5]{Mat}, the fact that ring of polynomials is catenary and the above discussion, we have that $\tilde{f}_{l+1}$ has the desired properties.   
\end{proof}

\begin{proof}[Proof of Proposition~\ref{p:SmallSolution}.]
By Lemma~\ref{l:CuttingByHyperplane} and Lemma~\ref{l:CompleteIntersection}, it is enough to prove that if 
\[
V(\qbar)=\{\xbf\in \qbar^n|\h f_1(\xbf)=\cdots=f_n(\xbf)=0\}
\]
is a finite non-empty set where $V=\Spec(k[\underline{X}]/\langle f_1,\ldots,f_n\rangle)$, then for any $\xbf\in V(\qbar)$ we have that
$
h(\xbf)\ll \max_i h(f_i),
$
where the implied constant depends on $n$ and $\max_i \deg f_i$. And this is an easy corollary of arithmetic B\'{e}zout theorem~\cite[Theorem 5.5.1]{BGiS}.
\end{proof}

\end{document}